\documentclass[11pt, preprint]{imsart}
\usepackage{xr}
\usepackage{amsmath,amssymb}
\usepackage{parskip}
\usepackage{marvosym}
\usepackage[hang,small,bf]{caption}
\usepackage{mathtools,bm}
\usepackage[colorlinks]{hyperref}
\usepackage{hyperref}
\hypersetup{
	colorlinks,%
	citecolor=blue,%
	filecolor=black,%
	linkcolor=blue,%
	urlcolor=blue
}
\usepackage{enumerate}
\usepackage{multirow}
\RequirePackage[OT1]{fontenc}

\usepackage{amsthm}
\usepackage{amsmath}
\usepackage{amssymb}
\usepackage{thmtools} % required to make autoref work
\usepackage{subcaption}
\usepackage{textcomp}

\usepackage{graphicx}
\usepackage{verbatim}
\usepackage{array, float}
\usepackage{fontenc}
\usepackage[toc,page]{appendix}
\usepackage[english]{babel}
\usepackage{marvosym}

\usepackage[pagewise]{lineno}
\usepackage{mathtools}
\usepackage{scalerel}
\allowdisplaybreaks
\usepackage{bbm}
\usepackage[noabbrev,capitalize]{cleveref}

\newtheoremstyle{exampstyle}
{8pt} % Space above
{8pt} % Space below
{\it} % Body font
{} % Indent amount
{\bfseries} % Theorem head font
{.} % Punctuation after theorem head
{.5em} % Space after theorem head
{} % Theorem head spec (can be left empty, meaning `normal')

\theoremstyle{exampstyle}
\newtheorem{theorem}{Theorem}

\newtheorem{lemma}{Lemma}
\newtheorem{corollary}{Corollary}
\newtheorem{remark}{Remark}

\newtheorem{prop}{Proposition}

\newtheorem{defn}{Definition}

\numberwithin{equation}{section}
\numberwithin{example}{section}
\numberwithin{theorem}{section}
\numberwithin{lemma}{section}
\numberwithin{corollary}{section}
\numberwithin{prop}{section}
\numberwithin{defn}{section}
\numberwithin{remark}{section}

\startlocaldefs

\newcommand{\eat}[1]{}

\DeclarePairedDelimiter{\abs}{|}{|}

\usepackage{enumitem}
\renewcommand{\hat}[1]{\widehat{#1}}
\renewcommand{\tilde}[1]{\widetilde{#1}}
\newcommand{\E}{\mathbbm{E}}
\renewcommand{\P}{\mathbbm{P}}

\newcommand{\ind}{\mathbbm{1}}
\newcommand{\Var}{\op{Var}}

\newcommand{\Cov}{\op{Cov}}
\newcommand{\R}{\mathbb{R}}
\newcommand{\F}{F_Y^{(n)}}

\newcommand{\mW}{\mathcal{W}}

\newcommand{\bpj}{\beta_{\pha}(\jtd^{(n)})}
\newcommand{\bpjr}{\beta_{\pha}(\jtdt)}
\newcommand{\bpjs}{\beta_{\pha}(\jtds)}
\newcommand{\bpjd}{\beta_{\pha}(\jtdd)}

\newcommand{\stp}{*\prime}
\newcommand{\jt}{F_{X,Y}}
\newcommand{\xj}{\xi(\jtd^{(n)})}
\newcommand{\mx}{F_X}
\newcommand{\my}{F_Y}
\newcommand{\jtd}{f_{X,Y}}

\newcommand{\mxd}{f_X}
\newcommand{\myd}{f_Y}
\newcommand{\cde}{g_{X,Y}}
\newcommand{\Dl}{\Delta}
\newcommand{\hloc}{\mathrm{H}_{1,n}^{\mathrm{loc}}(c_n;\Gamma)}

\newtheorem{Assumption}{Assumption}

\@addtoreset{proofpart}{theorem}
\renewcommand\thesection{\arabic{section}}
\makeatletter
\newcommand*{\rom}[1]{\expandafter\@slowromancap\romannumeral #1@}
\makeatother

\newcommand{\Hy}{\mathrm{H}}
\newcommand{\mD}{\mathfrak{D}}

\newcommand{\bX}{\mathbf{X}}

\newcommand{\pha}{\phi_{n}}
\newcommand{\jtdt}{f_{X,Y,r_n}}
\newcommand{\jtds}{f_{X,Y,\sigma_n}}
\newcommand{\jtdd}{f_{X,Y,\Delta_n}}
\newcommand\numberthis{\addtocounter{equation}{1}\tag{\theequation}}

\newcommand{\jtdc}{f_{Y|X}}

\providecommand{\abs}[1]{\bigg\lvert#1\bigg\rvert}

\renewcommand{\hat}{\widehat}

\def\var{\mbox{var}}
\def\Var{\mbox{Var}\,}
\def\Cov{\mbox{Cov}}

\def\disp{\displaystyle}

\newcommand{\bfm}[1]{\ensuremath{\mathbf{#1}}}

     \def\EE{\mathbb{E}}

     \def\PP{\mathbb{P}}

\def\bx{\bfm x}   \def\bX{\bfm X}  
\def\by{\bfm y}   \def\bY{\bfm Y}  
     	
\def\E{\mathbb{E}}
\def\P{\mathbb{P}}

\newcommand{\ccom}{\mathpunct{\raisebox{0.2ex}{,}}}

\allowdisplaybreaks

\endlocaldefs

\begin{document}
	
	\renewcommand{\abstractname}{}    % clear the title

	\begin{frontmatter}
		\title{Exact detection thresholds and minimax optimality of Chatterjee's correlation coefficient}
		
		\runtitle{On Chatterjee's correlation coefficient}

		\begin{aug}
	\author{\fnms{Arnab} \snm{Auddy,}\thanksref{t1,a1}
		\ead[label=e1]{aa4238@columbia.edu}}
	\author{\fnms{Nabarun} \snm{Deb,}\thanksref{t1,a2}\ead[label=e2]{nd2560@columbia.edu}}\and
	\author{\fnms{Sagnik} 
		\snm{Nandy}\thanksref{t1,a3} \ead[label=e3]{sagnik@wharton.upenn.edu}}
	% \address{1255 Amsterdam Ave. \\ New York, NY 10027 \\ \printead{e3}}
	\address{
		\thanksmark{a1}
		Department of Biostatistics, University of Pennsylvania, \\
		%\printead{e1}\\
		\thanksmark{a2}
		Booth School of Business, University of Chicago,\\
		%\printead{e2}\\
		\thanksmark{a3}
		Department of Statistics, University of Pennsylvania\\
		%\printead{e3}\\
	}
	
	\runauthor{Auddy, Deb, and Nandy}
	\thankstext{t1}{Alphabetical order. All authors have equal contribution.}% Authorship has been assigned alphabetically.}
%\address{\thanksmark{m1} University of Pennsylvania\\
%	423 Guardian Drive\\
%	Philadelphia, PA 19104\\
%	\printead{e1} \\}
%\address{\thanksmark{m2} 
%	The University of Chicago Booth School of Business
%	\\
%	5807 S Woodlawn Ave\\
%	 Chicago, IL 60637\\
%	\printead{e2} \\}
%\address{\thanksmark{m3} University of Pennsylvania \\
%	419 Wharton Academic Research Building \\
%	Philadelphia, PA 19104 \\
%	\printead{e3} \\			
%}
\end{aug}
		\vspace{0.2in}
		\begin{abstract} Recently, Chatterjee (2021) introduced a new rank-based correlation coefficient which can be used to measure the strength of dependence between two random variables. This coefficient has already attracted much attention as it converges to the Dette-Siburg-Stoimenov measure (see Dette et al. (2013)), which equals $0$ if and only if the variables are independent and $1$ if and only if one variable is a function of the other. Further, Chatterjee's coefficient is computable in (near) linear time, which makes it appropriate for large-scale applications. In this paper, we expand the theoretical understanding of Chatterjee's coefficient in two directions: (a) First we consider the problem of testing for independence using Chatterjee's correlation. We obtain its asymptotic distribution under any changing sequence of alternatives converging to the null hypothesis (of independence). We further obtain a general result that gives exact detection thresholds and limiting power for Chatterjee's test of independence under natural nonparametric alternatives converging to the null. As applications of this general result, we prove a $n^{-1/4}$ detection boundary for this test and compute explicitly the limiting local power on the detection boundary for popularly studied alternatives in the literature. (b) We then construct a test for non-trivial levels of dependence using Chatterjee's coefficient. In contrast to testing for independence, we prove that, in this case, Chatterjee's coefficient indeed yields a minimax optimal procedure with a $n^{-1/2}$ detection boundary. Our proof techniques rely on Stein's method of exchangeable pairs, a non-asymptotic projection result, and information theoretic lower bounds.
			
		\end{abstract}

		
		
	\end{frontmatter}

	% AOS,AOAS: If there are supplements please fill:
	%\begin{supplement}[id=suppA]
	%  \sname{Supplement A}
	%  \stitle{Title}
	%  \slink[doi]{10.1214/00-AOASXXXXSUPP}
	%  \sdatatype{.pdf}" 
	%  \sdescription{Some text}
	%\end{supplement}

	\section{Introduction}
	
	Suppose $(X_1,Y_1), \ldots , (X_n,Y_n)\overset{i.i.d.}{\sim}\jt(\cdot,\cdot)$ for some bivariate distribution function $\jt(\cdot,\cdot)$, with marginals $\mx(\cdot)$ and $\my(\cdot)$ for $X_1$ and $Y_1$, respectively. The problem of measuring and testing the \emph{extent of dependence} between $X_1$ and $Y_1$ has attracted much attention for over a century (see e.g., \cite{Blum1961,dette2013copula,kendall1938new,pearson1895notes,spearman1906footrule}). A fundamental question in this regard is the classical independence testing problem
	\begin{equation}\label{eq:baseprob}
		\Hy_0:\jt(x,y)=\mx(x)\my(y)\ \forall \ x,y\in\R \quad \mbox{versus} \quad \Hy_1: \textrm{not } \Hy_0. %\exists\ x,y\in \R\ \textrm{s.t.}\ \jt(x,y)\neq\mx(x)\my(y).
	\end{equation}
	Problem~\eqref{eq:baseprob} has been studied extensively in the statistics literature along with a variety of applications (see~\cite{bergsma2014consistent, Blomqvist1950,Blum1961,even2020independence,even2021counting,gamboa2018sensitivity,gini1914ammontare,heller2016computing,Hoeffding1948,kendall1938new,nandy2016large,spearman1906footrule,weihs2016efficient,csorgHo1985testing,Deb2023,Drton2020,Han2017,Rosenblatt1975,Wang2017,Weihs2018,yanagimoto1970measures,Hajek1999}). Note however that~\eqref{eq:baseprob} tests only whether $X$ and $Y$ are fully independent or otherwise, and does not give any indication as to how strong the dependence between $X$ and $Y$ is. Therefore to better understand the dependence, an alternative approach is to use a \emph{measure of dependence}, say $\rho(\jt)$, which can be chosen as classical measures such as Pearson's linear correlation coefficient~\cite{pearson1895notes}, Spearman's rank correlation coefficient~\cite{spearman1961proof} or more nonparametric measures such as~\cite{dette2013copula,gamboa2018sensitivity} to name a few, and consider the testing problem
	\begin{equation}\label{eq:genvalprob}
		\Hy_0:\rho(\jt)=\rho_0\  \quad \mbox{versus} \quad \Hy_1: \rho(\jt)>\rho_0, %\exists\ x,y\in \R\ \textrm{s.t.}\ \jt(x,y)\neq\mx(x)\my(y).
	\end{equation}
	for general $\rho_0$. Depending on the choice of $\rho(\jt)$, problem~\eqref{eq:genvalprob} gives a more interpretable understanding of the dependence between $X_1$ and $Y_1$. For example, if $\rho(\jt)$ is chosen to be Pearson's correlation, then~\eqref{eq:genvalprob} helps understand how well $Y_1$ can be predicted from $X_1$ using a linear function; and if $\rho(\jt)$ is chosen to be Spearman's correlation, then it helps understand how well $Y_1$ can be predicted from $X_1$ using a general monotone transform.
	
	In view of problems~\eqref{eq:baseprob}~and~\eqref{eq:genvalprob}, recently, Chatterjee~\cite{sc_corr} introduced a new nonparametric data rank-based correlation coefficient $\xi_n$ (see~\eqref{eq:T_n} below for definition). This coefficient in~\cite{sc_corr} possesses a combination of natural, but unique characteristics not exhibited by other measures. In particular, it converges to $0$ if and only if $X_1$ and $Y_1$ are independent and to $1$ if and only if $Y_1$ is a measurable function of $X_1$, as long as $X_1$ and $Y_1$ are non-degenerate. In fact, Chatterjee's coefficient converges to the Dette-Siburg-Stoimenov measure (see~\cite{dette2013copula}; also see~\eqref{eq:popdef}) for general bivariate distributions. It also produces a \emph{consistent test} against all \emph{fixed} alternatives for problems~\eqref{eq:baseprob}~and~\eqref{eq:genvalprob}, and is computable in (near) linear time (in terms of sample size), making it suitable for large-scale applications. Furthermore, through extensive simulations in~\cite{sc_corr}, the author argued that $\xi_n$ converges to a population measure $\xi(\jt)$ (introduced first in~\cite{dette2013copula}; see~\eqref{eq:popdef} below) that captures how well $Y_1$ can be predicted using general measurable functions of $X_1$ (see~\cref{sec:degtest} for more details). This gives the testing problem~\eqref{eq:genvalprob} using $\rho(\jt)=\xi(\jt)$, a natural and completely nonparametric interpretation. Consequently, it has attracted much attention in the past two years, in terms of both applications and theory (see~\cite{azadkia2019simple,cao2020correlations,chatterjee2020insights,deb2020kernel,fruciano2020does,huang2020kernel,shi2020power,strothmann2022rearranged,azadkia2021fast,han2022azadkia,lin2023failure,lin2021boosting,lin2022limit,chatterjee2022estimating,fuchs2021quantifying,Zhang2023,shi2021azadkia,bickel2022measures,ansari2022simple,holma2022correlation,Griessenberger2022}). The goal of this article is to expand the theoretical understanding of $\xi_n$ for the widely popular testing problems in~\eqref{eq:baseprob}~and~\eqref{eq:genvalprob} under general local alternatives (that is, when \emph{alternative} converges to \emph{null} as the sample size $n$ increases) and obtain exact detection thresholds. Before discussing our main results, let us first present Chatterjee's test statistic in~\cite{sc_corr}.
	
	%obtain the \emph{exact detection boundary} for this test under \emph{changing} sequence of both \emph{nonparametric} and \emph{parametric} alternatives converging to the null. In the process, we give precise characterizations of its \emph{local power}.   
	
	We arrange the data as $(X_{(1)},Y_{(1)}), \ldots , (X_{(n)},Y_{(n)})$ so that $X_{(1)}\leq \ldots \leq X_{(n)}$ and $Y_{(i)}$ is the $Y$ value \emph{concomitant} to $X_{(i)}$. Let $R_{n,i}$ be the rank of $Y_i$, i.e.,
	$$R_{n,i}:=\sum_{j=1}^n \ind(Y_j\leq Y_i).$$
	Now consider the statistic given by
	\begin{equation}
		\label{eq:T_n}
		\xi_n := 1-\frac{3}{n^2-1}\sum\limits_{i=1}^{n-1}\left|R_{n,(i+1)}-R_{n,(i)}\right|.
	\end{equation}
	Here $R_{n,(i)}$ is the rank of $Y_{(i)}$. Note that this statistic is well defined if there are no ties among $X_i$'s. For a more general definition of $\xi_n$ that allows for ties, see~\cite[Page 2]{sc_corr}. For technical convenience, we will work with the above definition of $\xi_n$ instead of the more general definition that takes potential ties into account. In particular, we will assume the existence of marginal densities of $X$'s and $Y$'s for the rest of the paper.
	
	It has been established in Theorem 1.1 of \cite{sc_corr} that $\xi_n$, as defined in~\eqref{eq:T_n}, converges almost surely to the \emph{Dette-Siburg-Stoimenov} measure (see~\cite{dette2013copula})
	\begin{equation}\label{eq:popdef}
		\xi(\jtd):=6\int \E\left[\P(Y\leq t|X)-\P(Y\leq t)\right]^2\myd(t)\,dt,
	\end{equation}
	where $\jtd(\cdot,\cdot)$ denotes the density corresponding to the distribution function $\jt(\cdot,\cdot)$, with marginal densities $\mxd(\cdot)$ and $\myd(\cdot)$. In various bivariate copula based models, $\xi(\jtd)$ turns out to be a monotonic function of the natural dependence parameter (see examples 1.(a) - (d) in~\cite[Page 9 and Theorem 2]{dette2013copula}; also see~\cref{sec:degtest}), thereby making it natural and interpretable as a way to measure the strength of dependence between $X_1$ and $Y_1$.
	
	\subsection{Problem setup and summary of main contributions}\label{sec:mainres}
	
	We will use the standard framework of local power analysis taken from~\cite{ingster1987minimax,ingster1993asymptotically} which is popularly used in independence testing procedures, see e.g.,~\cite{berrett2021optimal,kim2022minimax}. Towards this direction, consider a \emph{triangular array} of i.i.d. random variables $(X_{n,1}\ccom Y_{n,1}),\ldots ,(X_{n,n}\ccom Y_{n,n})$ from a bivariate density $\jtd^{(n)}(\cdot,\cdot)$ with marginals $f_X^{(n)}(\cdot)$, $f_Y^{(n)}(\cdot)$. Note that the joint distribution is no longer fixed, but changes with the sample size. We analyze the testing problem
	\begin{equation}\label{eq:revprob}
		\Hy_{0,n}:\xj=\xi_0 \quad \mbox{versus} \quad \Hy_{1,n}: \xj=\xi_0+c_n, \quad \xi_0\in [0,1)
	\end{equation}
	where $\xi(\cdot)$ is as defined in~\eqref{eq:popdef} and
	\begin{equation}\label{eq:nonreg}
		c_n\downarrow 0 \quad \implies \xj\downarrow \xi_0 \quad\mbox{as}\ n\to\infty.
	\end{equation} 
	Clearly distinguishing between $\Hy_{0,n}$ and $\Hy_{1,n}$ becomes harder as $c_n$ converges to $0$ faster. Also note that by~\cite[Theorem 1.1]{sc_corr} and~\cite[Theorem 2]{dette2013copula}, testing $\xi_0=0$ exactly corresponds to the test of independence as in~\eqref{eq:baseprob}. As $\xi_n-\xj$ converges almost surely to $0$ under mild assumptions (see~\cite[Theorem 1.1]{sc_corr}), a natural level $\alpha$ test function for~\eqref{eq:revprob} is given by
	\begin{equation}\label{eq:asymtest} \pha:=\ind(\xi_n-\xi_0\geq z_{n,\alpha}),
	\end{equation}
	where $z_{n,\alpha}$ is chosen appropriately so as to satisfy the level $\alpha$ condition. Define the power function of $\pha$ as
	\begin{equation}\label{eq:powdef}\bpj:=\P(\pha=1).\end{equation}
	%From~\cite[Theorems 1.1 and 2.1]{sc_corr} it follows that $\pha$ is consistent  against fixed alternatives, i.e., $$\P(\pha=1)\to 1$$
	%if $\xi(\jtd)>0$ where $\jtd(\cdot,\cdot)$ is fixed and does not change with $n$. 
	Our goal is to investigate the following: 
	\emph{``What is the fastest decaying $c_n\downarrow 0$ such that $\pha$ can distinguish between the null and the alternative, i.e., $\bpj\to 1$ as $n\to\infty$ under $\Hy_{1,n}$? Further, is $\pha$ optimal for the  testing problem~\eqref{eq:revprob}?''}
	
	In this paper, we provide an exact answer to this question. We find a dichotomy based on whether or not the limiting Dette-Siburg-Stoimenov measure $\xi_0$ is zero or positive. We make a two-fold contribution in both of these regimes. Note that the case where $\xi_0=0$ in \eqref{eq:revprob} corresponds to the important problem of \emph{independence} testing. On the other hand, when $\xi_0>0$, we refer to the null hypothesis in \eqref{eq:revprob} as the problem of testing for degree of association between $X$ and $Y$.
	
	\emph{Critical detection boundary of $\phi_n$ for independence testing $(\xi_0=0)$:} For this case, in~\cref{th:power}, we show that the power of $\phi_n$ converges to $\alpha$, or $1$, or a number in $(\alpha,1)$ depending on whether $\sqrt{n}c_n$ converges to $0$, or $\infty$, or some number in $(0,\infty)$, respectively. This indicates that the best choice of $c_n$ that ensure $\bpj\to 1$ is $c_n=O(n^{-1/2})$. 
	
	This is however a suboptimal threshold in terms of detecting dependence. To see this, consider the case where $(X_1,Y_1)$ is a bivariate normal distribution with correlation $\rho_n$. Theorem ~\ref{th:power} implies that $\bpj$ converges to $\alpha$, or $1$, or a number in $(\alpha,1)$ depending on whether $n^{1/4}\rho_n$ converges to $0$, or $\infty$, or some number in $(0,\infty)$, respectively. This implies that $\pha$ has a detection boundary of $O(n^{-1/4})$ in terms of $\rho_n$. However, it is well known from Le Cam's theory of local asymptotic normality that the optimal detection threshold for $\rho_n$ is of the order $n^{-1/2}$ and not $n^{-1/4}$.  In Section ~\ref{sec:applications}, we give concrete examples of this detection boundary in some other local parametric alternatives, viz. simple mixtures, and noisy nonparametric regression.

	\emph{Minimax optimality of $\pha$ for testing degree of association $(\xi_0>0)$:} For this case, in~\cref{thm:test_assoc}, part 1, we show that $\bpj$ converges to $1$ provided $\sqrt{n}c_n\to\infty$. In contrast to the $\xi_0=0$ case, we prove in~\cref{thm:test_assoc}, part 2, that  $c_n=O(n^{-1/2})$ is indeed the optimal threshold when $\xi_0>0$ in a local asymptotic minimax sense (see~\eqref{eq:poweralt}). In other words, if $\sqrt{n}c_n\to (0,\infty)$, then no test can uniformly have power converging to $1$. Therefore $c_n=O(n^{-1/2})$ is the correct detection boundary and this highlights the minimax optimality of Chatterjee's correlation coefficient for the testing problem~\eqref{eq:revprob} when $\xi_0>0$; see Section~\ref{sec:degtest} for more details.
	
	Additionally, as a technical device for the results above, we develop a central limit theorem for $\xi_n$.
	
	\emph{Central limit theorem for shrinking alternatives:} In Theorem~\ref{th:xi_n_wass}, we show that for any sequence of alternatives specified by $\xj\to 0$, $\xi_n$ is asymptotically normal. Further, we characterize the limiting mean and the limiting variance explicitly. This is a wide generalization of the asymptotic normality results in \cite{cao2020correlations,sc_corr,shi2020power} which are stated under independence, or the fast shrinking alternatives $\xj=O(n^{-1/2})$. Theorem~\ref{th:xi_n_wass} weakens this assumption only to $\xj\to 0$ and might be of independent interest. This CLT is obtained using Stein's method-based technique (see~\cite{chatt_nn}). After the first draft of our paper, a number of other interesting limit theorems for $\xi_n$ or modified versions of $\xi_n$ have been established that highlight the interest in Chatterjee's correlation; see e.g.~\cite{han2022azadkia,lin2021boosting,lin2022limit,lin2023failure,ansari2022simple, Zhang2023} to name a few. In the following section, we will summarize the other relevant contributions to the problem considered here.
	
	\subsection{Comparison with existing literature}\label{sec:comparelit}
	Prior to the first version of our paper, the theoretical analysis of $\xi_n$ had been carried out in three papers. In Chatterjee's paper~\cite{sc_corr}, the asymptotic distribution of $\xi_n$ was derived under $\mathrm{H}_0$ as in~\eqref{eq:baseprob} and its consistency against fixed (not changing with $n$) alternatives was proved. Two other papers~\cite{shi2020power}~and~\cite{cao2020correlations} have analyzed $\xi_n$ under smooth contiguous alternatives (see~\cite[Chapter 6]{vaart}). For example, under the mixture type alternatives in Section \ref{sec:appmix}, their results show that $\pha$ is powerless along ``contiguous" alternatives, i.e., in cases where the mixing probability shrinks to zero at a $n^{-1/2}$ rate. The proofs in~\cite{cao2020correlations,shi2020power} use Le Cam's third lemma (see~\cite[Example 6.7]{vaart}) which requires analyzing $\xi_n$ and the likelihood ratio jointly but \emph{only under the null}, (that is, when $X_{n,1}$ and $Y_{n,1}$ are independent) followed by a change of measure formula that only holds under contiguous scales and not beyond. In contrast, the focus of our paper is characterizing the exact detection boundary of $\xi_n$, which as we shall see, happens to be in the non-contiguous regime. We therefore adopt a proof strategy using Stein's method-based technique of local dependence (see~\cite{chatt_nn}) and non-asymptotic projection results. While the focus of the paper is in the bivariate setting, it should be noted that multivariate versions of $\xi_n$ have been studied in the literature (see \cite{azadkia2019simple,deb2020kernel,ansari2022simple}) and asymptotic distributions under independence have been obtained in \cite{deb2020kernel,shi2021azadkia,ansari2022simple}.
	
	After the first draft of our paper, a number of other results of interest have further solidified the understanding of $\xi_n$ or modified versions thereof. In \cite{lin2021boosting}, the authors modified $\xi_n$ by incorporating more ``right nearest neighbors" in its definition. They then proved that in the bivariate Gaussian independence testing problem (see~\cref{sec:applications}), as the number of right nearest neighbors grow, the detection boundary moves from $O(n^{-1/4})$ to nearly $O(n^{-1/2})$. On the other hand, in \cite{lin2022limit}, the authors show that $\xi_n$ is asymptotically normal even when $\xj=\xi_0>0$. The limiting variance, in that case, is no longer universal and depends on the data distribution. The authors in \cite{lin2022limit} obtain a consistent, analytic estimator of this limiting variance. In the follow-up paper~\cite{lin2023failure}, the authors show that a natural bootstrap estimator of this limiting variance is not consistent under independence. On the other hand, \cite{Zhang2023} proved the asymptotic normality of a symmetrized version of $\xi_n$. We also refer the reader to the recent review paper \cite{chatterjee2022survey} which provides a comprehensive overview of dependence/association measures that are based on  $\xi_n$.
	
	\subsection{Organization}\label{sec:org}
	The rest of the paper is organized as follows. In \cref{sec:main_results} we describe our main results when $\xi_0=0$. In particular,~\cref{th:xi_n_wass}~and~\cref{th:power} yield asymptotic limits of $\xi_n$ (centered and scaled) and asymptotic expressions for $\bpj$ depending on how fast $c_n\downarrow 0$. Applications of these results to test for independence in popular local parametric models are provided in \cref{sec:applications}. In particular,~\cref{cor:xi_mix}~and~\cref{cor:xi_reg}  highlight the $n^{-1/4}$ detection boundary. In \cref{sec:degtest} (see~\cref{thm:test_assoc}), we show that the test $\phi_n$ constructed in \eqref{eq:asymtest} is minimax optimal for the testing problem~\eqref{eq:revprob} when $\xi_0>0$. %Numerical studies can be found in \cref{sec:simuls}. 
	The proofs of all main results are presented in the Appendix A. Finally, Appendix B contains the proofs of some additional results and technical lemmas.
	
	\section{Critical detection boundary of $\phi_n$ for independence testing}\label{sec:main_results}
	In this section, we first show that under a wide class of bivariate distributions satisfying~\eqref{eq:nonreg}, $\xi_n$, appropriately centered by its mean and scaled by its standard deviation,  converges to a standard normal distribution. We provide a precise characterization of the limiting bias and standard deviation of $\xi_n$. We then use these findings to obtain the detection threshold and asymptotic power of the test $\pha$ as described in~\eqref{eq:asymtest}.
	%This section is broadly divided into two parts: 
	%\begin{itemize}
	%	\item[(a)] In~\cref{sec:asnorm}, we show that under a wide class of bivariate distributions satisfying~\eqref{eq:nonreg}, $\xi_n$, appropriately centered by its mean and scaled by its standard deviation,  converges to a standard normal distribution. We then provide a precise characterization of the limiting bias and standard deviation of $\xi_n$. Moreover all our results yield come with explicit sample guarantees. 
	%	\item[(b)] In~\cref{sec:aspow}, we use the findings from~\cref{sec:asnorm} to obtain the detection threshold and asymptotic power of the test $\pha$ as described in~\eqref{eq:asymtest}.
	%\end{itemize}
	
	%\subsection{Asymptotic Normality of $\xi_n$ under~\eqref{eq:nonreg}}\label{sec:asnorm}

	%The goal of this section is to derive the asymptotic normality of $\xi_n$ after suitable centering and scaling. 
	
	Recall the setting from the Introduction. We consider $(X_{n,k},Y_{n,k})_{1\leq k\leq n}$, $n\in\mathbb{N}$, a triangular array of i.i.d. random variables drawn from a bivariate density $\jtd^{(n)}(\cdot,\cdot)$. All the probabilities and expectations taken in the sequel are with respect to the measure induced by $\jtd^{(n)}(\cdot,\cdot)$.%\par 
	
	%To motivate the asymptotic normality of $\xi_n$, let us begin with a simple observation. 
	Notice that from~\eqref{eq:T_n}, with the identity  $\underset{i\neq j}{\sum\sum}|i-j|=n(n^2-1)/3$, one has
	$$
	\xi_n=\dfrac{3}{n^2-1}\left(\dfrac{1}{n}\underset{i\neq j}{\sum\sum }\abs*{i-j}-\sum_{i=1}^n\abs*{R_{n,(i+1)}-R_{n,(i)}}\right).
	$$
	
	%By the standard Glivenko-Cantelli Theorem, we know that $\hat{F}_n(\cdot)$ and $F^{(n)}_Y(\cdot)$ are ``close" almost surely in the $L^{\infty}$ norm. This motivates the definition of an oracle version of $\xi_n$ as follows:
	%\begin{equation}\label{eq:proj_defn} 
	%	\xi_n^*:=\dfrac{3n}{n^2-1}\left(\dfrac{1}{n}\underset{i\neq j}{\sum\sum }\abs*{\F(Y_{n,i})-\F(Y_{n,j})}-\sum_{i=1}^n\abs*{F^{(n)}_Y(Y_{n,(i+1)})-F^{(n)}_Y(Y_{n,(i)})}\right).
	%\end{equation}
	%Intuitively of course, $\xi_n^*$ is mathematically more tractable than $\xi_n$ as it replaces the random function $\hat{F}_n(\cdot)$ by the deterministic function $F^{(n)}_Y(\cdot)$. 
	We will need the following definitions.
	
	\begin{defn}[Kantorovic-Wasserstein distance]\label{def:wass} Given any two probability measure $\mu$ and $\nu$ on the real line, the Kantorovic-Wasserstein distance between $\mu$ and $\nu$ is defined as
		$$
		\mW(\mu,\nu):=\sup\left\{\bigg|\scaleobj{0.7}{\int} h\,d\mu-\scaleobj{0.7}{\int} h\,d\nu\bigg|:\ h(\cdot)\ \mathrm{ is\ Lipschitz },\ \lVert h\rVert_{\mathrm{Lip}}\leq 1\right\}.$$
	\end{defn}
	In our applications, we will fix $\nu$ as the \emph{standard Gaussian distribution}, which leads to the following natural notion of ``distance to Gaussianity" based on~\cref{def:wass}.
	\begin{defn}[Distance to Gaussianity]\label{def:dtog}
		Let $\nu$ be the standard Gaussian law and $T$ be a random variable with the law $\mu$. Then the distance between $T$ and the standard Gaussian is defined as
		$$\mD(T):=\mW(\mu,\nu),$$
		where $\nu$ is the standard normal law.
	\end{defn}

	We first state our assumptions.
	
	\begin{Assumption}\label{as:Lipschitz} There exist functions $L_i^{(n)}(\cdot):\R\to [0,\infty)$ for $i=1,2$ and numerical constants $\kappa_1>0$, $\eta\in (0,1]$, and $\theta\in (1,\infty]$ such that $\forall\, y,x_1,x_2,$
		\begin{equation}\label{eq:Lipschitz1}
			\,\,\abs*{\P(Y_{n,1}\geq y|X_{n,1}=x_1)-\P(Y_{n,1}\geq y|X_{n,1}=x_2)}\leq (1+L_1^{(n)}(x_1,y)+L_2^{(n)}(x_2,y))\abs*{x_1-x_2}^{\eta},
		\end{equation}
		\begin{equation}\label{eq:Lipschitz2}
			\limsup\limits_{n\to\infty} \int (L_i^{(n)}(x,y))^{\theta}\mxd^{(n)}(x)\myd^{(n)}(y)\,dx\, dy \leq \kappa_1,
		\end{equation}
		where $\mxd^{(n)}(\cdot)$ and $\myd^{(n)}(\cdot)$ are the marginal densities of $X_{n,1}$ and $Y_{n,1}$ under the joint density $\jtd^{(n)}(\cdot,\cdot)$.	
	\end{Assumption}
	%Note that~\cref{as:Lipschitz} is weaker than the standard $\eta$ H\"{o}lder condition, in that the H\"{o}lder constants are allowed to depend on $y,x_1,x_2$ and also $n$.
	\begin{Assumption}\label{as:nn_dists} There exist numerical constants $\gamma>1$ and $\kappa_2>0$ such that
		$$\limsup\limits_{n\to\infty} \E|X_{n,1}|^{\gamma}\leq \kappa_2.$$
	\end{Assumption}

	\noindent Note that~\cref{as:Lipschitz} is weaker than the standard $\eta$ H\"{o}lder condition, in that the H\"{o}lder constants are allowed to depend on $y,x_1,x_2$ and also $n$. In this sense, it is weaker than the assumptions in \cite{azadkia2019simple} and related papers.%~\cref{as:Lipschitz} ensures that $X_{N(1)}$ and $X_1$ are ``sufficiently" close.
	~\cref{as:nn_dists} is a standard moment assumption to control the tail of the distribution of $X_{n,i}$'s. This tail behavior crucially affects the distance between $X_{n,(i)}$ and $X_{n,(i+1)}$, see e.g.,~\cite[Section 2.2]{Biau2015}.
	
	We are now in position to state our main result. 
	In the following theorem (see Section~\ref{sec:proofs-sec2} for a proof), we show that $\xi_n$, appropriately centered and scaled, has a limiting normal distribution in the asymptotic regime $\xi(\jtd^{(n)})\to 0$.
	\begin{theorem}\label{th:xi_n_wass} 
		For any bivariate density $f^{(n)}_{X,Y}(x,y)$ satisfying Assumptions \ref{as:Lipschitz} and \ref{as:nn_dists}, there is a numerical constant $C(\eta, \theta, \gamma, \kappa_1,\kappa_2)>0$, i.e., depending only on the parameters $\eta, \theta, \gamma, \kappa_1,\kappa_2$ from Assumptions \ref{as:Lipschitz} and \ref{as:nn_dists}, such that: 
		\begin{itemize}
			\item[(i)] For all $n\ge 1,$  we have
			$$\sqrt{n}\abs*{\E(\xi_n)-\xi(f^{(n)}_{X,Y})}\le Cn^{-1/2}+ C\sqrt{n}b_n\ind(\xj>0).$$
			\item[(ii)] Moreover,
			$$
			\mD\left(\dfrac{\sqrt{n}\left(\xi_n-\xj\right)}{\sqrt{2/5}}\right) \le Cn^{-1/2}+C\left(\xj+\sqrt{n\log n}\,b_n\right)^{1/2}\ind(\xj>0),
			$$
			where $\mD$ is the Wasserstein distance to normality defined in  \cref{def:dtog}, and $b_n$ is defined as
			
			\begin{equation}\label{eq:defbn}
				b_n:=n^{-\frac{\gamma}{\gamma+1}}(\log{n})^2+\left(\dfrac{(\log{n})^2}{n}\right)^{\left(\frac{\gamma(\theta-1)}{\theta(\gamma+1)}\wedge \frac{\eta \gamma}{\gamma+1}\right)}.
			\end{equation}
		\end{itemize}
	\end{theorem}
	
	In particular, part (ii) of \cref{th:xi_n_wass} shows that if $\xi(\jtd^{(n)})\to 0$ and $\frac{\theta-1}{\theta}\wedge \eta > \frac{\gamma+1}{2\gamma}$, then
	\begin{equation}\label{eq:defbn3}
		\sqrt{n}(\xi_n-\xj)\overset{w}{\longrightarrow}\mathcal{N}(0,2/5).
	\end{equation}
	Therefore, in the entire asymptotic regime $\xi(\jtd^{(n)})\to 0$, we see that $\xi_n$ has the same limiting variance, which matches the case where the $X$'s and $Y$'s are mutually independent (also see \cite[Theorem 2.1]{sc_corr}). A couple of remarks on the assumptions needed in \cref{th:xi_n_wass} are now in order.
	
	\begin{remark}
		To understand the condition $\frac{\theta-1}{\theta}\wedge \eta>\frac{\gamma+1}{2\gamma}$, let us first focus on a simple case. Assume that the conditional probability in \eqref{eq:Lipschitz1} is uniformly Lipschitz in $y,x$. In that case, $\eta=1$ and $L_1^{(n)}(\cdot)$, $L_2^{(n)}(\cdot)$ are both uniformly bounded. In view of \eqref{eq:Lipschitz2}, this implies $\theta=\infty$. Therefore $\frac{\theta-1}{\theta}\wedge \eta=1$. Recall from \cref{as:nn_dists} that $\gamma>1$ denotes the number of finite moments of $X_{n,1}$. Therefore, $\frac{\gamma+1}{2\gamma}<1$. As a result, the condition $\frac{\theta-1}{\theta}\wedge \eta>\frac{\gamma+1}{2\gamma}$ holds in this case.
		
		\noindent More generally speaking, the condition $\frac{\theta-1}{\theta}\wedge \eta>\frac{\gamma+1}{2\gamma}$ can be rewritten as the combination of the following conditions:
		
		$$\frac{1}{\theta}+\frac{1}{2\gamma}<1, \quad \frac{1}{\gamma} < 2\eta-1.$$
		
		Therefore, if the H\"{o}lder exponent $\eta$ in \eqref{eq:Lipschitz1} is greater than $\frac{1}{2}$, then the condition $\frac{\theta-1}{\theta}\wedge \eta>\frac{\gamma+1}{2\gamma}$ holds whenever $L_1^{(n)}(X_{n,1},Y_{n,2})$, $L_2^{(n)}(X_{n,1},Y_{n,2})$, and $X_{n,1}$ have sufficiently light tails.
	\end{remark}
	
	\begin{remark}
		In this paper, we have restricted ourselves to the case where the H\"{o}lder exponent $\eta\le 1$ instead of expanding to higher-order H\"{o}lder regularity. This is because parts (i) and (ii) of \cref{th:xi_n_wass} show that the order of the bias $\E(\xi_n)-\xj$ is $o(1/\sqrt{n})$ which is already of a smaller order than its fluctuation $(\xi_n-\E \xi_n)=O_p(1/\sqrt{n})$. As a result, imposing stronger regularity leads to no further gains here. The situation would be different for the multivariate version(s) of Chatterjee's correlation (see \cite{azadkia2019simple,deb2020kernel}) where the bias would reflect a curse of dimensionality and be of a higher order than the fluctuations. While this is an interesting question, it is currently beyond the scope of this paper.
	\end{remark}
	%We conclude this section by summarizing its contribution.
	Note that~\cref{th:xi_n_wass} greatly generalizes the asymptotic normality theorems of \cite{cao2020correlations,sc_corr,shi2020power} that are only valid under the null hypothesis of independence, or for contiguous parametric alternatives and cannot be used to analyze $\xi_n$ along non-contiguous and nonparametric alternatives. Therefore we believe~\cref{th:xi_n_wass} is of independent interest and hence we have presented it here as a separate result. Since Theorem \ref{th:xi_n_wass} aims to provide asymptotic normality for any alternative with $\xj\to 0$, we can no longer use traditional instruments such as Le Cam's third lemma, which is the main tool in~\cite{cao2020correlations,shi2020power}. Note that~\cref{th:xi_n_wass} holds for a large nonparametric class of distributions and comes with finite sample guarantees. In order to prove~\cref{th:xi_n_wass}, we use Stein's method of normal approximation for locally dependent structures~\cite{chatt_nn} and some explicit bias and variance computations. To elaborate briefly, we observe that $\xi_n$ can be rewritten as 
	\begin{equation}\label{eq:to_proj}
		\begin{split}
			\xi_n=\dfrac{3n}{n^2-1}\left(\dfrac{1}{n}\underset{i\neq j}{\sum\sum }\abs*{\hat{F}_n(Y_{n,i})-\hat{F}_n(Y_{n,j})}-\sum_{i=1}^n\abs*{\hat{F}_n(Y_{n,(i+1)})-\hat{F}_n(Y_{n,(i)})}\right),
		\end{split}
	\end{equation}
	
	where $\hat{F}_n(\cdot )$ is the empirical cumulative distribution function (CDF) of $Y_{n,1},\ldots ,Y_{n,n}$. Let $F^{(n)}_Y(\cdot)$ denote the population CDF of $Y_{n,1}$. Recall that $X_{n,(1)}\leq \ldots \leq X_{n,(n)}$ and $Y_{n,(i)}$ is the $Y$ value concomitant to $X_{n,(i)}$. The main idea is to show that we can replace $\hat{F}_n(\cdot )$ by $F^{(n)}_Y(\cdot)$ asymptotically. In other words, we show that $\xi_n$ is close (with quantitative error bounds) to $\xi_n^*$, where
	\begin{equation}
		\label{eq:oracle_stat}
		\xi_n^*:=\dfrac{3n}{n^2-1}\left(\dfrac{1}{n}\underset{i\neq j}{\sum\sum }\abs*{\F(Y_{n,i})-\F(Y_{n,j})}-\sum_{i=1}^n\abs*{F^{(n)}_Y(Y_{n,(i+1)})-F^{(n)}_Y(Y_{n,(i)})}\right).
	\end{equation}
	Next, we quantify the proximity of the distribution of $\xi_n^*$ (appropriately centered and scaled) to a standard normal distribution, using~\cite[Theorem 3.4]{chatt_nn}, to establish~\cref{th:xi_n_wass}.
	%\subsection{Asymptotic Power of Independence Test}\label{sec:aspow}
	
	%Let us try to understand this limiting behavior using~\eqref{eq:defbn3}. 
	
	Now, we characterize the asymptotic behavior of $\bpj$ defined in~\eqref{eq:powdef}.  First, suppose that $\sqrt{n}\xj\to\infty$. As $\sqrt{n}\left(\xi_n-\xj\right)=O_p(1)$, by~\eqref{eq:defbn3}, it is clear that $$\sqrt{n}\xi_n=\underbrace{\sqrt{n}\left(\xi_n-\xj\right)}_{O_p(1)}+\sqrt{n}\xj\overset{P}{\longrightarrow}\infty.$$ Therefore, the last equality in~\eqref{eq:powdef} coupled with the above display implies that whenever $\sqrt{n}\xj\to\infty$, we have $\bpj\to 1$. By using a similar sequence of arguments, we get the complete picture of the limiting behavior of $\bpj$, formalized in the theorem below.
	
	\begin{theorem}\label{th:power}
		Suppose $\xj\to 0$ and Assumptions \ref{as:Lipschitz}, \ref{as:nn_dists} are satisfied with  $\eta,\,\theta,\,\gamma$ such that $\frac{\theta-1}{\theta}\wedge \eta > \frac{\gamma+1}{2\gamma}$. Let $z_{\alpha}$ be the upper $\alpha$ quantile of the standard normal distribution. Then the test $\phi_n$, defined in~\eqref{eq:asymtest}, with $z_{n,\alpha}=n^{-1/2}z_{\alpha}\sqrt{2/5}$ has a power function satisfying
		$$\bpj=\P(\sqrt{n}\xi_n\ge z_{\alpha}\sqrt{2/5})
		=1-\Phi\left(z_{\alpha}-\sqrt{n}\xi(f^{(n)}_{X,Y})/\sqrt{2/5}\right)+o(1).$$	
		On the other hand, when $\xj\to c>0$, $\phi_n$ has asymptotic power equal to one. In particular, we have the following explicit characterization of the asymptotic power of $\phi_n$:
		\begin{equation}\label{eq:asymp_power}
			\lim\limits_{n\to \infty}\bpj=\begin{cases}\alpha &\text{ if }\sqrt{n}\xi(f^{(n)}_{X,Y})\to 0
				\\1-\Phi\left(z_{\alpha}-c_0/\sqrt{2/5}\right) &\text{ if }\sqrt{n}\xi(f^{(n)}_{X,Y})\to c_0\in(0,\infty)
				\\1 &\text{ if }\sqrt{n}\xi(f^{(n)}_{X,Y})\to \infty.
			\end{cases}
		\end{equation}
	\end{theorem}

	Theorem \ref{th:power} reduces the power calculation to calculating the population measure of association under the alternative. Once again we emphasize that these results do not depend on any specific form of the alternative distribution. In \cref{sec:applications} we consider the applications of~\cref{th:power} to certain parametric classes of  alternatives, previously considered in the literature. In doing so, we discover that for smooth parametric alternatives, the detection threshold is seen on a non-standard scale of $n^{-1/4}$. This is much larger than the optimal detection threshold of $n^{-1/2}$ in parametric problems, thereby leading to the suboptimality of $\xi_n$ for testing independence (also see~\cite{shi2020power,cao2020correlations}). We provide a detailed account of this in the following section.
	
	\section{Applications}\label{sec:applications}
	
	Throughout this section, we will use the test $\phi_n$ in~\eqref{eq:asymtest} with $z_{n,\alpha}=n^{-1/2}z_{\alpha}\sqrt{2/5}$, where $z_{\alpha}$ is the upper $\alpha$ quantile of the standard normal distribution. To interpret the detection boundary from~\cref{th:power} in terms of its rate of decay with respect to $n$, it is crucial to note that $\xj$ (see \textbf{(P1)}) is the integrated \emph{squared} distance between a conditional and a marginal distribution function. For example, consider the case where $\jtd^{(n)}(\cdot,\cdot)$ is the standard bivariate density \emph{Gaussian} with correlation $\rho_n$. Then $Y|X=x\sim\mathcal{N}(\rho_n x,1-\rho_n^2)$. Let $\phi(\cdot)$ and $\Phi(\cdot)$ be the standard normal density and distribution functions, respectively. By first-order Taylor approximations, we then get
	\begin{align*}
		\xj&=6\int \E\left[\Phi\left(\frac{t-\rho_n X}{\sqrt{1-\rho_n^2}}\right)-\Phi(t)\right]^2\phi(t)\,dt
		=\rho_n^2\sqrt{3}/\pi+o(\rho_n^2).
		%\\ &=6\int \E\left[\Phi(t)-\rho_nX\phi(t)+o(\rho_n)-\Phi(t)\right]^2\phi(t)\,dt
	\end{align*} 
	This implies $\xj/\rho_n^2\to \sqrt{3}/\pi$, that is, $\xj$ scales like $\rho_n^2$ instead of $\rho_n$. This is the result of $\xj$ being an integrated \emph{squared} distance. Using this observation in~\eqref{eq:asymp_power}, we get that 
	\begin{equation}\label{eq:mixturedet}\lim\limits_{n\to \infty}\bpjr=\begin{cases}\alpha &\text{ if }n^{1/4}\rho_n\to 0\\
			1-\Phi\left(z_{\alpha}-c_1\right) 
			&\text{ if }n^{1/4}\rho_n\to c_0\in(0,\infty)\\
			1 &\text{ if }n^{1/4}\rho_n\to \infty.
		\end{cases}
	\end{equation}
	This shows that a $n^{-1/4}$ detection boundary emerges naturally out of Theorem~\ref{th:power} in the bivariate Gaussian setting.
	
	In this section we describe some applications of \cref{th:power} 
	in three popular local parametric models: mixture-type alternatives and noisy nonparametric regression. The detection thresholds and local powers for these two models are formally stated in~Corollaries~\ref{cor:xi_mix} and~\ref{cor:xi_reg}, respectively. 
	The analysis of local asymptotic power of various tests along parametric alternatives reveals specific features of popular parametric models that control the power of testing procedures. Consequently a lot of attention has been devoted to such analysis (see~\cite{Dhar2016,Kossler2007,Ledwina1986,Morgenstern1956,Nikitin1995,nikitin2003pitman}). Our general result as in Theorem~\ref{th:power} can be used directly to get detection boundaries and local powers under a number of popular alternatives, both along ``contiguous" (meaning $O(n^{-1/2})$ perturbations around the null) and non-contiguous scales, all in one go. We consider two such local parametric models in \cref{sec:applications}: 
	\begin{itemize}
		\item[(a)] Mixture type alternatives used in~~\cite{arias2020detection,Dhar2016,Farlie1960,Farlie1961,shi2020power,Stepanova2008}; see~\cref{sec:appmix},
		\item[(b)] Noisy nonparametric regression used in~\cite{sc_corr}; see~\cref{sec:noisereg}
	\end{itemize}

	\subsection{Simple mixture model}\label{sec:appmix}
	Consider the bivariate density of $(X,Y)$ defined by
	\begin{equation}\label{eq:contam_final}
		f_{X,Y,r}(x,y):=(1-r)\mxd(x)\myd(y)+r \cde(x,y)\quad \forall\ x,y\in\R,
	\end{equation}	
	where $\cde(\cdot,\cdot)$ is a bivariate density function that does not factor into the product of its marginals, $f_{X}(\cdot)$, $f_Y(\cdot)$ are univariate densities, and $r\in [0,1]$. We also note that if $r=0$, then $X$ is independent of $Y$. Therefore it suffices to test if $r=0$ or otherwise. Suppose that $g_{X,Y}(\cdot,\cdot)$ has marginals $f_X(\cdot)$ and $f_Y(\cdot)$, i.e.,
	\begin{equation}\label{eq:asnmix}
		\int_x g_{X,Y}(x,y)\,dx=f_Y(y), \quad \int_y g_{X,Y}(x,y)\,dy=f_X(x).
	\end{equation}
	Note that \eqref{eq:asnmix} implies that the marginals do not change under the alternative. Consequently, we cannot use marginal-based tests (e.g., goodness-of-fit on marginal distributions) to distinguish between the null and alternative. Instead, we will require an independence test as demonstrated above.
	Furthermore, we also assume that there exist $\kappa_1,\kappa_2 >0$, $\eta \in [0,1]$, and $\theta, \gamma>1$ such that
	\begin{equation}\label{eq:Lipschitz_for_mixture}
		\,\,\int_{y}^{\infty}\abs*{g_{Y|X=x_1}(t)-g_{Y|X=x_2}(t)}\,dt\leq (1+L_1(x_1,y)+L_2(x_2,y))\abs*{x_1-x_2}^{\eta},
	\end{equation}
	\begin{equation}\label{eq:Lipmoment}
		\limsup\limits_{n\to\infty} \int (L_i(x,y))^{\theta}\mxd(x)\myd(y)\,dx\, dy \leq \kappa_1,
	\end{equation}
	and
	\begin{equation}\label{eq:moment_for_mixture}
		\limsup\limits_{n\to\infty}\int|x|^{\gamma}f_X(x)\,dx < \kappa_2.
	\end{equation}
	Let us observe that~\eqref{eq:Lipschitz_for_mixture} is a mild regularity assumption on the conditional distribution of $Y$ given $X$ when their joint distribution has density $g(\cdot,\cdot)$. Many common bivariate density functions, like bivariate normal with finite mean and variance, satisfy this assumption. 
	
	%We further assume
	%\begin{equation} \label{eq:preserve_marginal}
	%\int\limits_{\mathbb R}\cde(x,y)\,dy = \mxd(x), \quad \mbox{and,} \quad \int\limits_{\mathbb R}\cde(x,y)\,dx = \myd(y).
	%\end{equation}
	
	To perform local power analysis under model~\eqref{eq:contam_final}, we adopt the same framework as in~\cite{shi2020power,Farlie1960,Stepanova2008}. Towards this direction, fix a sequence $\{r_n\}_{n\geq 1}$ with $r_n\in [0,1]$ for all $n\geq 1$ and consider the family of bivariate densities $\jtdt(\cdot,\cdot)$ as in~\eqref{eq:contam_final}. It is easy to check that the condition~\eqref{eq:nonreg}, i.e., $\xi(\jtdt)\to 0$ holds if $r_n\to 0$. In the same vein as in~\eqref{eq:revprob}, we consider the following testing problem:
	\begin{equation}
		\label{eq:alt_contam}
		\Hy_0: r_n=0 \quad \mbox{vs.} \quad \Hy_{1,n}:r_n>0.%,\ r_n\to 0.
	\end{equation}
	In view of \eqref{eq:nonreg} we focus on the shrinking alternative $r_n\to 0$.
	Our object of interest is the limiting power function, i.e.,
	$$\lim\limits_{n\to\infty} \bpjr , \quad \mbox{where}\ \lim_{n\to\infty} r_n=0.$$
	%It is not hard to verify that the sequence of alternatives $\Hy_{1,n}$ are a sequence of {\it shrinking alternatives} as defined in \ref{def:shrink_alt}. 
	Crucially,~\cref{th:power} reduces the above problem to characterizing the asymptotic behavior of $\xi(\jtdt)$. This is the subject of the following proposition.  
	
	\begin{prop}
		\label{pr:xi_mix}
		Let us consider $\xi(\jtdt)$ defined in \eqref{eq:popdef} for a sequence $\{r_n\}_{n\ge 1}$ such that $r_n\in[0,1]$ and $r_n\to 0$. Then we have
		\begin{equation*}
			\label{eq:asymp_power_contam1}
			\xi(\jtdt) = r^2_n \;\xi(g_{X,Y}).
		\end{equation*}
	\end{prop}
	
	From~\cref{pr:xi_mix} it is evident that $\sqrt{n}\xi(\jtdt)\to 0$ or $\infty$ accordingly as $n^{1/4}r_n\to 0$ or $\infty$. Further, $\xi(\jtdt)$ also increases with $\xi(g_{X,Y})$. Recall from \textbf{(P1)} that $\xi(\cde)$ is a measure of association between $X,Y$ jointly sampled according to $\cde(\cdot,\cdot)$. Therefore,~\cref{pr:xi_mix} shows that the power of $\pha$ is governed by the strength of association between $X,Y$ when they are jointly sampled from $\cde(\cdot,\cdot)$.
	
	We now present the complete characterization of the asymptotic power of $\pha$ for the problem~\eqref{eq:alt_contam}, which follows immediately from~\cref{pr:xi_mix} coupled with~\cref{th:power}.
	\begin{corollary}
		\label{cor:xi_mix}
		Consider the problem of testing $\Hy_0$ versus the sequence of alternatives $\Hy_{1,n}$ defined by \eqref{eq:alt_contam}, for a sequence $\{r_n\}_{n\ge 1}$ such that $r_n\in[0,1]$ and $r_n\to 0$. Suppose that~\eqref{eq:asnmix} holds and Assumptions~\eqref{eq:Lipschitz_for_mixture},~\eqref{eq:Lipmoment}~and~\eqref{eq:moment_for_mixture} are satisfied with $\frac{\theta-1}{\theta}\wedge \eta > \frac{\gamma+1}{2\gamma}$. Then the asymptotic power is given by
		\begin{equation}\label{eq:asymp_power_contam}
			\lim\limits_{n\to \infty}\bpjr=\begin{cases}\alpha &\text{ if }n^{1/4} r_n \to 0
				\\1-\Phi(z_{\alpha}-c_0^2\xi(g_{X,Y})/\sqrt{2/5}) &\text{ if }n^{1/4}r_n\to c_0\in(0,\infty)
				\\1 &\text{ if }n^{1/4}r_n\to \infty.
			\end{cases}
		\end{equation}
	\end{corollary}
	
	\begin{remark}\label{rem:gen}
		In~\cref{cor:xi_mix}, the assumption~\eqref{eq:asnmix} can be dropped. This will not change the detection threshold, but would alter the expression of the local power when $n^{1/4} r_n\to c_0$ to a more complicated and less interpretable expression. We refrain from presenting that version to facilitate easier understanding.
	\end{remark}
	
	In Figures~\ref{fig:pwmix}~and~\ref{fig:hist}, we provide a numerical illustration of~\cref{cor:xi_mix}. We generate data using the model in \cref{sec:appmix} with $n=4000$, $f_X\equiv \mathcal{N}(0,1)$, $f_Y\equiv \mathcal{N}(0,1)$,  $g_{X,Y}\equiv \mathcal{N}(\mathbf{0},\boldsymbol{\Sigma})$ where $$\boldsymbol{\Sigma}=\begin{pmatrix}
		1 &0.95 \\
		0.95 &1
	\end{pmatrix},
	$$
	and the mixing probability $r_n=n^{-b}$, $b\in [0,0.5]$. The power of the test, when averaged over $10000$ runs, is plotted in Figure \ref{fig:pwmix} as $b$ varies in $[0,0.5]$. Here $\beta_*$ is a Monte Carlo estimate (with $10000$ replications) of the limiting local power under model \eqref{eq:contam_final} with $r_n=n^{-0.25}$ and $\boldsymbol{\Sigma}$ as specified above; $\beta_*\approx 0.295$.~\cref{fig:pwmix} clearly shows that the power decays sharply from $1$ to $0$ as $b$ varies between $0.2$ and $0.3$. In fact, when $b$ is close to $0.25$, the empirical power is very close to the theoretical power $\beta^*$. A similar agreement between the empirical and the theoretical distributions is also observed in~\cref{fig:hist} where we plot the histogram of $\sqrt{n}(\xi_n-r_n^2\xi(g_{X,Y}))/\sqrt{2/5}$ under $r_n=n^{-0.25}$ and overlay it with the standard normal density curve, thus verifying \eqref{eq:defbn3} and \cref{pr:xi_mix}.
	
	\begin{figure}
		\centering
		\includegraphics[width = 8cm]{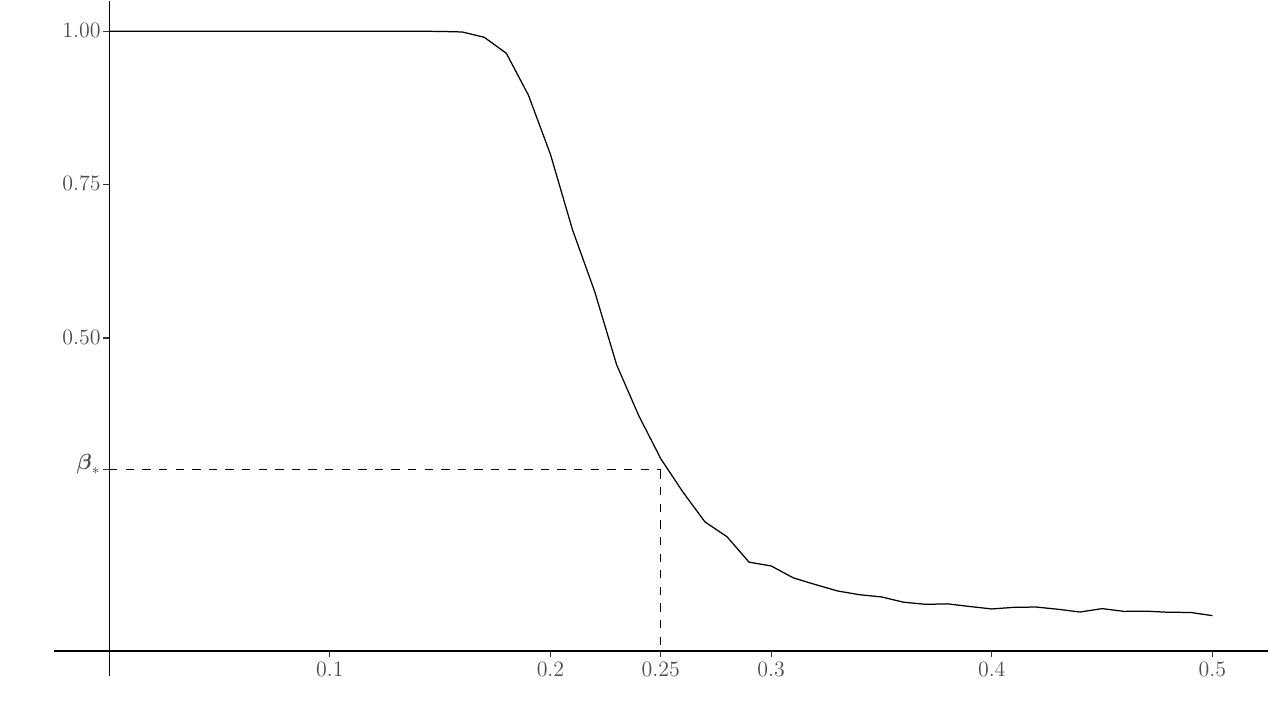}
		\caption{Power of $\phi_n$ for \eqref{eq:alt_contam} at different values of $b$.}\label{fig:pwmix}
	\end{figure}

	\begin{figure}
		\centering
		\includegraphics[width = 8cm]{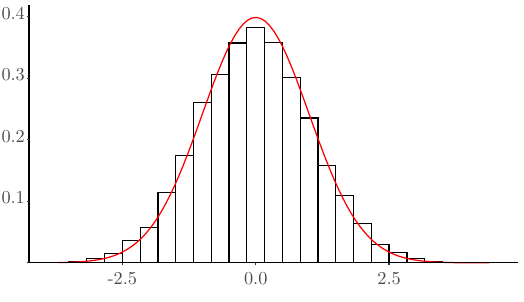}
		\caption{A histogram of $\sqrt{n}(\xi_n-r_n^2\xi(g_{X,Y}))/\sqrt{2/5}$ for $r_n=n^{-0.25}$, with the standard normal density curve overlaid in red.}\label{fig:hist}
	\end{figure}
	
	\subsection{Regression model}\label{sec:noisereg}
	In our main motivating paper~\cite{sc_corr}, the author considered the following model in numerical experiments: 
	\begin{equation}
		\label{eq:reg_mod}
		Y= g(X)+\sigma Z,
	\end{equation}
	where $\sigma \ge 1$, $g:\mathbb{R} \rightarrow \mathbb{R}$, $\mathbb{E}[|g(X)|^3]<\infty$, $Z \sim N(0,1)$ and $X,Z$ are independent. We also assume that $X$ has a finite $\gamma$-th moment for some $\gamma > 3$ and $g$ satisfies
	\[
	|g(x_1)-g(x_2)| \le (1+|x_1|+|x_2|)|x_2-x_1|.
	\]
	This is the classical noisy nonparametric regression model. Note that, as $\sigma \rightarrow \infty$, the ``noise" part of the model~\eqref{eq:reg_mod} dominates the ``signal" part given by $g(X)$. Therefore, as $\sigma\to\infty$, the independence between the ``noise" and the ``signal" makes it harder and harder for the independence testing procedures to have a high power. This makes it interesting to study the performance of $\pha$ in the context of model~\eqref{eq:reg_mod} and to obtain detection thresholds in terms of $\sigma$. 
	
	Therefore, we consider the natural parametric model with $\jtds(\cdot,\cdot)$ being the joint density of $(X,Y)$ drawn according to the model~\eqref{eq:reg_mod} with $\sigma\equiv \sigma_n$. It is easy to check that $\xi(\jtds)\to 0$ as $\sigma_n\to\infty$. Consequently, in the same spirit as in the previous section, we are interested in the following limiting power function:
	%Towards this direction, first reparametrize model~\eqref{eq:reg_mod} by setting $\tau = 1/\sigma$. It is not hard to see that if $h_{X,Y,\tau}$ be the joint distribution of $(X,Y)$, then $\xi(h_{X,Y,\tau})\to 0$ as $\tau\to 0$. 
	%In light of the above statement, let us consider a sequence of shrinking alternatives,
	\begin{equation}
		\label{eq:alt_reg}
		\lim\limits_{n\to\infty} \bpjs , \quad \mbox{where}\ \lim_{n\to\infty} \sigma_n=\infty.
	\end{equation}
	As in~\cref{sec:appmix}, we first state a proposition  characterizing the asymptotics of $\xi(\jtds)$ as $n \rightarrow \infty$.
	\begin{prop}
		\label{pr:xi_reg}
		Consider the model~\eqref{eq:reg_mod} with $\sigma\equiv \sigma_n$. We then have the following identity:
		\[
		\xi(\jtds) = {\sqrt{3}\over \pi} \sigma_n^{-2}\mathrm{Var}(g(X)) + o(\sigma_n^{-2}).
		\]
	\end{prop}
	%(see Section A.2.2 of the Supplementary Material~\cite{ADN21-supp})
	~\cref{pr:xi_reg} shows that $\sqrt{n}\xi(\jtds)\to 0$ or $\infty$ accordingly as $n^{-1/4}\sigma_n\to \infty$ or $0$. Also $\xi(\jtds)$ increases with $\mbox{Var}(g(X))$. This is indeed very intuitive. Note that if $\mbox{Var}(g(X))=0$ then $g(\cdot)$ is a constant function which means $X,Y$ are independent according to model~\eqref{eq:reg_mod}. Also note that $\mbox{Var}(\E(Y|X))=\mbox{Var}(g(X))$ and $\E\mbox{Var}(Y|X)=\sigma^2_n$, i.e., $\mbox{Var}(g(X))$ measures the proportion of the total variance of $Y$ which is explained by $X$. Therefore, it is only natural that the larger the value of $\mbox{Var}(g(X))$, the larger is the power of $\pha$.
	
	We now present the complete answer to problem~\eqref{eq:alt_reg}, which follows immediately from~\cref{pr:xi_reg} coupled with~\cref{th:power}.
	\begin{corollary}
		\label{cor:xi_reg}
		Consider the problem in~\eqref{eq:reg_mod}. Then the asymptotic power is given by
		\begin{equation}\label{eq:asymp_power_reg}
			\lim\limits_{n\to \infty}\bpjs=\begin{cases}\alpha &\text{ if }n^{-1/4}\sigma_n \to \infty
				\\1-\Phi(z_{\alpha}-c_0^2\sqrt{3}/\pi \,\mathrm{Var}(g(X))/\sqrt{2/5}) &\text{ if }n^{-1/4}\sigma_n\to c_0^{-1}\in(0,\infty)
				\\1 &\text{ if }n^{-1/4}\sigma_n\to 0.
			\end{cases}
		\end{equation}
	\end{corollary}
	
	\subsection{Rotation alternatives}\label{sec:rot-alt}
	As a third example, we now consider the pair of random variables $(X,Y)$ satisfying
	\begin{equation}
		\label{eq:rotation}
		\begin{pmatrix}X\\Y\end{pmatrix} = \begin{pmatrix}1 & \Delta\\\Delta & 1\end{pmatrix}\begin{pmatrix}U\\V\end{pmatrix},
	\end{equation}
	where $U,V$ are independent, zero mean random variables with densities $f_1$ and $f_2$. We further assume that $f_1$ and $f_2$ are twice differentiable with the $i$-th derivative, $i \in \{0,1,2\}$ being denoted by $f^{(i)}_1$ and $f^{(i)}_2$ respectively. The $0$-th derivative is the function itself. Note that $X$ and $Y$, drawn according to model~\eqref{eq:rotation}, are independent if and only if $\Delta=0$.
	
	To perform a local power analysis, we adopt the same framework as in~\cite{gieser1993,Konijn1956,shi2020power}. Consider $(X,Y)\sim \jtdd(\cdot,\cdot)$ as in~\eqref{eq:rotation} with $\Delta\equiv\Delta_n$. It is easy to check that $\xi(\jtdd)\to 0$ holds if $\Delta_n\to 0$. In the same vein of the earlier examples, we are interested in studying the following limiting power function
	\begin{equation}
		\label{eq:alt_rot}
		\lim\limits_{n\to\infty} \bpjd , \quad \mbox{where}\ \lim_{n\to\infty} \Dl_n=0.
	\end{equation}
	
	Before stating the main results of this section, we need some assumptions which are encapsulated below. Note that the first two assumptions are also required for the our main results in Section~\ref{sec:main_results}.
	
	\begin{Assumption}\label{as:Lipschitz} There exist functions $L_i^{(n)}(\cdot):\R\to [0,\infty)$ for $i=1,2$ and numerical constants $\kappa_1>0$, $\eta\in (0,1]$, and $\theta>1$ such that $\forall\, y,x_1,x_2,$
		\begin{equation}\label{eq:Lipschitz1}
			\,\,\abs*{\P(Y_{n,1}\geq y|X_{n,1}=x_1)-\P(Y_{n,1}\geq y|X_{n,1}=x_2)}\leq (1+L_1^{(n)}(x_1,y)+L_2^{(n)}(x_2,y))\abs*{x_1-x_2}^{\eta},
		\end{equation}
		\begin{equation}\label{eq:Lipschitz2}
			\limsup\limits_{n\to\infty} \int (L_i^{(n)}(x,y))^{\theta}\mxd^{(n)}(x)\myd^{(n)}(y)\,dx\, dy \leq \kappa_1,
		\end{equation}
		where $\mxd^{(n)}(\cdot)$ and $\myd^{(n)}(\cdot)$ are the marginal densities of $X_{n,1}$ and $Y_{n,1}$ under the joint density $\jtd^{(n)}(\cdot,\cdot)$.	
	\end{Assumption}
	\begin{Assumption}\label{as:nn_dists} There exist numerical constants $\gamma>1$ and $\kappa_2>0$ such that
		$$\limsup\limits_{n\to\infty} \E|X_{n,1}|^{\gamma}\leq \kappa_2.$$
	\end{Assumption}
	
	Further, we make the following  additional assumption for the rotation alternative.
	
	\begin{Assumption} \label{as:rotation}\hfill 
		\begin{itemize}
			\item[(1)] $\EE U=\EE V=0$ and $\EE U^2=\EE V^2=1$.
			\item[(2)] Both $f_1$ and $f_2$ are continuous and twice differentiable. As $t \rightarrow \infty$,
			\[
			|tf^\prime_1(t)| \rightarrow 0,  \quad \quad \mbox{and} \quad \quad |tf^\prime_2(t)| \rightarrow 0.   
			\]
			\item[(3)] There exist an $\epsilon>0$, and real-valued functions $R_{1,\epsilon}(\cdot,\cdot)$ and $R_{2,\epsilon}(\cdot,\cdot)$ such that
			$$\max_{k\in\{0,1,2\}}\sup_{|\Dl|\leq\epsilon} \bigg|\frac{\partial^k}{\partial\Delta^k}f_i\left(\frac{u-\Dl v}{1-\Dl^2}\right)\bigg|\leq R_{i,\epsilon}(u,v)$$
			for all $u,v$ and $i=1,2$. Further, for all 
			$ \ell \in \{0,1,2,3\}$,
			\[
			\mathbb{E}\left[\frac{R_{1,\epsilon}(U,V)}{f_1(U)}|U|^\ell\right]^2 <\infty, \quad\quad \mathbb{E}\left[\left(R_{2,\epsilon}(U,V)\right)|V|^\ell\right]^2 < \infty.
			\]
		\end{itemize}
	\end{Assumption}
	These assumptions are natural and hold in various commonly used models including the normal distribution, the $t$ distribution with sufficiently high degrees of freedom (10 or more) and various other distributions in the exponential family. Under these assumptions we have the following proposition which is proved in Section B.2.3.
	\begin{prop}
		\label{pr:xi_rot}
		Consider the model defined by \eqref{eq:rotation} with $\Dl\equiv \Dl_n$. Then, under the \cref{as:rotation}, we have
		\begin{equation}
			\xi(\jtdd) = \Delta_n^2V_0 + O(\Delta_n^3),
		\end{equation}
		where $$V_0:=6\Bigg(\mathbb{E}_V[f^2_2(V)]+\mathbb{E}_U\Big[\frac{f^\prime_1(U)}{f_1(U)}\Big]^2\mathbb{E}_V[J^2(V)]-2\mathbb{E}_V[J(V)f_2(V)]\Bigg)
		$$
		and $
		J(t):=\int\limits_{-\infty}^{t}y_2f_2(y_2)\,dy_2.
		$
	\end{prop}
	
	The complete answer to problem~\eqref{eq:alt_rot} now follows by combining~\cref{pr:xi_rot} with Theorem~\ref{th:power}, and is presented below.
	\begin{corollary}
		\label{cor:xi_rot}
		Suppose Assumptions~\ref{as:Lipschitz},~\ref{as:nn_dists},~and~\ref{as:rotation} hold with $\frac{\theta-1}{\theta}\wedge \eta > \frac{\gamma+1}{2\gamma}$. Then the asymptotic power in~\eqref{eq:alt_rot} is given by
		\begin{equation}\label{eq:asymp_power_rot}
			\lim\limits_{n\to \infty}\bpjd=\begin{cases}\alpha &\text{ if }n^{1/4}\Dl_n \to 0
				\\1-\Phi(z_{\alpha}-c_0^2\,V_0/\sqrt{2/5}) &\text{ if }n^{1/4}\Dl_n\to c_0\in(0,\infty)
				\\1 &\text{ if }n^{1/4}\Dl_n\to \infty,
			\end{cases}
		\end{equation}
		where $V_0$ is as defined in~\cref{pr:xi_rot}.
	\end{corollary}
	
	%\begin{remark}
	%	It is not hard to check that the alternatives considered in this section do not fall into the standard contiguity framework even at the $n^{-1/2}$ scale, i.e., when $n^{-1/2}\sigma_n=c_0$. Therefore, for this class of models, the Le Cam's third lemma based techniques used in~\cite{cao2020correlations,shi2020power} no longer apply. In contrast, our results, for example,~\cref{th:power} which covers a significantly bigger class of alternatives, continue to apply.
	%\end{remark}
	
	\section{Minimax optimality of $\pha$ for testing degree of association}\label{sec:degtest}
	Recall our basic setting, that is, $(X_1,Y_1),\ldots , (X_n,Y_n)\overset{i.i.d.}{\sim}f_{X,Y}^{(n)}$. In the  earlier sections, we focused on the case where $\xj$, the Dette-Siburg-Stoimenov measure, \emph{converges to $0$} as $n\to\infty$. On the contrary, the focus of this section is on the other regime.
	\begin{equation}\label{eq:gencon}
		\xj\downarrow\xi_0>0\qquad \mbox{as} \ n\to\infty.
	\end{equation} 
	This regime is of particular importance because the value of $\xj$ provably encodes how noisy the functional relationship between $X$ and $Y$ is. In particular, in~\cite[Theorem 2]{dette2013copula}~and~\cite[Theorem 1.1]{sc_corr}, the authors show that $\xi_0=1$ implies $Y$ is a noiseless function of $X$. Further, in~\cite[Theorems 1 and 2]{dette2013copula}, the authors also show that $\xj$ can be used to define a natural notion of \emph{dependence ordering} based on how well $Y$ can be predicted from $X$. Through some explicit computations, in~\cite[Section 4, Examples 1.(a) --- (d)]{dette2013copula}, the authors further prove that in multiple copula based dependence models, $\xj$ is a \emph{strictly monotonic} function of the natural dependence parameters. Moving to~\cite{sc_corr}, the author uses \emph{extensive numerical experiments} to show that in a number of other examples, including variants of noisy nonparametric regression, \cite{sc_corr} also shows through extensive simulations, that $\xj$ decreases monotonically with increasing noise levels. Overall, these results show that through invertible transformations, one can directly convert $\xj$ to natural dependence parameters in a variety of models for $(X_1,Y_1)$. Therefore, drawing inference about $\xj$ immediately leads to interpretable inference about the nature of dependence in a large collection of models. We expand this collection of models further in the following proposition by showing that $\xj$ is  a monotonic function of natural dependence parameters in all examples from~\cref{sec:applications}. 
	
	\begin{prop}[Monotonicity of $\xj$]\label{prop:ximon}
		Recall the definition of $\xj$ from~\eqref{eq:popdef}.
		
		\begin{itemize}
			\item[(1)] Suppose $$(X_1,Y_1)\sim \mathcal{N}\left(\begin{pmatrix} \mu_1\\ \mu_2\end{pmatrix},\begin{pmatrix} \sigma_1^2 & \sigma_1\sigma_2\rho\\ \sigma_1\sigma_2\rho & \sigma_2^2\end{pmatrix}\right),$$
			for $\mu_1,\mu_2\in\R$, $\sigma_1,\sigma_2>0$, $|\rho|\leq 1$. 
			Then $\xj$ is free of $\mu_1,\mu_2,\sigma_1,\sigma_2$ and a strictly increasing function of $|\rho|$.
			\item[(2)] Suppose that $(X_1,Y_1)$ is distributed according to the mixture model in~\eqref{eq:contam_final} with $g_{X,Y}(\cdot)$ further satisfying~\eqref{eq:asnmix}. Then $\xj$ is a strictly increasing function of $r$.
			\item[(3)] Suppose $(X_1,Y_1)$ is distributed according to the regression model~\eqref{eq:reg_mod}, for $\sigma>0$. Then $\xj$ is a strictly decreasing function of $\sigma$.
		\end{itemize}
	\end{prop}
	Motivated by the observations made above, we focus on the following natural hypothesis testing problem in this section:
	\begin{equation}\label{eq:newhyp}
		\mathrm{H}_0: \xj=\xi_0\qquad \mbox{versus}\qquad \mathrm{H}_{1,n}: |\xj-\xi_0|\geq c_n,
	\end{equation}
	for some \emph{positive} sequence $\{c_n\}_{n\geq 1}$ and $\xi_0\in (0,1)$. Clearly if $c_n=O(1)$, then from~\cite[Theorem 1.1]{sc_corr}, Chatterjee's correlation $\xi_n$ can be used to consistently separate $\mathrm{H}_0$ from $\mathrm{H}_{1,n}$. On the other hand, our focus here is when $c_n=o(1)$ which leads to~\eqref{eq:gencon}. When $\xi_0\in (0,1)$, \eqref{eq:newhyp} can be viewed as a slightly general (two-sided) version of \eqref{eq:revprob}. The goal of this section is to address the following pair of subtle 
	questions:
	
	\begin{itemize}
		\item What is the fastest decaying sequence $c_n$ such that Chatterjee's correlation coefficient $\xi_n$ can still separate $\mathrm{H}_0$ from $\mathrm{H}_{1,n}$?
		\item Conversely, what is the slowest decaying sequence $c_n$ such that no test can separate $\mathrm{H}_0$ from $\mathrm{H}_{1,n}$ in the local asymptotic minimax sense?
	\end{itemize}
	
	We will prove in this section that, for Chatterjee's correlation, the detection boundary occurs at $c_n\gg n^{-1/2}$. However, in sharp contrast to the $\xi_0=0$ case, $c_n\approx n^{-1/2}$ is indeed the minimax optimal threshold, in that no test can consistently separate the two hypotheses when $c_n\approx n^{-1/2}$. This indicates that Chatterjee's correlation based test is indeed \emph{minimax optimal} for testing the degree of association when the two variables are not exactly independent. In the sequel, we will formalize these two  notions.
	
	We begin with some notation. Let $\Gamma:=(C,\eta,\theta,\gamma)$ where the aforementioned constants are taken from Assumptions~\ref{as:Lipschitz}~and~\ref{as:nn_dists}. Consider the following family of distributions:
	\begin{equation}\label{eq:formalt}
		\begin{small}
			\hloc:=\left\{f_{X,Y}^{(n)}(\cdot): |\xj-\xi_0|\geq c_n,\ f_{X,Y}^{(n)}\ \mbox{satisfies\ \ref{as:Lipschitz},\ \ref{as:nn_dists}\ with\ constants}\ \Gamma=(C,\eta,\theta,\gamma)\right\}.
		\end{small}
	\end{equation}
	In other words, the family $\hloc$ consists of the family of joint distributions which admit $\xi(\cdot)$ values at a $\ge c_n$ distance away from the null value $\xi_0$, and also satisfy Assumptions~\ref{as:Lipschitz}~and~\ref{as:nn_dists}. Next, given a test function $\Phi_n\equiv \Phi_n((X_1,Y_1),\ldots ,(X_n,Y_n))$, consider its power function
	\begin{equation}\label{eq:poweralt}
		\beta_{\Phi_n}(\hloc):=\inf_{f_{X,Y}^{(n)}(\cdot,\cdot)\in \hloc} \P(\Phi_n\ \mbox{rejects}\ \hloc).
	\end{equation}
	In terms of the definitions in~\eqref{eq:formalt}~and~\eqref{eq:poweralt}, the goal of a good test $\Phi_n$ would be to ensure that $\beta_{\Phi_n}(\hloc)\to 1$ as $n\to\infty$ for ``small" values of $c_n$. Given a level parameter $\alpha\in (0,1)$, We define our candidate good test as follows: 
	\begin{equation}\label{eq:goodtest}
		\phi_n:=\mathbbm{1}\left(|\xi_n-\xi_0|\ge \frac{1+5\sqrt{\log{(2/\alpha)}}}{\sqrt{n}}\right).
	\end{equation}
	The following theorem provides matching upper and lower bounds for $\beta_{(\cdot)}(\hloc)$.
	\begin{theorem}[Testing for degree of association]\label{thm:test_assoc}\hfill
		\begin{itemize}
			\item[(1)] \emph{(Upper bound for $\xi_n$)}. Fix an arbitrary $\alpha\in (0,1)$ and consider the test $\phi_n$ from~\eqref{eq:goodtest}. Then it is an asymptotically level $\alpha$ under $\mathrm{H}_0$ and 
			$$\beta_{\phi_n}(\hloc)\to 1$$
			provided $\sqrt{n}c_n\to\infty$.
			\item[(2)] \emph{(Minimax lower bound)}. There exists some $\alpha\in (0,1)$ such that for any level $\alpha$ test $\Phi_n$ based on $(X_1,Y_1),\ldots ,(X_n,Y_n)$, the following holds:
			$$\liminf\limits_{n\to\infty} \beta_{\Phi_n}(\hloc)< 1$$
			whenever $\liminf\limits_{n\to\infty}\sqrt{n}c_n<\infty$.
		\end{itemize}
	\end{theorem}
	
	The two parts of~\cref{thm:test_assoc} suggest that Chatterjee's correlation based test is particularly suitable for testing the strength of association between two random variables and is able to detect small departures from a fixed \emph{non-zero} value of $\xj$ (i.e., $\xi_0$) at the optimal rate. It is worth noting that the lower bound in~\cref{thm:test_assoc} heavily relies on the fact that $0<\xi_0<1$.
	
	\begin{remark}
		In~\cite{lin2022limit}, the authors show that $$\frac{\sqrt{n}(\xi_n-\xj)}{\hat{\sigma}_n}\overset{w}{\longrightarrow} \mathcal{N}(0,1),$$
		where $\hat{\sigma}_n$ is a function of the data, and $\hat{\sigma}_n=O_p(1)$. Therefore, a natural alternative to $\phi_n$ from \eqref{eq:goodtest} would be
		$$\phi_n=\mathbbm{1}(\sqrt{n}|\xi_n-\xi_0|\ge z_{\alpha/2}),$$
		where $z_{\alpha/2}$ is the upper $\alpha/2$ standard Gaussian quantile and $\xj=\xi_0$. This test was proposed in \cite[Remark 1.4]{lin2022limit} and it has asymptotic size exactly equal to $\alpha$. On the flip side, from the definition of $\hat{\sigma}_n$ (see \cite[Theorem 1.1]{lin2022limit}), it seems that it has $O(n^2)$ time complexity which may be a greater computational burden depending on the application at hand.
	\end{remark}
	
	\begin{remark}
		In \cite{lin2021boosting}, the authors show that if $(X,Y)$ have a bivariate Gaussian distribution, then the detection boundary for $\xi_n$ when $\xi_0=0$ can be improved to near parametric rates by incorporating more ``right nearest neighbors". We would conjecture that the test in \cite{lin2021boosting} achieves $O(n^{-1/2})$ detection boundary when $\xi_0\in (0,1)$ and $(X,Y)$ are non-Gaussian. In view of \cref{thm:test_assoc}, this would imply that the test in \cite{lin2021boosting} attains (near) parametric efficiency whenever $\xi_0\in [0,1)$ (at the expense of greater computational complexity). A conclusive answer on this and an inspection of the relevant applications might be of independent interest.
	\end{remark}
	
	\appendix
	
	\section{Proofs of main results}\label{sec:proofs}
	In this section, we will prove the main results in this paper. The proofs require a number of technical results, which are proved in Sections 2 and 3 of the Appendix B.
	\subsection{Proofs from Section~\ref{sec:main_results}}
	\label{sec:proofs-sec2}
	\subsubsection{Proof of Theorem~\ref{th:xi_n_wass}}\label{sec:pfxi_n_wass}
	\emph{Part (i).} As mentioned in Section~\ref{sec:main_results}, our proof proceeds through studying an oracle version of $\xi_n$. Observe that $\xi_n$ can be rewritten as 
	\begin{equation}\label{eq:to_proj}
		\begin{split}
			\xi_n=\dfrac{3n}{n^2-1}\left(\dfrac{1}{n}\underset{i\neq j}{\sum\sum }\abs*{\hat{F}_n(Y_{n,i})-\hat{F}_n(Y_{n,j})}-\sum_{i=1}^n\abs*{\hat{F}_n(Y_{n,(i+1)})-\hat{F}_n(Y_{n,(i)})}\right),
		\end{split}
	\end{equation}
	
	where $\hat{F}_n(\cdot )$ is the empirical cumulative distribution function (CDF) of $Y_{n,1},\ldots ,Y_{n,n}$. Let $F^{(n)}_Y(\cdot)$ denote the population CDF of $Y_{n,1}$. Recall that $X_{n,(1)}\leq \ldots \leq X_{n,(n)}$ and $Y_{n,(i)}$ is the $Y$ value concomitant to $X_{n,(i)}$. The main idea is to show that we can replace $\hat{F}_n(\cdot )$ by $F^{(n)}_Y(\cdot)$ asymptotically. In other words, we show that $\xi_n$ is close (with quantitative error bounds) to $\xi_n^*$ where
	\begin{equation}
		\label{eq:oracle_stat}
		\xi_n^*:=\dfrac{3n}{n^2-1}\left(\dfrac{1}{n}\underset{i\neq j}{\sum\sum }\abs*{\F(Y_{n,i})-\F(Y_{n,j})}-\sum_{i=1}^n\abs*{F^{(n)}_Y(Y_{n,(i+1)})-F^{(n)}_Y(Y_{n,(i)})}\right).
	\end{equation}
	The following theorem characterizes the asymptotic variances and the distance between $\xi^*_n$ and $\xi_n$.
	\begin{theorem}\label{th:fstarvar} 
		Suppose that Assumptions \ref{as:Lipschitz} and \ref{as:nn_dists} hold. Then there is a  constant $C>0,$ such that for $n\ge 1,$ 
		$$\max\left\{n\mathrm{Var}(\xi_n-\xi_n^*),\abs*{n\mathrm{Var}(\xi_n)-{2\over 5}}\right\}\le Cn^{-1}+C\xj+C\sqrt{n\log n}\,b_n\ind(\xj>0)$$
		where 
		\begin{equation}\label{eq:defbn}
			b_n:=n^{-\frac{\gamma}{\gamma+1}}(\log{n})^2+\left(\dfrac{(\log{n})^2}{n}\right)^{\left(\frac{\gamma(\theta-1)}{\theta(\gamma+1)}\wedge \frac{\eta \gamma}{\gamma+1}\right)}.
		\end{equation}
		%for any shrinking alternative $H_{1,n,\downarrow}(\Delta_n)$ defined in \ref{def:shrink_alt}. 
	\end{theorem}
	The next theorem characterizes the rate of convergence of $\frac{\sqrt{n}\big(\xi_n^*-\EE(\xi_n^*)\big)}{\sigma_n^*}$ to normality.
	\begin{theorem}
		\label{th:oracle_wass} 
		There is a constant $C^*>0$ such that for all $n\ge 1$ we have
		$$\mD\left({\sqrt{n}\big(\xi_n^*-\EE(\xi_n^*)\big)\over\sigma_n^*}\right) \le \dfrac{C^*}{\sqrt{n}}\left(\dfrac{1}{\sigma_n^{*2}}+\dfrac{1}{2\sigma_n^{*3}}\right),$$
		where $\sigma_n^*:=\sqrt{\mathrm{Var}(\xi_n^*)}$ and $\mD$ is the Wasserstein distance to normality defined in Definition 2.2 of the main paper.
	\end{theorem}
	
	We defer the proofs of these theorems to Section 2 of Appendix B, and proceed to use these results to prove Theorem~\ref{th:xi_n_wass}.
	
	For any  $p\in(0,1)$ we can write
	\begin{equation}\label{eq:def_g}
		f_{Y|X}^{(n)}(y|x)=(1-p)f_Y^{(n)}(y)+p g(y|x),\text{ where } g(y|x)=\frac{f_{Y|X}^{(n)}(y|x)-f_Y^{(n)}(y)}{p}+f_Y^{(n)}(y).
	\end{equation}
	Here $f_Y^{(n)}(y)=\int \jtd^{(n)}(x,y)dx$ is the marginal density of $Y_{n,1}$. We have
	\begin{equation}\label{eq:g_cond_CDF}
		G(y|x):=\int _{-\infty}^yg(t|x)dt=\tfrac{1}{p}F_{Y|X}^{(n)}(y|x)+(1-\tfrac{1}{p})F_Y^{(n)}(y);\quad\quad \int G(y|x)f_X^{(n)}(x)dx=F_Y^{(n)}(y).
	\end{equation}
	
	For each $i \in [n]$, let $N(i)\in \{j\in[n]:\mathrm{Rank}(X_{n,j})=\mathrm{Rank}(X_{n,i})+1\}$ be the unique index $j$ such that $X_{n,j}$ is immediately to the right of $X_{n,i}$ when $X_{n,i}$'s are arranged in increasing order. If there are no such indices for some $i$, set the corresponding $N(i)=1$. To show part (i) of Theorem~\ref{th:xi_n_wass} we begin by observing that $\EE(\sqrt{n}\xi_n)$ can be simplified as follows.
	\begin{align}
		\label{eq:mean_xi_n}
		&\EE(\sqrt{n}\xi_n)\nonumber
		\\=&\sqrt{n}\EE\left(1-\frac{3}{n^2-1}\sum_{i=1}^{n-1}\abs*{R_i-R_{N(i)}}\right)\nonumber
		\\=&\sqrt{n}\EE\left[1-\frac{3}{n^2-1}\sum_{i=1}^{n-1}(R_i+R_{N(i)}-2\min\{R_i,\,R_{N(i)}\})\right]\nonumber
		\\=&\sqrt{n}\EE\left(-2+\frac{6}{n^2-1}\sum_{i=1}^{n-1}\min\{R_i,\, R_{N(i)}\}\right)+O(n^{-1/2})\nonumber
		\\=&-2\sqrt{n}+\dfrac{6n^{3/2}}{n^2-1}\left(1+\sum_{k\neq 1, N(1)}\EE \left(\EE(\ind(Y_{n,k}\le \min\{Y_{n,1},Y_{n,N(1)}\})|\bX^{(n)})\right)\right)+O(n^{-1/2}).
	\end{align} 
	Hence, to characterize the asymptotic behavior of $\EE(\sqrt{n}\xi_n)$, we focus on the asymptotics of the term $\EE \left(\EE(\ind(Y_{n,k}\le \min\{Y_{n,1},Y_{n,N(1)}\})|\bX^{(n)})\right)$.
	
	Depending on which part of $f_{Y|X}^{(n)}$ contributes to the random variables $Y_{n,1},Y_{n,N(1)},Y_{n,k}$, we can separate $\EE(\sqrt{n}\xi_n)$, we focus on the asymptotics of the term $\EE \left(\EE(\ind(Y_{n,k}\le \min\{Y_{n,1},Y_{n,N(1)}\})|\bX^{(n)})\right)$ into the following terms.
	\begin{equation}\label{eq:sumterms}
		\EE\left(\EE(\ind(Y_{n,k}\le \min\{Y_{n,1},Y_{n,N(1)}\})|\bX^{(n)})\right)=:(1-p)^3T_0+(1-p)^2p T_1+(1-p)p^2T_2+p^3T_3.
	\end{equation}
	Here 
	\begin{align*}
		T_0:=&~
		\int\ind(y_k<\min\{y_1,y\})f_Y^{(n)}(y)f_Y^{(n)}(y_1)f_Y^{(n)}(y_k)
		dydy_1dy_k
	\end{align*}
	while the other terms are defined as
	\begin{align*}
		T_1:=
		&~
		\int\int \sum_{j\neq 1}\ind(N(1)=j) \ind(y_k<\min\{y_1,y_j\})f_Y(y_1)f_Y(y_k)g(y_j|x_j)dy_1dy_kdy_jf_{\bX}(\bx)\,d\bx
		\\&
		+\int\int \sum_{j\neq 1}\ind(N(1)=j) \ind(y_k<\min\{y_1,y_j\})f_Y(y_1)g(y_k|x_k)f_Y(y_j)dy_1dy_kdy_j f_{\bX}(\bx)\,d\bx
		\\&
		+\int\int \sum_{j\neq 1}\ind(N(1)=j) \ind(y_k<\min\{y_1,y_j\})g(y_1|x_1)f_Y(y_k)f_Y(y_j)dy_1dy_kdy_j f_{\bX}(\bx)\,d\bx,\\
		T_2:=
		&~
		\int \int \sum_{j\neq 1}\ind(N(1)=j) \ind(y_k<\min\{y_1,y_j\})g(y_1|x_1)g(y_k|x_k)f_Y(y_j)dy_1dy_kdy_j f_{\bX}(\bx)\,d\bx
		\\&
		+\int \int \sum_{j\neq 1}\ind(N(1)=j) \ind(y_k<\min\{y_1,y_j\})f_Y(y_1)g(y_k|x_k)g(y_j|x_j)dy_1dy_kdy_j f_{\bX}(\bx)\,d\bx
		\\&
		+\int \int \sum_{j\neq 1}\ind(N(1)=j) \ind(y_k<\min\{y_1,y_j\})g(y_1|x_1)f_Y(y_k)g(y_j|x_j)dy_1dy_kdy_j f_{\bX}(\bx)\,d\bx.	\\
		T_3:=
		&~
		\int \int \sum_{j\neq 1,k}\ind(N(1)=j) \ind(y_k<\min\{y_1,y_j\})g(y_1|x_1)g(y_k|x_k)g(y_j|x_j)dy_1dy_kdy_j f_{\bX}(\bx)\,d\bx
		\\&
		=\int \int \sum_{j\neq 1}\ind(N(1)=j) \ind(y_k<\min\{y_1,y_j\})g(y_1|x_1)f_Y(y_k)g(y_j|x_j)dy_1dy_kdy_j f_{\bX}(\bx)\,d\bx.
	\end{align*}
	We can show that $T_0$ satisfies the following.
	\begin{align*}
		T_0=&\int\ind(y_k<\min\{y_1,y\})f_Y^{(n)}(y)f_Y^{(n)}(y_1)f_Y^{(n)}(y_k)dydy_1dy_k\\
		=&\PP(Y_{n,k}<\min\{Y_{n,1},Y_{n,2}\}|Y_{n,1},Y_{n,2},Y_{n,k}\sim f_Y^{(n)})=\frac{1}{3}.\numberthis\label{eq:term_T0}
	\end{align*}

	Let us define
	$$
	H(f,g):=\int \EE\left(\int_t^{\infty} g(y|X_{n,1})dy\right)^2f_Y^{(n)}(t)dt=\int \EE(1-G(t|X_{n,1}))^2f_Y^{(n)}(t)dt.
	$$
	The second, third, and fourth terms can be controlled using the following lemma which has been proved in Section 3 of Appendix B:
	\begin{lemma}\label{lem:bias_rems} 
		$$		T_1=1+r_n, \,\,
		T_2=\dfrac{2}{3}+H(f,g)+r_n, \,\,
		T_3=H(f,g)+r_n,
		$$
		where $|r_n|\lesssim b_n\ind(\xj>0).$
	\end{lemma}
	
	We have from \eqref{eq:sumterms} that
	\begin{align*}\label{eq:sumbound}
		&\EE\left(\EE(\ind(Y_{n,k}\le \min\{Y_{n,1},Y_{n,N(1)}\})|\bX^{(n)})\right)
		\\=&(1-p)^3T_0+(1-p)^2p T_1+(1-p)p^2T_2+p^3T_3
		\\=&(1-p)^3\cdot\dfrac{1}{3}+(1-p)^2p\left[1+r_n\right]
		+(1-p )p ^2\left[\dfrac{2}{3}+H(f,g)+r_n\right]+p ^3\cdot A_3
		\\=&\dfrac{1}{3}+p \left(-1+1\right)+p^2\left(1-2+\dfrac{2}{3}+H(f,g)\right)+p ^3\left(-\dfrac{1}{3}+1-\dfrac{2}{3}-H(f,g)+H(f,g)\right)+r_n
		\\=&\dfrac{1}{3}+p ^2\left(H(f,g)-\dfrac{1}{3}\right)+r_n,\numberthis
	\end{align*}
	where $|r_n|\le Cn^{-\frac{\gamma}{\gamma+1}}(\log{n})^2+C\left(\dfrac{(\log{n})^2}{n}\right)^{\left(\frac{\gamma(\theta-1)}{\theta(\gamma+1)}\wedge \frac{p \gamma}{\gamma+1}\right)}$ for some constant $C>0$.
	
	Recall that $\xi(f_{X,Y}^{(n)})$ is the population measure of association when $X\sim f_X^{(n)},$ $Y\sim f_Y^{(n)},$  and $Y|X\sim f_{Y|X}^{(n)}(y|x)\equiv (1-p )f_Y^{(n)}(y)+p g(y|x).$ Then
	\begin{align*}\label{eq:popcorr_alt}
		\xi(f_{X,Y}^{(n)})/6 
		=&\int\EE[\PP(Y_{n,1}\ge t|X_{n,1})-\PP(Y_{n,1}\ge t)]^2f_Y^{(n)}(t)dt \\
		=&\int\bigg[(1-p )^2\Var\PP(Y_{n,1}\ge t) +p ^2\Var(1-G(t|X_{n,1}))\\
		&\hspace*{2cm}+2p (1-p )\Cov(\PP(Y_{n,1}\ge t),1-G(t|X_{n,1}))\bigg]f_Y^{(n)}(t)dt\\
		=&p^2\int\Var(1-G(t|X_{n,1}))f_Y^{(n)}(t)dt.
	\end{align*}
	The last equality follows since $\PP(Y_{n,1}\ge t)$ does not depend on $X_{n,1}.$
	By the definition, $H(f,g)=\int\EE((1-G(t|X_{n,1}))^2)f_Y^{(n)}(t)dt.$ Next, 
	$$
	\int(\EE(1-G(t|X_{n,1})))^2f_Y^{(n)}(t)dt=\int (1-F_Y^{(n)}(t))^2f_Y^{(n)}(t)dt=\EE (1-F_Y^{(n)}(Y))^2=\int_0^1u^2du=\dfrac{1}{3}.
	$$
	This means 
	\begin{equation}\label{eq:H_to_xi}
		6p ^2\left(H(f,g)-\dfrac{1}{3}\right)=6p ^2\int\Var(1-G(t|X_{n,1}))f_Y^{(n)}(t)dt=\xi(f_{X,Y}^{(n)}).
	\end{equation}
	
	\noindent Plugging \eqref{eq:sumbound} and \eqref{eq:H_to_xi} back into \eqref{eq:mean_xi_n},
	\begin{align*}
		&\EE(\sqrt{n}\xi_n)
		\\=&-2\sqrt{n}+\dfrac{6n^{3/2}}{n^2-1}\left(1+\sum_{k\neq 1, N(1)}\EE \left(\EE(\ind(Y_{n,k}\le \min\{Y_{n,1},Y_{n,N(1)}\})|\bX^{(n)})\right)\right)+O(n^{-1/2})	
		\\=&-2\sqrt{n}+\dfrac{6n^{5/2}}{n^2-1}\left[\dfrac{1}{3}+p ^2\left(H(f,g)-\dfrac{1}{3}\right)+3r_n\right]+O\left(\dfrac{1}{\sqrt{n}}\right)
		\\=&\,6\sqrt{n}p ^2\left(H(f,g)-\dfrac{1}{3}\right)+18\sqrt{n}r_n+O\left(\dfrac{1}{\sqrt{n}}\right)
		\\=&\,\sqrt{n}\xi(f^{(n)}_{X,Y})+O(n^{-1/2})+O(\sqrt{n}b_n)\ind(\xj>0).
	\end{align*}
	This finishes the proof of part (i) of Theorem~\ref{th:xi_n_wass}. \qed
	
	\bigskip
	
	\noindent \emph{Part (ii).} 
	Let $\mu^*_n$ and $\mu_n$ be the laws of $W_n^*=(\xi^*_n-\EE \xi^*_n)\sqrt{\Var(\xi_n^*)}$ and $\sqrt{n}(\xi_n-\EE \xi_n)/\sqrt{2/5}$. By Theorem \ref{th:fstarvar}, we have the upper bound
	$$
	\mW(\mu^*_n,\mu_n)\le \mW_2(\mu^*_n,\mu_n)\lesssim n^{-1/2}+\left(\xj+\sqrt{n\log n}\,b_n\right)^{1/2}\ind(\xj>0),
	$$
	where $\mW_2$ is the Wasserstein-2 distance, and $\mW$ is the Wasserstein-1 distance defined in Definition~\ref{def:wass}. 
	
	Now if we define $\mu'_n$ to be the law of $\sqrt{n}(\xi_n-\xj)/\sqrt{2/5}$, then by part (i) of the theorem
	\begin{align*}
		\mW(\mu^*_n,\mu'_n)\le & \sqrt{n}\abs*{\EE\xi_n-\xj}+\mW(\mu^*_n,\mu_n)\\ 
		\lesssim &n^{-1/2}+\left(\xj+\sqrt{n\log n}\,b_n\right)^{1/2}\ind(\xj>0).
	\end{align*}
	Finally notice that for the standard normal law $\nu$,
	\begin{align*}
		\mD\left(\dfrac{\sqrt{n}(\xi_n-\xj)}{\sqrt{2/5}}\right)
		=&\mW(\mu'_n,\,\nu)\le \mW(\mu'_n,\mu^*_n)+\mW(\mu^*_n,\,\nu)\\
		=& \mW(\mu_n^*,\mu_n')+\mD(W^*_n) \\
		\lesssim & \left(\xj+\sqrt{n\log n}\,b_n\right)^{1/2}\ind(\xj>0)+n^{-1/2},
	\end{align*}
	where we use \cref{th:oracle_wass} in the last step. \qed

	\subsection{Proofs from Section~\ref{sec:applications}}

	\subsubsection{Proof of Proposition~\ref{pr:xi_mix}}\label{pf:xi_mix}
	Let us observe that using equation~\eqref{eq:contam_final}, the marginal densities of $X$ and $Y$ under $\jtdt(\cdot,\cdot)$ are $f_X(\cdot)$ and $f_Y(\cdot)$ respectively for any $r_n$. Let us define, \begin{equation}\label{eq:xi_jtd1} F_{Y|X;r_n}(y|x):=\int\limits_{-\infty}^{y}\frac{\jtdt(t,x)}{f_X(x)}\,dt
	\end{equation}
	and observe that
	\begin{equation}
		\label{eq:xi_jtd}
		\xi(\jtdt)=6\,\Bigg[\int\limits_{-\infty}^{\infty}\mathbb{E}_X[F^2_{Y|X;r_n}(t|X)]\,f_Y(t)dt-\frac{1}{3}\Bigg].
	\end{equation}
	Let $F_Y(\cdot)$ be the CDF of $Y$. Next, note that~\eqref{eq:xi_jtd1} implies:
	\[
	F_{Y|X;r_n}(t|x) = (1-r_n)F_Y(t)+r_n\int\limits_{-\infty}^{t}\frac{g_{X,Y}(x,w)}{f_X(x)}\,dw.
	\]
	Let $g_X(\cdot)$ be the marginal density of $X$ under $\cde(\cdot,\cdot)$. By equation~\eqref{eq:contam_final}, $g_X(\cdot)=f_X(\cdot)$. Since
	$
	\int\limits_{-\infty}^{\infty}F^2_Y(t)f_Y(t)\,dt = \frac{1}{3},
	$
	we have that
	\begin{align*}
		\int\limits_{-\infty}^{\infty}\mathbb{E}_X[F^2_{Y|X;r_n}(t|X)]\,f_Y(t)dt &= (1-r_n)^2\frac{1}{3}+2r_n(1-r_n)\int\limits_{-\infty}^{\infty}F_Y(t)\mathbb{E}_X\left[\int\limits_{-\infty}^{t}\frac{g_{X,Y}(X,w)}{f_X(X)}\,dw\right]f_Y(t)\,dt\\
		&\quad\quad+ r^2_n \int\limits_{-\infty}^{\infty}\mathbb{E}_X\left[\int\limits_{-\infty}^{t}\frac{g_{X,Y}(X,w)}{f_X(X)}\,dw\right]^2f_Y(t)\,dt\\
		& = \frac{1}{3}-\frac{1}{3}r^2_n + r^2_n \int\limits_{-\infty}^{\infty}\mathbb{E}_X\left[\int\limits_{-\infty}^{t}\frac{g_{X,Y}(X,w)}{f_X(X)}\,dw\right]^2f_Y(t)\,dt.\\
		& =\frac{1}{3}+\frac{r_n^2}{6}\left(-2+\int\limits_{-\infty}^{\infty}\mathbb{E}_X\left[\int\limits_{-\infty}^{t}\frac{g_{X,Y}(X,w)}{g_X(X)}\,dw\right]^2f_Y(t)\,dt\right)\\ &=\frac{1}{3}+\frac{r_n^2\xi(\cde)}{6}.
	\end{align*}
	Plugging the above display in \eqref{eq:xi_jtd} completes the proof.
	\qed
	
	\subsubsection{Proof of Proposition~\ref{pr:xi_reg}}\label{pf:xi_reg}
	Let $(X',Y',Z')\overset{d}{=}(X,Y,Z)$ where $(X',Y',Z')$ is independent of $(X,Y,Z)$ and $(X,Y,Z)$ be drawn according the distribution defined by equation~\eqref{eq:reg_mod} with $\sigma\equiv \sigma_n$. 
	
	Let us observe that
	\begin{align}\label{eq:xi_reg1}
		\xi(\jtds)&=-2+6\E_{Y'}\E_{X}\left[\P(Y\leq Y'|X,Y')\right]^2\nonumber \\ &=-2+6\E_{X',Z'}\E_X\left[\P(Y\leq g(X')+\sigma_n Z'|X,X',Z')\right]^2.
	\end{align}
	
	Under equation~\eqref{eq:reg_mod}, it is easy to check that
	\begin{equation}\label{eq:xi_reg2}
		\P^2(Y\leq g(X')+\sigma_n Z'|X,X',Z')=\Phi^2\left(Z'+\frac{g(X')-g(X)}{\sigma_n}\right),
	\end{equation}
	where $\Phi(\cdot)$ is the CDF of $N(0,1)$. All the derivatives of $\Phi^2(\cdot)$ are uniformly bounded. Using this fact with a standard Taylor series expansion, we get that, for any $x,x',z'$, we get:
	\begin{multline}
		\Phi^2(z' + \sigma_n^{-1}(g(x^\prime)-g(x))) = \Phi^2(z) + 2 \sigma_n^{-1} \Phi(z)\phi(z)(g(x^\prime)-g(x)) + {\sigma_n^{-2}\over 2} \Big[2\phi^2(z)(g(x^\prime)-g(x))^2\\
		+2\Phi(z)\phi^\prime(z)(g(x^\prime)-g(x))^2\Big]
		+ O(\sigma_n^{-3}|g(x^\prime)-g(x)|^3),
	\end{multline}
	where $\phi(\cdot)$ is the density of the $N(0,1)$ distribution. By combining~\eqref{eq:xi_reg1}~and~\eqref{eq:xi_reg2}, we have,
	\begin{align*}
		\sqrt{n}\,\xi(\jtds) & = -2+6\sqrt{n}\E_{X',Z'}\E_{X}\Phi^2\left(Z'+\frac{g(X')-g(X)}{\sigma_n}\right)\\
		& = 6\sqrt{n}\,\sigma_n^{-2}\; \mathbb{E}_{X,X'}\Big\{g(X)-g(X^\prime)\Big\}^2\;\mathbb{E}_{Z'}\Big[\phi^2(Z')+\Phi(Z')\phi^\prime(Z')\Big] + O(\sqrt{n}\sigma_n^{-3})\\
		& = 6\times 2 \sqrt{n}\,\sigma_n^{-2} \,\mbox{Var}(g(X)) \times \frac{1}{4\sqrt{3}\pi} + O(\sqrt{n}\sigma_n^{-3})\\
		& = ({\sqrt{3n}/ \pi})\,\sigma_n^{-2} \mbox{Var}(g(X)) + O(\sqrt{n}\sigma_n^{-3}).
		\hspace{5.85cm}
		\qed
	\end{align*}
	\subsubsection{Proof of Proposition \ref{pr:xi_rot}}\label{pf:xi_rot}
	Let us begin by observing that the joint density of $(X,Y)$ is given by
	\[
	\jtdd(x,y) = \frac{1}{(1-\Dl_n^2)}\;f_1\left(\frac{x-\Dl_n y}{1-\Dl_n^2}\right)f_2\left(\frac{y-\Dl_n x}{1-\Dl_n^2}\right).
	\]
	Let $(X',Y')\overset{d}{=}(X,Y)$ where $(X',Y')$ is independent of $(X,Y)$ and $(X,Y)$ is drawn according to~\eqref{eq:rotation} with $\Dl\equiv \Dl_n$.  Notice that
	\begin{equation}\label{eq:altxi}
		\xi(\jtdd)/6= \E_{Y'\sim Y}\E_X\left[\P(Y\leq Y'|X,Y')-P(Y\leq Y'|Y')\right]^2
	\end{equation}
	Under the rotation model \eqref{eq:rotation} for any $t$, 
	\begin{align*}
		\PP(Y\le t|X)-\PP(Y\le t)=:\dfrac{\int\limits_{-\infty}^tA(x,y,\Delta_n)dy}{\int\limits_{-\infty}^\infty A(x,y,\Delta_n)dy}-\frac{1}{1-\Delta_n^2}\int\limits_{-\infty}^t\int\limits_{-\infty}^\infty A(x,y,\Delta_n)dx dy\numberthis\label{eq:tlr_prp}
	\end{align*}
	where
	$
	A(x,y,\Delta_n)=f_1\left(\frac{x-\Delta_n y}{1-\Delta_n^2}\right)f_2\left(\frac{y-\Delta_n x}{1-\Delta_n^2}\right)=:f_1(h_1(x,y,\Delta_n))f_2(h_2(x,y,\Delta_n))\,
	$.
	Without loss of generality, we can assume $|\Dl_n|\leq\epsilon$ where $\epsilon$ is as specified in Assumption (A3), part 3 from Section~\ref{sec:rot-alt}. Next, by Taylor expansion around $0$, we have
	\begin{align*}
		f_1\left(\frac{x-\Delta_n y}{1-\Delta_n^2}\right)=&f_1(x)-yf_1'(x)\Delta_n
		+\left(y^2\cdot \frac{\partial^2}{\partial\theta^2}f_1(h_1(x,y,\theta))\big|_{\theta=\xi} +2x\frac{\partial}{\partial\theta}f_1(h_1(x,y,\theta))\big|_{\theta=\xi}\right)\Delta_n^2\\
		f_2\left(\frac{y-\Delta_n x}{1-\Delta_n^2}\right)=&f_2(y)-xf_2'(y)\Delta_n
		+\left(x^2\cdot \frac{\partial^2}{\partial\theta^2}f_2(h_2(x,y,\theta))\big|_{\theta=\xi'} +2y\frac{\partial}{\partial\theta}f_2(h_2(x,y,\theta))\big|_{\theta=\xi'}\right)\Delta_n^2
	\end{align*}
	for some $\xi,\xi'\in[-\epsilon,\epsilon]$. Multiplying the last two equations, we have
	\begin{align*}
		A(x,y,\Delta_n)=:f_1(x)f_2(y)-\Delta_n(xf_1(x)f_2'(y)+yf_1'(x)f_2(y))+\Delta_n^2g(x,y,\Delta_n),
	\end{align*}
	which implies (since $\int xf_1(x)dx=\int yf_2(y)dy=0$)
	\begin{align*}
		\int\limits_{-\infty}^tA(x,y,\Delta_n)dy=&f_1(x)F_2(t)-\Delta_n[xf_1(x)f_2(t)+f_1'(x)J(t)]+\Delta_n^2\int\limits_{-\infty}^tg(x,y,\Delta_n)dy,\\
		\int\limits_{-\infty}^\infty A(x,y,\Delta_n)dy=&f_1(x)+\Delta_n^2\int\limits_{-\infty}^\infty g(x,y,\Delta_n)dy,\\
		\int\limits_{-\infty}^t\int\limits_{-\infty}^\infty A(x,y,\Delta_n)dx dy=& F_2(t)+\Delta_n^2\int\limits_{-\infty}^t\int\limits_{-\infty}^\infty g(x,y,\Delta_n)dx dy.\numberthis\label{eq:tlrdisp}
	\end{align*}	
	We define $G(x,t,\Delta_n):=\int\limits_{-\infty}^tg(x,y,\Delta_n)dy$. We also observe that the marginal density of $X$, say $f_{X,\Delta_n}(\cdot)$, satisfies the following:
	\begin{align*}
		f_{X,\Delta_n}(x)=\dfrac{1}{1-\Delta_n^2}\int\limits_{-\infty}^\infty A(x,y,\Delta_n)dy 
	\end{align*}
	Then by \eqref{eq:altxi}, \eqref{eq:tlr_prp} and \eqref{eq:tlrdisp} we now have
	\begin{align*}
		\xj/6=&\int\int \dfrac{\left[-\Delta_n(xf_1(x)f_2(t)+f_1'(x)J(t))+\Delta_n^2g_2(x,t,\Delta_n)\right]^2}{f_1(x)+\Delta_n^2G(x,\infty,\Delta_n)}dxf_{Y,\Delta_n}(t)dt\\
	\end{align*}
	where $f_{Y,\Dl_n}(\cdot)$ is the marginal density of $Y$ and  
	\begin{align*}
		g_2(x,t,\Delta_n):=&G(x,t,\Delta_n)-f_1(x)F_2(t)\int G(x,t,\Delta_n)dx-F_2(t)G(x,t,\Delta_n)\\
		&+\Delta_n^2F_2(t)G(x,t,\Delta_n)\int G(x,t,\Delta_n)dx.
	\end{align*}
	Next, let us observe that $G(x,t,\Delta_n)$ can be written as a sum of $\left(\dfrac{\partial^k}{\partial u^k}f_1(u)|u|^l\right)$ for $k\in\{0,1,2\}$ and $\ell\in\{0,1,2,3\}$. It can be checked using Assumption (A3) from Section \ref{sec:rot-alt}, that this implies $\int g_2(x,t,\Delta_n)^2/f_1(x)dx\lesssim 1.$ 
	
	We now expand the square above (and use the Cauchy-Schwarz inequality for the cross terms) to obtain
	\begin{align*}
		\xj/6=\Delta_n^2\int\int \dfrac{(xf_1(x)f_2(t)+f_1'(x)J(t))}{f_1(x)}dxf_{Y,\Delta_n}(t)dt+O(\Delta_n^3).
	\end{align*}
	Moreover, the marginal $f_{Y,\Delta_n}(y)$ under the alternative can be written, by a similar Taylor expansion, as 
	\begin{align*}
		f_{Y,\Delta_n}(y)=\frac{1}{1-\Delta_n^2}\int\limits_{-\infty}^\infty A(x,y,\Delta_n) dx=f_2(y)+\Delta_n^2\int\limits_{-\infty}^\infty g(x,y,\Delta_n)dx.
	\end{align*}
	Consequently,
	\begin{align*}
		&\xj/6=\Delta_n^2\int\int \dfrac{(xf_1(x)f_2(t)+f_1'(x)J(t))^2}{f_1(x)}dxf_{Y,\Delta_n}(t)dt+O(\Delta_n^3)\\
		=&\Delta_n^2\int\int \dfrac{x^2f_1(x)^2f_2(t)^2+(f_1'(x))^2(J(t))^2+2xf_1(x)f_2(t)f_1'(x)J(t)}{f_1(x)}dxf_{2}(t)dt+O(\Delta_n^3).\numberthis\label{eq:xjfin}
	\end{align*}
	It is not hard to see that
	\begin{align*}
		\int\int \frac{x^2f_1(x)^2f_2(t)^2}{f_1(x)}dxf_2(t)dt=&(\EE U^2)\EE(f_2(V)^2)=\EE(f_2(V)^2), \\
		\int\int \frac{(f_1'(x))^2J^2(t)}{f_1(x)}dx f_2(t)dt=&I(f_1)\EE(J^2(V)), \\
	\end{align*}
	and finally
	\begin{align*}
		\int\limits_{-\infty}^\infty\int\limits_{-\infty}^\infty xf_1'(x)f_2(t)J(t)dx f_2(t)dt=&\int\limits_{-\infty}^\infty J(t)f_2^2(t)dt\times \int\limits_{-\infty}^\infty xf_1'(x)dx
		%=&\int\limits_{-\infty}^\infty J(t)f_2^2(t)dt\times \left([xf_1(x)]_{-\infty}^\infty -\int\limits_{-\infty}^\infty f_1(x)dx \right)\\
		%=&-\int\limits_{-\infty}^\infty  J(t)f_2^2(t)dt
		=-\EE(J(V)f_2(V)).
	\end{align*}
	Plugging these back into \eqref{eq:xjfin} finishes the proof.\qed
	
	%\subsection*{Supplementary Material}
	%The proofs of results in Section~\ref{sec:degtest}, as well as those of some technical intermediate steps and lemmas are given in the supplement~\cite{ADN21-supp}.
	% in Appendix \ref{sec:proofs} require a number of technical results. These 
	%\begin{supplement}
	%\stitle{ }
	%\sdescription{The proofs of results in Section~\ref{sec:degtest}, as well as those of some technical intermediate steps and lemmas are given in the supplement~\cite{ADN21-supp}.}
	%\end{supplement}
	
	%\medskip

	Next, in Appendix~\ref{sec:proofs-add}, we prove the theorems as stated in Section 4 above, and state some technical results required in the process. In Appendix~\ref{sec:thmA1A2}, we prove Theorems A1 and A2 from Appendix A. Finally in Appendix~\ref{sec:lemmas}, we present the proofs of the 
	technical lemmas used to prove the theorems in the Appendix A as well as those in Sections~\ref{sec:proofs-add} and~\ref{sec:thmA1A2}. 
	
	\section{Proofs from Section 4}\label{sec:proofs-add}
	\subsection{Proof of Proposition 4.1}% in the main paper~\cite{ADN21}}
	
	Let $(X_2,Y_2)$ be generated independently of $(X_1,Y_1)$ and with the same distribution. Note that it suffices to show that
	$$\mathcal{J}:=\E_{(X_1,Y_1),\ Y_2} \left[\P(Y_1\leq Y_2|X_1,\ Y_2)\right]^2$$
	is a strictly increasing function of $|\rho|$, $r$ and $\sigma^{-1}$ in parts (1), (2) and (3) respectively.
	
	\vspace{0.1in} 
	
	\emph{Part (1).} Clearly, by replacing $X_1$, $Y_1$, and $Y_2$ by $(X_1-\mu_1)/\sigma_1$, $(Y_1-\mu_2)/\sigma_2$, and $(Y_2-\mu_2)/\sigma_2$, $\mathcal{J}$ does not change. Consequently we can assume without loss of generality $\mu_1=\mu_2=0$ and $\sigma_1=\sigma_2=1$. In the sequel, we will use $\Phi(\cdot)$ and $\phi(\cdot)$ to denote the probability distribution function and the probability density function of the standard normal distribution. With this in view, note that
	
	\begin{align}\label{eq:ximon1}
		\mathcal{J}&=\int_x \int_t \Phi^2\left(\frac{t-\rho x}{\sqrt{1-\rho^2}}\right)\phi(t)\phi(x)\,dt \,dx\nonumber \\&=\sqrt{1-\rho^2}\int_x \int_z \Phi^2(z)\phi\left(\rho x + \sqrt{1-\rho^2} z\right)\phi(x)\,dz \,dx \quad \quad \quad \left[\mbox{Put}\ z=\frac{t-\rho x}{\sqrt{1-\rho^2}}\right] \nonumber \\&=\frac{\sqrt{1-\rho^2}}{2\pi}\int_x \int_z \Phi^2(z)\exp\left(-\frac{1}{2}\left(\rho^2 x^2+z^2(1-\rho^2)+2zx\rho\sqrt{1-\rho^2}\right)\right)\exp\left(-\frac{x^2}{2}\right)\,dz\,dx \nonumber \\ &=\frac{\sqrt{1-\rho^2}}{2\pi}\int_z \Phi^2(z)\exp\left(-\frac{z^2}{2}\cdot \frac{1-\rho^2}{1+\rho^2}\right)\int_x \exp\left(-\frac{1+\rho^2}{2}\left(x+\frac{2\rho\sqrt{1-\rho^2}}{\sqrt{1+\rho^2}}\right)^2\right)\,dx\,dz \nonumber \\&=\frac{1}{\sqrt{2\pi}\rho_0}\int \Phi^2(z)\exp\left(-\frac{z^2}{2\rho_0^2}\right)\,dz \quad \quad \quad \quad \quad \quad \quad \quad \quad  \left[\mbox{Define}\ \rho_0:=\frac{1+\rho^2}{1-\rho^2}\right]\nonumber \\&=\int_{z>0} \left(\Phi^2(\rho_0 z)+\Phi^2(-\rho_0 z)\right)\phi(z)\,dz.
	\end{align}
	By differentiating the above with respect to $\rho_0$, we get from~\eqref{eq:ximon1} that:
	$$\frac{d}{d\rho_0}\mathcal{J}=\int_{z>0} z(2\Phi(\rho_0 z)-1)\phi(\rho_0 z)\rho(z)\,dz>0,$$
	which implies that $\mathcal{J}$ is a strictly increasing function of $\rho_0$ which in turn, is a strictly increasing function of $|\rho|$, thereby completing the proof.
	
	\vspace{0.1in} 
	
	\emph{Part (2).} Note that $Y_1,Y_2\overset{i.i.d.}{\sim}f_Y(\cdot)$ and $X_1\sim f_X(\cdot)$ for all $r\in [0,1]$. Here, $F_Y(\cdot)$ and $f_X(\cdot)$ are probability densities, and we will write $F_Y(\cdot)$ and $F_X(\cdot)$ to denote the corresponding distribution functions. We will write $g_{Y|X}(\cdot)$ to denote the conditional density of $Y|X$ under $g_{X,Y}(\cdot)$.  Therefore, 
	\begin{align*}
		\mathcal{J}&=\int_x \int_t \left(\int_{-\infty}^t \big((1-r)f_Y(y)+r g_{Y|X=x}(y)\big)\,dy\right)^2 f_Y(t) f_X(x)\,dt\,dx\\ &=\int_t F_Y^2(t)f_Y(t)+2r\int_t \int_{-\infty}^t \left(\int_x (g_{Y|X=x}(y)-f_Y(y))f_X(x)\,dx\right)f_Y(t)\,dy\,dt\\ &+r^2\int_x\int_t\left(\int_{-\infty}^t (g_{Y|X=x}(y)-f_Y(y))\,dy\right)^2 f_Y(t) f_X(x)\,dt\,dx\\ &=\frac{1}{3}+r^2\int_x\int_t\left(\int_{-\infty}^t (g_{Y|X=x}(y)-f_Y(y))\,dy\right)^2 f_Y(t) f_X(x)\,dt\,dx,
	\end{align*}
	which is clearly a strictly increasing function of $r$.
	
	\vspace{0.1in}
	
	\emph{Part (3).} We write $Y_2=g(X_2)+\sigma Z_2$, $Z_2\sim\mathcal{N}(0,1)$. Further, let $\tilde{F}_X(\cdot)$ be the probability distribution function of the random variable $g(X_2)-g(X_1)$. Observe that $\tilde{F}_X(\cdot)$ is symmetric around $0$, in the sense that $\tilde{F}_X(t)=1-\tilde{F}_X(-t)$ for all $t$. By simple computations, we then have:
	\begin{align*}
		\mathcal{J}&=\E\Phi^2\left(Z_2+\frac{g(X_2)-g(X_1)}{\sigma}\right)\\ &=\int_z \int_x \Phi^2\left(z+\frac{x}{\sigma}\right)\,\phi(z)\,dz \,d \tilde{F}_X(x)
	\end{align*}
	Let us define $\sigma_0:=\sigma^{-1}$ and take derivative of $\mathcal{J}$ with respect to $\sigma_0$, to get:
	\begin{align}\label{eq:ximon31}
		\frac{d}{d\sigma_0}\mathcal{J}&=2\int_z\int_{x} \Phi(z+\sigma_0 x)\phi(z+\sigma_0 x)\phi(z)x\,dz\,d\tilde{F}_X(x)\nonumber \\ &=2\int_z\int_{x>0} \Phi(z)\phi(z)(x \phi(z-\sigma_0 x)-x \phi(z+\sigma_0 x))\,dz \,d\tilde{F}_X(x)\nonumber \\ &=2\int_{z>0}\int_{x>0} (\Phi(z)-\Phi(-z))\phi(z)x(\phi(z-\sigma_0 x)-\phi(z+\sigma_0 x))\,dz \,d\tilde{F}_X(x). 
	\end{align}
	
	Note that for $z>0$, we have
	$$\Phi(z)>\Phi(-z)$$
	and for $x,z>0$, we have:
	$$\phi(z-\sigma_0 x)-\phi(z+\sigma_0 x)=\exp\left(-\frac{1}{2}\left(z^2+\sigma_0^2 x^2\right)\right)\left(\exp(\sigma_0 zx)-\exp(-\sigma_0 zx)\right)>0.$$
	Combining the two observations above, with~\eqref{eq:ximon31}, we get:
	$$\frac{d}{d\sigma_0}\mathcal{J}>0$$
	which implies $\mathcal{J}$ is a strictly increasing function of $\sigma_0$ and consequently a strictly decreasing function of $\sigma$.
	
	\subsection{Proof of Theorem 4.1}% in the main paper~\cite{ADN21}}
	
	\emph{Part (i).} Consider $(X_1',Y_1'),\ldots ,(X_n',Y_n')\overset{i.i.d.}{\sim}f_{X,Y}^{(n)}$. Let $\xi_n^{i}$ be Chatterjee's correlation coefficient defined with $(X_i,Y_i)$ replaced by $(X_i',Y_i')$. By the same argument as in~\cite[Lemma 9.11]{sc_corr}, $|\xi_n-\xi_n^i|\leq 6n^{-1}$. We apply the bounded differences inequality~\cite{Mcdiarmid1989} to get:
	$$\P(|\xi_n-\E\xi_n|\geq t)\leq 2\exp\left(-\frac{nt^2}{18}\right).$$
	Also by Theorem 2.1, there exists $N_0$ depending only on $\Gamma=(C,\eta,\theta,\gamma)$ from Assumptions (A1) and (A2) in the main paper, such that
	$$\sqrt{n}|\E\xi_n-\xj|\leq \frac{1}{2}$$
	for all $n\geq N_0$. Define $K:= 1+5\sqrt{\log{(2/\alpha)}}$. Combining the two displays above, under $\mathrm{H}_0$, we get:
	$$\P_{\mathrm{H}_0}(\sqrt{n}|\xi_n-\xi_0|\geq K)\leq \P_{\mathrm{H}_0}(\sqrt{n}|\xi_n-\E_{\mathrm{H}_0}\xi_n|\geq 5\sqrt{\log{(2/\alpha)}})\leq \alpha$$
	for all $n\geq N_0$. 
	
	Next note that, by the triangle inequality, for all $f_{X,Y}^{(n)}\in \hloc$, we have:
	$$\sqrt{n}|\xi_n-\xi_0|\geq \sqrt{n}c_n-\sqrt{n}|\xi_n-\E_{\mathrm{H}_{1,n}}\xi_n|-\sqrt{n}|\E_{\mathrm{H}_{1,n}}\xi_n-\xi_0|.$$
	As $\sqrt{n}c_n\to\infty$ and the other two terms are $O_p(1)$ and $O(1)$ from the preceding displays, we have 
	$$
	\beta_{\phi_n}(\hloc)\to 1\qquad \mbox{as}\qquad n\to\infty.
	$$
	
	\vspace{0.1in}
	
	\emph{Part (ii).} The proof of this result will use Le Cam's two-point method; see~\cite[Chapter 2]{Tsybakov2009}. Towards this direction, let $f_{X,Y}(\cdot,\cdot)$ be a joint distribution on $\R^2$, with marginals $f_X(\cdot)$ and $f_Y(\cdot)$ such that the following conditions hold:
	\begin{itemize}
		\item $\xi(f_{X,Y})=\xi_0$.
		\item $f_{X,Y}(\cdot,\cdot)$ satisfies Assumptions (A1) and (A2) in the main paper with paramers $C,\eta,\theta,\gamma$ given in the problem statement.
		\item $f_{X,Y}(\cdot,\cdot)$ is compactly supported on $[-1,1]^2$ and is uniformly upper and lower bounded on $[-1,1]^2$.
	\end{itemize}
	This can be easily ensured by choosing $f_{X,Y}(\cdot,\cdot)$ to be a truncated bivariate Gaussian with appropriate parameters depending on $C,\eta,\theta,\gamma$. Fix $\{r_n\}_{n\geq 1}$ and define
	$$f_{X,Y}^{(n)}(\cdot,\cdot)=(1-r_n)f_{X,Y}(\cdot,\cdot)+r_n f_X(\cdot)f_Y(\cdot),$$
	where $r_n\in (0,1)$. Note that $f_{X,Y}^{(n)}$ also satisfies Assumptions (A1) and (A2) with the same parameters. It then suffices to show the following:
	\begin{itemize}
		\item[(1).] For all $n$ large enough and some $c_1>0$, we have $\Big|\xi(f_{X,Y}^{(n)})-\xi_0\Big|\geq c_1 r_n$.
		\item[(2).] For all $n$ large enough and some $c_2>0$, we have $\mathrm{KL}\Big(\otimes_n f_{X,Y}^{(n)}|| \otimes_n f_{X,Y}\Big)\leq c_2 n r_n^2$. Here $KL(p||q)$ denotes the Kullback-Leibler divergence between probability measures $p$ and $q$, and $\otimes_n$ denotes the $n$-fold product measure.
	\end{itemize}
	
	\emph{Proof of 1.} Let $F^{(n)}_{Y|X}(\cdot|X)$ and $F_{Y|X}(\cdot|X)$ denote the conditional distribution functions of $Y|X$ under $f_{X,Y}^{(n)}$ and $f_{X,Y}$ respectively. Also, let $F_Y(\cdot)$ be the distribution function of $Y$. Then the following holds:
	\begin{align}\label{eq:lowerbd1}
		&~\xi(f_{X,Y}^{(n)})\\
		=&~6\int \E[F^{(n)}_{Y|X}(t|X)]^2 f_Y(t)\,dt-2\nonumber\\ 
		=&~6\int \E[F_{Y|X}(t|X)+r_n(F_Y(t)-F_{Y|X}(t|X))]^2 f_Y(t)\,dt-2\nonumber\\ =&~\xi_0+2r_n\underbrace{\int \left(\E[F_Y(t)]^2-\E[F_{Y|X}(t|X)]^2\right)f_Y(t)\,dt}_{D}+r_n^2\underbrace{\int \E\left(F_Y(t)-F_{Y|X}(t|X)\right)^2f_Y(t)\,dt}_{E}.
	\end{align}
	Note that, by the conditional version of Jensen's inequality
	$$\E[F_Y(t)]^2\leq \E[F_{Y|X}(t|X)]^2\qquad \mbox{for}\ \mbox{all}\ t\in\R.$$
	As $\xi_0>0$, this implies $X$ and $Y$ are not independent (see~\cite[Theorem 1.1]{sc_corr}) and consequently, the above display implies $D<0$. For the same reason $E>0$. By replacing $r_n$ by $c r_n$ for a small constant $c$ if necessary, we can ensure
	$$r_n\leq \frac{|D|}{E}.$$
	Combining the above display with~\eqref{eq:lowerbd1}, we get
	$$\big|\xi(f_{X,Y}^{(n)})-\xi_0|\geq 2r_n|D|-r_n^2 E \geq r_n |D|.$$
	This proves 1.
	
	\vspace{0.1in}
	
	\emph{Proof of 2.} As $\mathrm{KL}(\otimes_n p||\otimes_n q)= n \mathrm{KL}(p||q)$, it suffices to show that $\mathrm{KL}(f_{X,Y}^{(n)}||f_{X,Y})\leq r_n^2$. Towards this direction, note that
	\begin{align*}
		&\;\;\;\;\mathrm{KL}(f_{X,Y}^{(n)}||f_{X,Y})\\ &=\int \left[f_{X,Y}(x,y)+r_n(f_X(x)f_Y(y)-f_{X,Y}(x,y))\right]\log{\frac{f_{X,Y}(x,y)+r_n(f_X(x)f_Y(y)-f_{X,Y}(x,y))}{f_{X,Y}(x,y)}}\,dx\,dy\\ &\overset{(a)}{=}\int \left[f_{X,Y}(x,y)+r_n(f_X(x)f_Y(y)-f_{X,Y}(x,y))\right]\left(r_n\frac{f_X(x)f_Y(y)-f_{X,Y}(x,y)}{f_{X,Y}(x,y)}\right)+\mathcal{O}(r_n^2)\\ &\leq c_2 r_n^2,
	\end{align*}
	where (a) follows from a Taylor Series expansion and $c_2$ depends on the parameters of $f_{X,Y}(\cdot,\cdot)$. This proves 2.
	
	\section{Proofs of Theorems~A.1 and~A.2}\label{sec:thmA1A2}
	\subsection{Proof of Theorem A.1}\label{sec:pfstarvar}We consider a triangular array $(X_{n,1},Y_{n,1}),\dots,(X_{n,n},Y_{n,n})$ from a bivariate density $\jtd^{(n)}(\cdot,\cdot)$. We shall show that 
	\begin{align*}
		\max&\left\{\abs*{n\Var(\xi_n^*)-\frac{2}{5}},\,
		\abs*{n\Var(\xi_n)-\frac{2}{5}},\,
		\abs*{n\Cov(\xi_n^*,\xi_n)-\frac{2}{5}}\right\}\\
		\lesssim &\dfrac{1}{n}+\xi(f^{(n)}_{X,Y})+\sqrt{n\log n}\,b_n\ind(\xi(f^{(n)}_{X,Y}) > 0)\numberthis\label{eq:xi_varcovs}
	\end{align*}
	for $b_n:=n^{-\frac{\gamma}{\gamma+1}}(\log{n})^2+\left(\dfrac{(\log{n})^2}{n}\right)^{\left(\frac{\gamma(\theta-1)}{\theta(\gamma+1)}\wedge \frac{\eta \gamma}{\gamma+1}\right)}$.
	
	By the standard Glivenko-Cantelli Theorem, we know that $\hat{F}_n(\cdot)$ and $F^{(n)}_Y(\cdot)$ are ``close" almost surely in the $L^{\infty}$ norm. This motivates the definition of an oracle version of $\xi_n$ (see Section 2 in the main paper) as follows:
	\begin{equation}\label{eq:proj_defn} 
		\xi_n^*:=\dfrac{3n}{n^2-1}\left(\dfrac{1}{n}\underset{i\neq j}{\sum\sum }\abs*{\F(Y_{n,i})-\F(Y_{n,j})}-\sum_{i=1}^n\abs*{F^{(n)}_Y(Y_{n,(i+1)})-F^{(n)}_Y(Y_{n,(i)})}\right).
	\end{equation}
	Intuitively, of course, $\xi_n^*$ is mathematically more tractable than $\xi_n$ as it replaces the random function $\hat{F}_n(\cdot)$ by the deterministic function $F^{(n)}_Y(\cdot)$. 
	
	Let $N(i)\in \{j\in[n]:\mathrm{Rank}(X_{n,j})=\mathrm{Rank}(X_{n,i})+1\}$ be the unique index $j$ such that $X_{n,j}$ is immediately to the right of $X_{n,i}$ when $X_{n,i}$'s are arranged in increasing order. If there are no such indices for some $i$, set the corresponding $N(i)=1$. Let us define
	\begin{align*}\label{eq:xiprime}
		\xi'_n:=&\dfrac{6}{n^2-1}\disp \sum_{i=1}^{n-1}\min\{R_{i}, R_{N(i)}\}, 		\\
		\xi^{*\prime }_n:=&\dfrac{6n}{n^2-1}\left(\sum_{i=1}^n\min\{\F(Y_{n,i}),\F(Y_{n,N(i)})\}-\dfrac{1}{n}\underset{i\neq j}{\sum\sum }\min\{\F(Y_{n,i}),\F(Y_{n,j})\}\right). \numberthis
	\end{align*}
	Here and in the rest of the supplement, we remove the subscript $n$ in $R_{n,i}$ and write $R_i$ for notational convenience.
	
	Since $\{R_i\}$ and $\{R_{N(i)}\}$ form two permutations of $[n]$, we have that $\sum R_i=\sum R_{N(i)}=n(n-1)/2$. Then using the simple identity $|a-b|=a+b-2\min\{a,b\}$, one can check that
	$$n\E(\xi_n^*-\xi^{*\prime}_n)^2\lesssim n^{-1}\quad \text{and}\quad n\E(\xi_n-\xi'_n)^2\lesssim n^{-1}.$$
	Therefore it suffices to prove \eqref{eq:xi_varcovs} with $\xi'_n$ and $\xi^{\stp}_n$ instead of $\xi_n$ and $\xi_n^*$ respectively. Define $\bX^{(n)}:=(X_{n,1},\ldots ,X_{n,n})$. Conditioning on $\bX^{(n)}$, we have
	$$\begin{aligned}
		\Var(\xi_n^{\stp})=\EE(\Var(\xi_n^{\stp}|\bX^{(n)}))+\Var(\EE(\xi_n^{\stp}|\bX^{(n)})),&\quad \Var(\xi_n')=\EE(\Var(\xi_n'|\bX^{(n)}))+\Var(\EE(\xi_n'|\bX^{(n)}))\\
		\Cov(\xi_n',\xi_n^{\stp})=\EE(\Cov(\xi_n',\xi_n^{\stp}|\bX^{(n)}))+&\Cov(\EE(\xi_n'|\bX^{(n)}),\EE(\xi_n^{\stp})|\bX^{(n)}).
	\end{aligned}
	$$
	
	We now decompose each of the six terms on the right hand side, starting with the three expectation of covariance terms. To analyze the first term, that is, $\EE(\Var(\xi_n^{\stp}|\bX^{(n)}))$, let us introduce the following notations.
	\begin{align*}
		T_1^*:=&\frac{1}{n}\sum_{i}\mathrm{Var}(\min\{\F(Y_{n,i}),\F(Y_{n,N(i)})\}|\bX^{(n)})\\
		T_2^*:=&\frac{1}{n}\sum_{i}\mbox{Cov}(\min\{\F(Y_{n,i}),\F(Y_{n,N(i)})\},\min\{\F(Y_{n,N(i)}),\F(Y_{n,N(N(i))})\}|\bX^{(n)})\\
		T_3^*:=&\frac{1}{n^3}\underset{(i,j,l)\ \text{distinct}}{\sum\sum\sum}\mbox{Cov}(\min\{\F(Y_{n,i}),\F(Y_{n,j})\},\min\{\F(Y_{n,i}),\F(Y_{n,l})\}|\bX^{(n)})\\
		T_4^*:=&\frac{1}{n^2}\underset{(i,j,N(i))\ \text{distinct}}{\sum\sum\sum}\mbox{Cov}(\min\{\F(Y_{n,i}),\F(Y_{n,j})\},\min\{\F(Y_{n,i}),\F(Y_{n,N(i)})\}|\bX^{(n)})\\
		T_5^*:=&\frac{1}{n^2}\underset{(i,j,N(i))\ \text{distinct}}{\sum\sum\sum}\mbox{Cov}(\min\{\F(Y_{n,N(i)}),\F(Y_{n,j})\},\min\{\F(Y_{n,i}),\F(Y_{n,N(i)})\}|\bX^{(n)})\\
		T_6^*:=&\frac{1}{n^3}\underset{(i,j,k,l)\ \text{distinct}}{\sum\sum\sum\sum}\mbox{Cov}(\min\{\F(Y_{n,i}),\F(Y_{n,j})\},\min\{\F(Y_{n,k}),\F(Y_{n,l})\}|\bX^{(n)})\\
		T_7^*:=&\frac{1}{n^2}\underset{(i,j,k,N(k))\ \text{distinct}}{\sum\sum\sum\sum}\mbox{Cov}(\min\{\F(Y_{n,i}),\F(Y_{n,j})\},\min\{\F(Y_{n,k}),\F(Y_{n,N(k)})\}|\bX^{(n)})\\
		T_8^*:=&\frac{1}{n^3}\underset{(i,j,N(i),N(j))\ \text{distinct}}{\sum\sum\sum\sum}\mbox{Cov}(\min\{\F(Y_{n,i}),\F(Y_{n,N(i)})\},\min\{\F(Y_{n,j}),\\ &\qquad\qquad\qquad\qquad\qquad\F(Y_{n,N(j)})\}|\bX^{(n)})
		\numberthis\label{eq:evxistr}	
	\end{align*}
	Let us define
	\begin{equation}\label{eq:def-Tj-ind}
		T_{j,{\rm ind}}^*:=
		\text{the value of }
		T_j^*
		\text{ when }
		X,\,Y
		\text{ are independent },
	\end{equation}
	and the corresponding deviation terms:
	\begin{equation}\label{eq:def-Qj}
		Q_j^*:=T_j^*-T_{j,\rm ind}^*.
	\end{equation}
	Through an explicit calculation of the $T_{j,{\rm ind}}$'s, it follows that
	\begin{align*}
		T_1^*= \dfrac{1}{18}+Q_1^*, & \quad T_2^*=\dfrac{1}{45}+Q_2^*,\\
		T_3^*=\dfrac{1}{45}+Q_3^*, & \quad T_4^*=\dfrac{1}{45}+Q_4^*,\\
		T_5^*=\dfrac{1}{45}+Q_5^*, & \quad T_6^*=0,\\
		T_7^*=0,  & \quad T_8^*=0,
	\end{align*}
	where the last three equations follow once again by the fact that conditional on $\bX^{(n)}$ the random variables $Y_{n,1},\dots,Y_{n,n}$ are independent. Now, we observe that 
	\begin{align}\label{eq:convar1}
		\frac{(n^2-1)^2}{36n^3}\mathrm{Var}(\xi^{\stp}_n|\bX^{(n)})=T_1^*+2T_2^*+4T_3^*-4T_4^*-4T_5^*+T_6^*-2T_7^*+T_8^*+O(n^{-1}).
	\end{align}
	Hence, to bound the absolute value of $\frac{(n^2-1)^2}{36n^3}\mathrm{Var}(\xi^{\stp}_n|\bX^{(n)})$, we need to bound the $Q^*_i$s. To bound the $Q_i^*$'s stochastically, we shall use the following lemma, which is proved in~\cref{sec:pfevxsterr}.
	\begin{lemma}\label{lem:evxsterr}
		$\max\{\EE|Q_1^*|\,,\dots,\,\EE|Q_5^*|\}\lesssim \,\, n^{-1}+\xj 
		+b_n\ind(\xj>0).$
	\end{lemma}
	Using~\cref{lem:evxsterr}, adding the equations in~\eqref{eq:evxistr} and taking expectation over $\bX^{(n)}$, we have
	\begin{align*}
		&\abs*{n\EE\Var(\xi_n^{\stp}|\bX^{(n)})-\dfrac{2}{5}}\\
		=&
		\abs*{36\,\EE(T_1^*+2T_2^*+4T_3^*-4T_4^*-4T_5^*+T_6^*-2T_7^*+T_8^*)-\dfrac{2}{5}} +O(n^{-1})\\
		=& 36\,\abs*{\EE(Q_1^*+2Q_2^*+4Q_3^*-4Q_4^*-4Q_5^*)}+O(n^{-1})\\
		\lesssim & n^{-1}+\xj +b_n\ind(\xj>0).	\numberthis\label{eq:evxistest}
	\end{align*}
	
	Moving on to $\EE(\Var(\xi_n^{\stp}|\bX^{(n)}))$ we notice that
	$$
	\min\{R_i,R_{N(i)}\}=\sum_{k=1}^n\ind(Y_{n,k}\le \min\{Y_{n,i},Y_{n,N(i)}\})=1+\sum_{k\neq i,N(i)}\ind(Y_{n,k}\le \min\{Y_{n,i},Y_{n,N(i)}\}).
	$$
	Let us now define the following notation.
	\begin{align*}
		T_1:=&\dfrac{n}{(n^2-1)^2}\sum_{i=1}^{n-1}\Var(\min\{R_i,R_{N(i)}|\bX^{(n)}\})	\\
		=&\PP(Y_{n,1}\le \min\{Y_{n,3},Y_{n,4}\},Y_{n,2}\le \min\{Y_{n,3},Y_{n,4}\})\\
		&-\PP(Y_{n,1}\le \min\{Y_{n,3},Y_{n,4}\})\PP(Y_{n,2}\le \min\{Y_{n,3},Y_{n,4}\})+Q_1	\\
		=&\frac{1}{18}+Q_1, \numberthis\label{eq:evxi1}\\
		T_2:=&\dfrac{n}{(n^2-1)^2}\sum_{i=1}^{n-1} \Cov(\min\{R_i,R_{N(i)}\},\,\min\{R_{N(i)},R_{N(N(i))}\}|\bX^{(n)}) \\
		=& \PP(Y_{n,1}\le\min\{Y_{n,3},Y_{n,4}\},Y_{n,2}\le\min\{Y_{n,4},Y_{n,5}\})\\
		&-\PP(Y_1\le \min\{Y_3,Y_4\})\PP(Y_2\le \min\{Y_4,Y_5\})+Q_2 \\
		=&\dfrac{1}{45}+Q_2, \numberthis\label{eq:evxi2}\\
		T_3:=& \dfrac{n}{(n^2-1)^2}\sum_{i}\sum_{j\neq i}\sum_{\underset{N(i),N(j)}{l\neq i, j,}}\mbox{Cov}(\ind(Y_{n,l}\le \min\{Y_{n,i},Y_{n,N(i)}\}),\ind(Y_{n,l}\le \min\{Y_{n,j},Y_{n,N(j)}\})|\bX^{(n)}) \\
		=& \PP(Y_{n,1}\le\min\{Y_{n,2},Y_{n,3}\},Y_{n,1}\le\min\{Y_{n,4},Y_{n,5}\})\\
		&-\PP(Y_{n,1}\le \min\{Y_{n,2},Y_{n,3}\})\PP(Y_{n,1}\le \min\{Y_{n,4},Y_{n,5}\})+Q_3 \\
		=& \dfrac{4}{45}+Q_3, \numberthis\label{eq:evxi3}\\
		T_4:=& \dfrac{n}{(n^2-1)^2}\sum_{i}\sum_{j\neq i}\sum_{l\neq  j,N(j)} \Cov(\ind(Y_{n,j}\le \min\{Y_{n,i},Y_{n,N(i)}\}),\ind(Y_{n,l}\le \min\{Y_{n,j},Y_{n,N(j)}\})|\bX^{(n)}) \\
		=& \PP(Y_{n,1}\le \min\{Y_{n,2},Y_{n,3}\},Y_{n,4}\le\min\{Y_{n,1},Y_{n,5}\})\\
		&-\PP(Y_{n,1}\le \min\{Y_{n,2},Y_{n,3}\})\PP(Y_{n,4}\le\min\{Y_{n,1},Y_{n,5}\})+Q_4, \\
		=& -\dfrac{2}{45}+Q_4.	\numberthis\label{eq:evxi4}	\\	
		T_5:= & \dfrac{n}{(n^2-1)^2}\\
		\times& \underset{i,j,k,l,N(i),N(j)\,\text{distinct}}{\sum_i\sum_j\sum_k\sum_l}\Cov(\ind(Y_{n,k}\le \min\{Y_{n,i},Y_{n,N(i)}\}),\ind(Y_{n,l}\le \min\{Y_{n,j},Y_{n,N(j)}\})|\bX^{(n)}) \\
		= &\, 0,		\numberthis\label{eq:evxi5}
	\end{align*}
	where $Q_j$'s are defined in a way similar to \eqref{eq:def-Tj-ind}-\eqref{eq:def-Qj}. More explicitly, let
	\begin{equation}\label{eq:def-Tj-ind-2}
		T_{j,{\rm ind}}
		\text{ be the value of }
		T_j
		\text{ when }
		X,\,Y
		\text{ are independent. }
	\end{equation}
	and
	\begin{equation}\label{eq:def-Qj-2}
		Q_j:=T_j-T_{j,\rm ind}.
	\end{equation}
	When conditioned on $\bX^{(n)}$, the variables $N(1),\ldots ,N(n)$ are measurable. Using this observation, some straightforward but tedious calculation then yields
	$$
	\dfrac{n}{36}\,\Var(\xi_n'|\bX^{(n)})=T_1+2T_2+T_3+4T_4+T_5+O(n^{-1}).
	$$ We have \eqref{eq:evxi5} because, conditioned on $\bX^{(n)}$, the random variables $Y_{n,1},\dots,Y_{n,n}$ are independent. To bound the $Q_i$'s stochastically, we shall use the following lemma, which is proved in~\cref{sec:pfevxierr}.
	\begin{lemma}\label{lem:evxierr}
		$
		\max\{\EE|Q_1|,\,\EE|Q_2|,\,\EE|Q_3|,\,\EE|Q_4|\}\lesssim \,\, n^{-1}+\xj 
		+b_n\ind(\xj>0).
		$
	\end{lemma}
	Using~\cref{lem:evxierr}, adding equations \eqref{eq:evxi1}-\eqref{eq:evxi5} and taking expectation over $\bX^{(n)}$ we have
	\begin{align*}
		\abs*{n\EE\Var(\xi_n'|\bX^{(n)})-\dfrac{2}{5}}=&
		\abs*{36\,\EE(T_1+2T_2+T_3+4T_4+T_5)-\dfrac{2}{5}} +O(n^{-1})\\
		=& 36\,\abs*{\EE(Q_1+2Q_2+Q_3+4Q_4)}		+O(n^{-1})\\
		\lesssim & n^{-1}+\xj +b_n\ind(\xj>0).	\numberthis\label{eq:evxiest}
	\end{align*}
	Similarly, we can also prove
	\begin{equation}\label{eq:ecov_est}
		\abs*{n\EE(\Cov(\xi_n',\xi_n^{\stp}|\bX^{(n)}))-\dfrac{2}{5}}\lesssim  n^{-1}+\xj +b_n\ind(\xj>0).
	\end{equation}
	We now move on to the `variance of expectation' terms. We begin with the following lemma which is proved in~\cref{sec:pfvarex}.
	\begin{lemma}\label{lem:varex}
		Suppose $\jtd^{(n)}(\cdot,\cdot)$ satisfies Assumptions (A1) and (A2) in the main paper. Then,
		$$\max\{\mathrm{Var}(\E[\sqrt{n}\xi'_n|\bX^{(n)}]),\,\mathrm{Var}(\E[\sqrt{n}\xi^{\stp}_n|\bX^{(n)}])\}\lesssim n^{-1}+\xj+\sqrt{n\log{n}}b_n\ind(\xj>0).$$
	\end{lemma}
	
	By the Cauchy-Schwarz inequality, the same upper bound holds for $\Cov(\EE(\xi'_n|\bX^{(n)}),\,\EE(\xi^{\stp}_n|\bX^{(n)}))$. Putting together equations \eqref{eq:evxiest}, \eqref{eq:evxistest}, \eqref{eq:ecov_est},  \cref{lem:varex} and finally \eqref{eq:xiprime}, the proof of \eqref{eq:xi_varcovs} is complete.
	\qed
	
	\medskip
	
	\subsection{Proof of Theorem A.2}\label{sec:pforacle_wass} Recall that $\bX^{(n)}=(X_{n,1},\ldots ,X_{n,n})$. Also, recall that $N(i)= \{j\in[n]:\mathrm{Rank}(X_{n,j})=\mathrm{Rank}(X_{n,i})+1\}$ is the unique index $j$ such that $X_{n,j}$ is immediately to the right of $X_{n,i}$ when the $X_{n,i}$'s are arranged in increasing order. If there are no such indices for some $i$, set the corresponding $N(i)=1$. We shall use the oracle statistic $\xi^*_n$ defined as
	\begin{align}\label{eq:adjust_to_mean}
		\xi_n^*=&\dfrac{3n}{n^2-1}\left(\dfrac{1}{n}\underset{i\neq j}{\sum\sum }\abs*{F_Y^{(n)}(Y_{n,i})-F_Y^{(n)}(Y_{n,j})}-\sum_{i=1}^n\abs*{F_Y^{(n)}(Y_{n,i})-F_Y^{(n)}(Y_{n,N(i)})}\right)\nonumber
		\\=&\dfrac{3}{n(n-1)}\underset{i\neq j}{\sum\sum }\abs*{F_Y^{(n)}(Y_{n,i})-F_Y^{(n)}(Y_{n,j})}-\dfrac{3}{n}\sum_{i=1}^n\abs*{F_Y^{(n)}(Y_{n,i})-F_Y^{(n)}(Y_{n,N(i)})}+O(n^{-1}).
	\end{align}
	%The remainder  $E_{1n}\le Cn^{-1},$ since $|F_Y^{(n)}(Y_i)-F_Y^{(n)}(Y_j)|\le 1$ for all $i,j.$ By similar calculation, $\EE\xi_n^*$ can be written as the expectation of the sum above, with a remainder of $E'_{1n}\le Cn^{-1}.$
	
	By the U-statistic projection theorem (see Theorem 12.3 of \cite{vaart}), 
	\begin{align*}
		&\sqrt{n}\,\left[\frac{1}{{n \choose 2}}\sum\limits_{i<j}\Big\{|F_Y^{(n)}(Y_{n,i})-F_Y^{(n)}(Y_{n,j})|-\mathbb{E}|F_Y^{(n)}(Y_{n,i})-F_Y^{(n)}(Y_{n,j})|\Big\}\right] \\
		=& \frac{2}{\sqrt{n}}\sum\limits_{i=1}^{n}h(Y_{n,i})+O(n^{-1/2})\numberthis\label{eq:u_stat}
	\end{align*}
	where %$B_n\le Cn^{-1/2}$ and 
	$h(y):= \mathbb{E}|F_Y^{(n)}(y)-U| - \mathbb{E}|U_1-U_2|$
	with $U,U_1,U_2$ being i.i.d $\mbox{Uniform}\,(0,1)$ random variables. Combining equations \eqref{eq:adjust_to_mean} and \eqref{eq:u_stat} we have
	\begin{align}\label{eq:cha_prep}
		&\sqrt{n}(\xi_n^*-\EE\xi_n^*)\nonumber
		\\=&\dfrac{3}{\sqrt{n}}\sum_{i=1}^n\Big\{2h(Y_{n,i})-|F_Y^{(n)}(Y_{n,i})-F_Y^{(n)}(Y_{n,N(i)})|+\mathbb{E}|F_Y^{(n)}(Y_{n,i})-F_Y^{(n)}(Y_{n,N(i)})|\Big\}+O(n^{-1/2})\nonumber
		\\=&\, S_n+O(n^{-1/2}),
	\end{align}
	where 
	\begin{equation}
		\label{eq:s_n}
		S_n=\dfrac{3}{\sqrt{n}}\displaystyle\sum_{i=1}^n\Big\{2h(Y_{n,i})-|F_Y^{(n)}(Y_{n,i})-F_Y^{(n)}(Y_{n,N(i)})|+\mathbb{E}|F_Y^{(n)}(Y_{n,i})-F_Y^{(n)}(Y_{n,N(i)})|\Big\}.
	\end{equation}
	In the rest of the proof, we will study the asymptotic distribution of $S_n$.
	Let us define $W_n:=S_n/\sqrt{\Var(S_n)}$ and $W^*_n:=\sqrt{n}(\xi_n^*-\EE\xi_n^*)/\sqrt{\Var(\xi_n^*)}$. Using the standard properties of the Wasserstein-1 distance (same as the Kantorovic-Wasserstein distance in Definition 1.1 of the main paper)  %It is clear that   $\EE(\sqrt{n}(A_n+A_n')+B_n)^2\le Cn^{-1}.$ 
	and \eqref{eq:cha_prep} we get
	\begin{equation}\label{eq:cha_prep_wass}
		\mD(W_n^*)\le \mD(W_n)+Cn^{-1/2}.
	\end{equation}
	Now we proceed to bound $\mD(W_n)$. Let us define
	\[
	\mathcal{M}:=\{(X_{n,1},Y_{n,1}),\cdots,(X_{n,n},Y_{n,n})\}\text{ and }\mathcal{M}^\prime := \{(X^\prime_{n,1},Y^\prime_{n,1}),\cdots,(X^\prime_{n,1},Y^\prime_{n,n})\},\]
	where $\mathcal{M}$ and $\mathcal{M}^\prime$ are independent and identically distributed. Let 
	\begin{align*}
		\mathcal{M}^{i}:=& \Big\{(X_{n,1},Y_{n,1}),\cdots,(X^\prime_{n,i},Y^\prime_{n,i}),\cdots,(X_{n,n},Y_{n,n})\Big\}\\
		\text{and}\quad\mathcal{M}^{ij}:=& \Big\{(X_{n,1},Y_{n,1}),\cdots,(X^\prime_{n,i},Y^\prime_{n,i}),\cdots,(X^\prime_{n,j},Y^\prime_{n,j}),\cdots,(X_{n,n},Y_{n,n})\Big\}.
	\end{align*}
	Let us define $W_\ell:(\mathbb{R}^2)^n \rightarrow \mathbb{R}$ as,
	\begin{equation}
		\label{eq:def_w_l}
		W_\ell((x_1,y_1),\cdots,(x_n,y_n)):=3\,[2h(y_{\ell})-|F_Y^{(n)}(y_{\ell})-F_Y^{(n)}(y_{N(\ell)})|+\mathbb{E}|F_Y^{(n)}(y_{\ell})-F_Y^{(n)}(y_{N(\ell)})|].\end{equation}
	
	By graphical rule on a measure space $\mathcal{X}^n$ we mean a mapping from $\mathcal{X}^n$ to the space of undirected graphs on $n$ vertices. For $\bm x \in \mathcal{X}^n$, the graphical rule at $\bm x$ will be denoted by $G(\bm x)$. Such graphical rules will be called symmetric if for any permutation $\pi$ on $[n]$, the edges in $G(x_{\pi(1)},\cdots,G_{\pi(n)})$ will precisely be $\{(\pi(i),\pi(j)): (i,j) \in \mathcal{E}(G(x))\}$. This definition is in line with Section 2.3 of \cite{chatt_nn}. As defined in Section 2.3 of \cite{chatt_nn}, we also consider the definition of a vector $\bm x \in \mathcal{X}^n$ being embedded in a vector $\bm y \in \mathcal{X}^m$ for $m>n$. For a function $f$ defined on $\mathcal{X}^n$, we shall call indices the $i$ and $j$ non interacting with respect to the triplet $(f,\bm x,\bm x^\prime)$ if 
	\[
	f(\bm x)-f(\bm x^j)=f(\bm x^i)-f(\bm x^{ij}),
	\]
	where 
	\[
	x^i_k=\begin{cases}x_k & \mbox{if $k \neq i$,} \\x^\prime_i & \mbox{if $k=i$,}
	\end{cases}
	\] 
	and
	\[
	x^{ij}_k=\begin{cases}x_k & \mbox{if $k \in \{i,j\}$} \\ x^\prime_k & \mbox{if $k \in \{i,j\}$.}
	\end{cases}
	\]
	A graphical rule is called an interaction rule for a function $f$ if the edge $(i,j)$ is absent in all of $G(\bm x),G(\bm x^i),G(\bm x^j)$ and $G(\bm x^{ij})$ implies the pair $(i,j)$ is non-interacting with respect to the triplet $(f,\bm x,\bm x^\prime)$. 
	
	Now let us consider the measure space $\mathcal{X}=\mathbb{R}^2$. For $\bm x \in \mathcal{X}^n$, let us define the following distance of $\{1,\cdots,n\}$.
	\begin{equation}
		\label{eq:def_d}
		D_{\bm x}(i,j) := \begin{cases}
			n+20 & \mbox{if $x_{i} > x_{j}$}\\ \#\Big\{\ell: x_{i} < x_{\ell} < x_{j}\Big\} & \mbox{if $x_{i} \le x_{j}$}.
		\end{cases}
	\end{equation}
	Given a configuration $\bm x \in \mathcal{X}^n$, let us define a graph $\mathcal{G}(\bm x)$ on $[n]$ as follows. For a pair $\{i,j\}$, there exist an edge between $i$ and $j$ if there exists an $\ell \in [n]$ such that,
	\[
	D_{\bm x}(\ell,i) \le 2 \quad \mbox{and} \quad D_{\bm x}(\ell,j) \le 2.
	\]
	Note that this graphical rule is a symmetric rule. The next lemma shows that this is an interaction rule for the function $W_\ell$ for all $1 \le \ell \le n$, and has been proved in~\cref{sec:pfinteraction_rule}.	
	
	\begin{lemma}\label{lem:interaction_rule}
		Consider the graphical rule $\mathcal{G}(\bm x)$ for $\bm x \in (\mathbb{R}^ 2)^n$, as defined above. For any pair of vertices $i,\,j$ if there exists no edge  $\{i,j\}$ in $\mathcal{G}(\bm x), \mathcal{G}(\bm x^{i}), \mathcal{G}(\bm x^{j})$ and $\mathcal{G}(\bm x^{ij})$, then for all $1 \le \ell \le n$, we have,
		\[
		W_\ell(\bm x)-W_\ell(\bm x^{i})-W_\ell(\bm x^{j})+W_\ell(\bm x^{ij})=0,
		\]
		where $W_\ell$ is defined in \eqref{eq:def_w_l}.
	\end{lemma}
	By Lemma \ref{lem:interaction_rule},  $\mathcal{G}$ is a symmetric graphical interaction rule. Let us define
	\[
	\Delta_j := S_n(\mathcal{M})-S_n(\mathcal{M}^j),
	\]
	where $S_n$ is defined in \eqref{eq:s_n}, and $S_n(\mathcal{M})$ is $S_n$ computed with dataset $\mathcal{M}$. Let us define
	\[
	M := \max\limits_{j}|\Delta_j|.
	\]
	Since $F_Y^{(n)}(y)\le 1$ for all $y$, we have a constant $C>0$ such that for all $j \in [n]$,
	\begin{equation}
		\label{eq:M_and_Delta}
		|\Delta_j| \le \frac{C}{\sqrt{n}} \quad \mbox{and} \quad M \le \frac{C}{\sqrt{n}}.
	\end{equation}
	Let us now construct a new graph $\mathcal{G}^\prime(\mathcal{M})$ on vertices $\{1,\cdots,n+4\}$ as follows. There exists an edge between the vertices $i$ and $j$ if and only if there exists an $\ell \in [n+4]$ such that,
	\[
	D_{\mathcal{M}}(\ell,i) \le 6 \quad \mbox{and} \quad D_{\mathcal{M}}(\ell,j) \le 6.
	\]
	Clearly all the edges of $\mathcal{G}(\mathcal{M})$ are in $\mathcal{G}^\prime(\mathcal{M})$, meaning that $\mathcal{G}(\mathcal{M})$ is embedded in $\mathcal{G}^\prime(\mathcal{M})$. Moreover, $\mathcal{G}^\prime(\mathcal{M})$ is a symmetric graphical rule. The degree of any vertex in $\mathcal{G}^\prime(\mathcal{M})$ is bounded by $14$ as if there exists an $\ell \in [n]$ such that,
	\[
	D_{\mathcal{M}}(\ell,i) \le 2 \quad \mbox{and} \quad D_{\mathcal{M}}(\ell,j) \le 2,
	\]
	then $i-7 \le j \le i+7$. Hence, almost surely
	\begin{equation}
		\label{eq:delta_1}
		\delta := 1+ \mbox{degree of vertex $1$ in $\mathcal{G}^\prime$} \le 15.
	\end{equation}
	Now, using Theorem 2.5 of \cite{chatt_nn}, we have an absolute constant $C_1>0$ such that
	\begin{equation}
		\mD(W_n) \le \frac{C_1 n^{1/2}}{\sigma^{*2}_n}\mathbb{E}(M^8)^{1/4}\mathbb{E}(\delta^4)^{1/4} + \frac{1}{2\sigma^{*3}_n}\sum\limits_{j=1}^{n}\mathbb{E}|\Delta_j|^3,
	\end{equation}
	where, 
	\[
	\sigma^{*2}_n=\var\left(\frac{1}{\sqrt{n}}\sum\limits_{\ell=1}^{n}W_\ell(\mathcal{M})\right),
	\]
	and $W_\ell(\mathcal{M})$ is the function $W_\ell$ computed with the dataset $\mathcal{M}$.
	The proof of Theorem A.2 is now completed by plugging in the bounds from  \eqref{eq:cha_prep}, \eqref{eq:M_and_Delta}, and \eqref{eq:delta_1}, into \eqref{eq:cha_prep_wass}.
	
	\qed
	
	\section{Proofs of Lemmas}\label{sec:lemmas}
	
	\subsection{Proof of Lemma \ref{lem:evxsterr}}\label{sec:pfevxsterr} To reduce notation, hide the dependence on $n$ and write $\jtd(\cdot,\cdot)$, $\mxd(\cdot)$, $\myd(\cdot)$ and $F_Y(\cdot)$ to mean $\jtd^{(n)}(\cdot,\cdot)$, $\mxd^{(n)}(\cdot)$, $\myd^{(n)}$ and $F_Y^{(n)}(\cdot)$ respectively. We will also write $(X_1,Y_1),\dots,(X_n,Y_n)$ to mean the random variables $(X_{n,1},Y_{n,1}),\ldots,(X_{n,n},Y_{n,n})\sim\jtd^{(n)}$ from the triangular array.
	
	We defer the steps for $T_1^*$ to the proof of \cref{lem:evxierr} and start with the possibly more complicated term $T_2^*$. Observe that for any $i,j$,
	\begin{align}\label{eq:usiden}
		\min\{F(Y_i),F(Y_j)\}=\int \ind(y\leq \min\{Y_i,Y_j\})\myd(y)\,dy.
	\end{align}  
	Using~\eqref{eq:usiden} in the definition of $T_2^*$ and writing $Q_{21}^*$ and $Q_{22}'^*$ for terms to be bounded at the end of the proof, we get:
	\begin{align}\label{eq:usiden1}
		\E T_2^*=&\frac{1}{n}\sum_{i=1}^n \E\int \ind(y_4\leq \min\{y_1,y_2\})\ind(y_5\leq \min\{y_2,y_3\})\myd(y_4)\myd(y_5)\jtdc(y_1|X_i)\nonumber\\ &\hspace{60pt}\jtdc(y_2|X_{N(i)})\jtdc(y_3|X_{N(N(i))})\,dy_1\,dy_2\,dy_3\,dy_4\,dy_5\nonumber \\ 
		-&\frac{1}{n}\sum_{i=1}^n \E\int \ind(y_5\leq \min\{y_1,y_2\})\ind(y_6\leq \min\{y_3,y_4\})\myd(y_5)\myd(y_6)\jtdc(y_1|X_i)\nonumber \\ &\hspace{60pt}\jtdc(y_2|X_{N(i)})\jtdc(y_3|X_{N(i)})\jtdc(y_4|X_{N(N(i))})\,dy_1\,dy_2\,dy_3\,dy_4\,dy_5\,dy_6\nonumber\\ 
		=&\,Q_{21}^*+\frac{1}{n}\sum_{i=1}^n \E\int \ind(y_4\leq \min\{y_1,y_2\})\ind(y_5\leq \min\{y_2,y_3\})\myd(y_4)\myd(y_5)\jtdc(y_1|X_i)\nonumber\\ 
		&\hspace{80pt}\jtdc (y_2|X_{i})\jtdc (y_3|X_{i})\,dy_1\,dy_2\,dy_3\,dy_4\,dy_5\nonumber\\ &\hspace{20pt}-\frac{1}{n}\sum_{i=1}^n \E\int \ind(y_5\leq \min\{y_1,y_2\})\ind(y_6\leq \min\{y_3,y_4\})\myd (y_5)\myd (y_6)\jtdc (y_1|X_i)\nonumber \\ 
		&\hspace{80pt}\jtdc(y_2|X_i)\jtdc(y_3|X_{i})\jtdc (y_4|X_i)\,dy_1\,dy_2\,dy_3\,dy_4\,dy_5\,dy_6\nonumber\\
		=&~Q_{21}^*+Q_{22}^*+\PP(Y_4\le \min\{Y_1,Y_2\},\,Y_5\le \min\{Y_2,Y_3\}) \nonumber\\
		&\hspace{60pt}-\PP(Y_5\le \min\{Y_1,\,Y_2\})\PP(Y_6\le \min\{Y_3,\,Y_4\})\nonumber \\ 
		=&\frac{1}{45}+Q_{21}^*+Q_{22}^*.
	\end{align}
	In terms of the notation used in the proof of Theorem 2.1 in the main paper, $Q_2^*=Q_{21}^*+Q_{22}^*$. We now move on to the error terms $Q_{21}^*$ and $Q_{22}^*$. 
	
	\medskip
	
	\noindent\textbf{Bound for $Q_{21}^*$}. Note that by our definition of $Q_1$, one has 
	\begin{align}\label{eq:usiden2}
		&|Q_{21}^*|\nonumber\\
		=&~
		\bigg\vert\frac{1}{n}\sum_{i=1}^n \E\int \ind(y_4\leq \min\{y_1,y_2\})\ind(y_5\leq \min\{y_2,y_3\})\myd(y_4)\myd(y_5)\jtdc(y_1|X_i)\nonumber\\ 
		&\hspace{60pt}
		\jtdc(y_2|X_{N(i)})\jtdc(y_3|X_{N(N(i))})\,dy_1\,dy_2\,dy_3\,dy_4\,dy_5
		\nonumber \\ 
		&-\frac{1}{n}\sum_{i=1}^n \E\int \ind(y_5\leq \min\{y_1,y_2\})\ind(y_6\leq \min\{y_3,y_4\})\myd(y_5)\myd(y_6)\jtdc(y_1|X_i)\nonumber \\ 
		&
		\hspace{60pt}\jtdc(y_2|X_{N(i)})\jtdc(y_3|X_{N(i)})\jtdc(y_4|X_{N(N(i))})\,dy_1\,dy_2\,dy_3\,dy_4\,dy_5\,dy_6\nonumber\\ 
		&
		-\frac{1}{n}\sum_{i=1}^n \E\int \ind(y_4\leq \min\{y_1,y_2\})\ind(y_5\leq \min\{y_2,y_3\})\myd(y_4)\myd(y_5)\jtdc(y_1|X_i)\nonumber\\ 
		&\hspace{80pt}
		\jtdc (y_2|X_{i})\jtdc (y_3|X_{i})\,dy_1\,dy_2\,dy_3\,dy_4\,dy_5\nonumber\\ &\hspace{20pt}
		+\frac{1}{n}\sum_{i=1}^n \E\int \ind(y_5\leq \min\{y_1,y_2\})\ind(y_6\leq \min\{y_3,y_4\})\myd (y_5)\myd (y_6)\jtdc (y_1|X_i)\nonumber \\ 
		&\hspace{80pt}
		\jtdc(y_2|X_i)\jtdc(y_3|X_{i})\jtdc (y_4|X_i)\,dy_1\,dy_2\,dy_3\,dy_4\,dy_5\,dy_6
		\bigg\vert
		\nonumber\\
		\lesssim&~  \frac{1}{n}\sum_{i=1}^n \E\int\big|\P(Y\geq \max\{y_4,y_5\}|X_i)-\P(Y\geq \max\{y_4,y_5\}|X_{N(i)})\big|\myd (y_4)\myd (y_5)\,dy_4\,dy_5\nonumber \\&+\frac{1}{n}\sum_{i=1}^n \E\int\big|\P(Y\geq y_5|X_i)-\P(Y\geq y_5|X_{N(N(i))})\big|\myd (y_5)\,dy_5\nonumber \\ &+\frac{1}{n}\sum_{i=1}^n \E\int\big|\P(Y\geq y_5|X_i)-\P(Y\geq y_5|X_{N(i)})\big|\myd (y_5)\,dy_5+O(n^{-1}).
	\end{align}
	Let us now focus on the first term on the right hand side of~\eqref{eq:usiden2}. Towards this direction, by using Assumption (A1),~(2.2) from the main paper, we further get:
	\begin{align}\label{eq:usiden3st}
		&\;\;\;\;\;\E\int\big|\P(Y\geq \max\{y_4,y_5\}|X_1)-\P(Y\geq \max\{y_4,y_5\}|X_{N(1)})\big|\myd (y_4)\myd (y_5)\,dy_4\,dy_5\nonumber \\ &\leq \E\left[\min\left\{(1+L_1 (X_1,\max\{Y_4,Y_5\})+L_2 (X_{N(1)},\max\{Y_4,Y_5\}))|X_1-X_{N(1)}|^{\eta},1\right\}\right]\nonumber \\ &\lesssim \P(|X_1-X_{N(1)}|\geq 1)+\left\{\E\left(1+(L_1 (X_1,\max\{Y_4,Y_5\}))^{\theta}+(L_2 (X_1,\max\{Y_4,Y_5\}))^{\theta}\right)\right\}^{\frac{1}{\theta}}\nonumber \\ &\left\{\E|X_1-X_{N(1)}|^{\frac{\theta\eta}{\theta-1}}\ind(|X_1-X_{N(1)}|\leq 1)\right\}^{\frac{\theta-1}{\theta}},
	\end{align}
	where the last two lines follow from Assumption (A1),~(2.2) in the main paper and H\"{o}lder's inequality. Now we will bound each term on the right hand side of~\eqref{eq:usiden3st} separately. First note that, 
	\begin{equation}\label{eq:usiden4}
		\limsup\limits_{n\to\infty}\E(L_1^{(n)}(X_1,\max\{Y_4,Y_5\}))^2\leq 2\limsup\limits_{n\to\infty}E(L_1^{(n)}(X_1,Y_4))^2\lesssim 1,
	\end{equation}
	where the last line follows from Assumption (A2),~(2.3) in the main paper. Next, on using~\cite[Lemma 9.4]{sc_corr} and~\eqref{eq:usiden4}, it further follows that,
	\begin{equation}\label{eq:usiden5}
		\limsup\limits_{n\to\infty}\E(L_1^{(n)}(X_{N(1)},\max\{Y_4,Y_5\}))^2\lesssim 1.
	\end{equation}
	Using~\cref{lem:nndist}, we then get:
	\begin{align}\label{eq:usiden6}
		\left\{\E|X_1-X_{N(1)}|^{\frac{\theta\eta}{\theta-1}}\ind(|X_1-X_{N(1)}|\leq 1)\right\}^{\frac{\theta-1}{\theta}}&\leq \left\{\E|X_1-X_{N(1)}|^{{\frac{\theta\eta}{\theta-1}}\wedge 1}\ind(|X_1-X_{N(1)}|\leq 1)\right\}^{\frac{\theta-1}{\theta}}\nonumber \\&\lesssim \left(\frac{(\log{n})^2}{n}\right)^{\left(\frac{\gamma(\theta-1)}{\theta(\gamma+1)}\wedge \frac{\eta \gamma}{\gamma+1}\right)}.
	\end{align}
	Plugging in the conclusions from~\eqref{eq:usiden4},~\eqref{eq:usiden5},~and~\eqref{eq:usiden6} into~\eqref{eq:usiden3st}, we get that the first term on the right hand side of~\eqref{eq:usiden2} satisfies,
	\begin{align*}\label{eq:usiden16st}
		&\frac{1}{n}\sum_{i=1}^n \E\int\big|\P(Y\geq \max\{y_4,y_5\}|X_i)-\P(Y\geq \max\{y_4,y_5\}|X_{N(i)})\big|\myd (y_4)\myd (y_5)\,dy_4\,dy_5\\ 
		&\leq n^{-\frac{\gamma}{\gamma+1}}(\log{n})^2+\left(\frac{(\log{n})^2}{n}\right)^{\left(\frac{\gamma(\theta-1)}{\theta(\gamma+1)}\wedge \frac{\eta \gamma}{\gamma+1}\right)}:=b_n.\numberthis
	\end{align*}
	By using a similar argument as above, we get the same bound for the other two terms on the right hand side of~\eqref{eq:usiden2}, which implies,
	\begin{align*}
		|Q_{21}^*|\lesssim b_n.
	\end{align*}
	Also, from~\eqref{eq:usiden2}, provided $X_1$ and $Y_1$ are independent, i.e., $\xj=0$, we also have $R^{(1)}=O(n^{-1})$. Therefore,
	\begin{align}\label{eq:usiden7}
		|Q_{21}^*|\lesssim b_n\ind(\xj>0)+n^{-1}.
	\end{align}
	
	\medskip
	
	\noindent\textbf{Bound for $Q_{22}^*$}. We have
	\begin{align*}
		Q_{22}^*
		=&~
		\frac{1}{n}\sum_{i=1}^n \E\int \ind(y_4\leq \min\{y_1,y_2\})\ind(y_5\leq \min\{y_2,y_3\})\myd(y_4)\myd(y_5)\jtdc(y_1|X_i) \\
		&\hspace{80pt}\jtdc (y_2|X_{i})\jtdc (y_3|X_{i})\,dy_1\,dy_2\,dy_3\,dy_4\,dy_5\\ &\hspace{20pt}-\frac{1}{n}\sum_{i=1}^n \E\int \ind(y_5\leq \min\{y_1,y_2\})\ind(y_6\leq \min\{y_3,y_4\})\myd (y_5)\myd (y_6)\jtdc (y_1|X_i) \\ 
		&\hspace{80pt}\jtdc(y_2|X_i)\jtdc(y_3|X_{i})\jtdc (y_4|X_i)\,dy_1\,dy_2\,dy_3\,dy_4\,dy_5\,dy_6\\
		& \nonumber\\
		&~-\PP(Y_4\le \min\{Y_1,Y_2\},\,Y_5\le \min\{Y_2,Y_3\}) +\PP(Y_5\le \min\{Y_1,\,Y_2\})\PP(Y_6\le \min\{Y_3,\,Y_4\}).
	\end{align*}
	Then it can be checked that
	\begin{align}\label{eq:usiden8}
		|Q_{22}^*|\leq |S_1|+2|S_2|+|S_3|+2|S_4|+|S_5|+4|S_6|+2|S_7|+4|S_8|+4|S_9|+|S_{10}|,
	\end{align}
	where
	\begin{align*}
		S_1&:=\E\int \ind(y_4\leq \min\{y_1,y_2\})\ind(y_5\leq \min\{y_2,y_3\})\myd (y_4)\myd (y_5)\myd (y_1)\nonumber\\ 
		&\hspace{30pt}(\jtdc (y_2|X_{1})-\myd (y_2))\myd (y_3)\,dy_1\,dy_2\,dy_3\,dy_4\,dy_5
		\\
		S_2&:=\E\int \ind(y_4\leq \min\{y_1,y_2\})\ind(y_5\leq \min\{y_2,y_3\})\myd (y_4)\myd (y_5)(\jtdc (y_1|X_1)-\myd (y_1))\nonumber\\ 
		&\hspace{30pt}\myd (y_2)\myd (y_3)\,dy_1\,dy_2\,dy_3\,dy_4\,dy_5
		\\
		S_3&:=\E\int \ind(y_4\leq \min\{y_1,y_2\})\ind(y_5\leq \min\{y_2,y_3\})\myd (y_4)\myd (y_5)(\jtdc (y_1|X_1)-\myd (y_1))\nonumber\\ 
		&\hspace{30pt}\myd (y_2)(\jtdc (y_3|X_1)-\myd (y_3))\,dy_1\,dy_2\,dy_3\,dy_4\,dy_5
		\\
		S_4&:=\E\int \ind(y_4\leq \min\{y_1,y_2\})\ind(y_5\leq \min\{y_2,y_3\})\myd (y_4)\myd (y_5)\myd (y_1)\nonumber\\ 
		&\hspace{30pt}(\jtdc (y_2|X_1)-\myd (y_2))(\jtdc (y_3|X_1)-\myd (y_3))\,dy_1\,dy_2\,dy_3\,dy_4\,dy_5
		\\
		S_5&:=\E\int \ind(y_4\leq \min\{y_1,y_2\})\ind(y_5\leq \min\{y_2,y_3\})\myd (y_4)\myd (y_5)(\jtd (y_1|X_1)-\myd (y_1))\nonumber\\ 
		&\hspace{30pt}(\jtdc (y_2|X_1)-\myd (y_2))(\jtdc (y_3|X_1)-\myd (y_3))\,dy_1\,dy_2\,dy_3\,dy_4\,dy_5
		\\
		S_6&:=\E\int \ind(y_5\leq \min\{y_1,y_2\})\ind(y_6\leq \min\{y_3,y_4\})\myd (y_5)\myd (y_6)(\jtdc (y_1|X_1)-\myd (y_1))\nonumber \\ 
		&\hspace{30pt}\myd (y_2)\myd (y_3)\myd (y_4)\,dy_1\,dy_2\,dy_3\,dy_4\,dy_5\,dy_6
		\\
		S_7&:=\E\int \ind(y_5\leq \min\{y_1,y_2\})\ind(y_6\leq \min\{y_3,y_4\})\myd (y_5)\myd (y_6)(\jtdc (y_1|X_1)-\myd (y_1))\nonumber \\ 
		&\hspace{30pt}(\jtdc (y_2|X_1)-\myd (y_2))\myd (y_3)\myd (y_4)\,dy_1\,dy_2\,dy_3\,dy_4\,dy_5\,dy_6
		\\
		S_8&:=\E\int \ind(y_5\leq \min\{y_1,y_2\})\ind(y_6\leq \min\{y_3,y_4\})\myd (y_5)\myd (y_6)(\jtdc (y_1|X_1)-\myd (y_1))\nonumber \\ 
		&\hspace{30pt}\myd (y_2)(\jtdc (y_3|X_1)-\myd (y_3))\myd (y_4)\,dy_1\,dy_2\,dy_3\,dy_4\,dy_5\,dy_6
		\\
		S_9&:=\E\int \ind(y_5\leq \min\{y_1,y_2\})\ind(y_6\leq \min\{y_3,y_4\})\myd (y_5)\myd (y_6)(\jtdc (y_1|X_1)-\myd (y_1))\nonumber \\ 
		&\hspace{30pt}(\jtdc (y_2|X_1)-\myd (y_2))(\jtdc (y_3|X_1)-\myd (y_3))\myd (y_4)\,dy_1\,dy_2\,dy_3\,dy_4\,dy_5\,dy_6
		\\
		S_{10}&:=\E\int \ind(y_5\leq \min\{y_1,y_2\})\ind(y_6\leq \min\{y_3,y_4\})\myd (y_5)\myd (y_6)(\jtdc (y_1|X_1)-\myd (y_1))\nonumber \\ 
		&\hspace{30pt}(\jtdc (y_2|X_i)-\myd (y_2))(\jtdc (y_3|X_1)-\myd (y_3))(\jtdc (y_4|X_1)-\myd (y_4))\,dy_1\,dy_2\dots\,dy_6
	\end{align*}
	Now we will bound each of the $S_i$'s individually. We start with $S_1$. Observe that
	\begin{align}\label{eq:usiden9}
		S_1&=\int \ind(y_4\leq \min\{y_1,y_2\})\ind(y_5\leq \min\{y_2,y_3\})\myd (y_4)\myd (y_5)\myd (y_1)\times \nonumber\\ &\times \left(\int (\jtd (y_2|X_1=x)-\myd (y_2))\mxd (x)\,dx\right)\myd (y_3)\,dy_1\,dy_2\,dy_3\,dy_4\,dy_5=0
	\end{align}
	A similar calculation shows that $S_2=0$. We now look at $S_3$. Note that,
	\begin{align}\label{eq:usiden10}
		S_3=&\,\E\int\limits_{y_4,y_5}\left(\int_{y_2} \ind(y_2\geq \max\{y_4,y_5\})\myd (y_2)\,dy_2\right)\bigg(\int_{y_1}\ind(y_1\geq y_4)(\jtdc (y_1|X_1)\nonumber \\
		&-\myd (y_1))\,dy_1\bigg)\left(\int_{y_3}\ind(y_3\geq y_5)(\jtdc (y_3|X_1)-\myd (y_3))\,dy_3\right)\myd (y_4)\myd (y_5)\,dy_4\,dy_5\nonumber \\
		=&\,\E\int\limits_{y_4,y_5}\left(\P(Y\geq \max\{y_4,y_5\})\right)\left(\P(Y\geq y_4|X_1)-\P(Y\geq y_4)\right)\left(\P(Y\geq y_5|X_1)-\P(Y\geq y_5)\right)\nonumber \\
		&\hspace{20pt}\myd (y_4)\myd (y_5)\,dy_4\,dy_5\nonumber \\
		\leq &\,\E\left[\int_{y_4} \big|\P(Y\geq y_4|X_1)-\P(Y\geq y_4)\big|\myd (y_4)\,dy_4\right]^2\nonumber \\ 
		\leq &\,\E\int \left(\P(Y\geq y_4|X_1)-\P(Y\geq y_4)\right)^2\myd (y_4)\,dy_4=\frac{1}{6}\xj,
	\end{align}
	where the last step uses the Cauchy-Schwarz inequality. In order to bound $S_4$, let us start with:
	\begin{align}\label{eq:usiden11}
		S_4=&\E\int\limits_{y_4,y_5}\left(\int_{y_2} \ind(y_2\geq \max\{y_4,y_5\})(\jtdc (y_2|X_1)-\myd (y_2))\,dy_2\right)\bigg(\int_{y_1}\ind(y_1\geq y_4)\myd (y_1)\,dy_1\bigg)\nonumber \\
		&\hspace{40pt}\left(\int_{y_3}\ind(y_3\geq y_5)(\jtdc (y_3|X_1)-\myd (y_3))\,dy_3\right)\myd (y_4)\myd (y_5)\,dy_4\,dy_5\nonumber\\
		&=\E\int\limits_{y_4,y_5}\left(\P(Y\geq \max\{y_4,y_5\}|X_1)-\P(Y\geq \max\{y_4,y_5\})\right)\left(\P(Y\geq y_5|X_1)-\P(Y\geq y_5)\right)\nonumber \\
		&\hspace{40pt}\left(\P(Y\geq y_4)\right)\myd (y_4)\myd (y_5)\,dy_4\,dy_5\nonumber \\
		=&\E\int\limits_{y_4}\int\limits_{y_5\leq y_4}\left(\P(Y\geq y_5|X_1)-\P(Y\geq y_5)\right)\left(\P(Y\geq y_4|X_1)-\P(Y\geq y_4)\right)\left(\P(Y\geq y_4)\right)\nonumber \\ 
		&\hspace{50pt}\myd (y_4)\myd (y_5)\,dy_5\,dy_4 \nonumber\\
		&+\E\int\limits_{y_4}\int\limits_{y_5\geq y_4}\left(\P(Y\geq y_5|X_1)-\P(Y\geq y_5)\right)^2\left(\P(Y\geq y_4)\right)
		\myd (y_4)\myd (y_5)\,dy_5\,dy_4\nonumber \\ 
		\leq& \E\left[\int_{y_4} \big|\P(Y\geq y_4|X_1)-\P(Y\geq y_4)\big|\myd (y_4)\,dy_4\right]^2+\E\int \left(\P(Y\geq y_4|X_1)-\P(Y\geq y_4)\right)^2\myd (y_4)\,dy_4\nonumber 
		\\ \leq&\,\,\frac{1}{6}\xj+\frac{1}{6}\xj=\frac{1}{3}\xj.
	\end{align}
	For $S_5$, we observe that
	\begin{align*}
		S_5=&~\E\int\limits_{y_4,y_5} \left(\P(Y\geq y_4|X_1)-\P(Y\geq y_4)\right)\left(\P(Y\geq y_5|X_1)-\P(Y\geq y_5)\right) \\
		&\hspace{30pt}\ind(y_2\geq \max\{y_4,y_5\})(\jtdc (y_2|X_1)-\myd (y_2))\,dy_2\myd (y_4)\myd (y_5)\,dy_4\,dy_5\\
		\leq &~
		2 \E\int \left(\P(Y\geq y_4|X_1)-\P(Y\geq y_4)\right)^2\myd(y_4)\,dy_4=\frac{1}{3}\xj\numberthis\label{eq:usiden12}
	\end{align*}
	A similar set of calculations can be used to show that $\max_{i\geq 6} |S_i|\lesssim \xj$. Combining the above observation with~\eqref {eq:usiden9},~\eqref{eq:usiden10},~\eqref{eq:usiden11}~and~\eqref{eq:usiden12}, we get $\max_{i\geq 1} |S_i|\lesssim \xj$. Applying this observation in~\eqref{eq:usiden8}, we further have:
	$$|Q_{22}^*|\lesssim \sum_{i\geq 1} |S_i|\lesssim \xj.$$
	Combining the above display with~\eqref{eq:usiden7}~and~\eqref{eq:usiden1} we finally obtain,
	\begin{equation}\label{eq:usiden14}
		\bigg|\EE T_2^*-\frac{1}{45}\bigg|\lesssim Q_{21}^*+Q_{22}^*\lesssim b_n\ind(\xj>0)+\xj+\frac{1}{n}.
	\end{equation}
	A similar set of computations can be used in $\EE T_i^*$, $i=1,3,4,5$. We skip the relevant algebraic details for brevity. We present the corresponding conclusions below:
	\begin{equation}\label{eq:usiden15}
		\max\left\{\bigg|\EE T_1^*-\frac{1}{18}\bigg|,\bigg|\E T_3^*-\frac{1}{45}\bigg|,\bigg|\E T_4^*-\frac{1}{45}\bigg|,\bigg|\E T_5^*-\frac{1}{45}\bigg|\right\}\lesssim n^{-1}+\xj+b_n\ind(\xj>0).
	\end{equation}
	By the definition of $Q_{i}^*$ for $i=1,\dots,5$ this completes the proof of \cref{lem:evxsterr}.
	\qed
	
	\subsection{Proof of Lemma \ref{lem:evxierr}}\label{sec:pfevxierr} The proof is similar to that of \cref{lem:evxsterr}. We decompose the first term $T_1$. Notice that $T_1^*$ from \cref{lem:evxsterr} is very similar and can be analysed similarly. To reduce notation, we will hide the dependence on $n$ and write $\jtd(\cdot,\cdot)$, $\mxd(\cdot)$, $\myd(\cdot)$ and $F_Y(\cdot)$ to mean $\jtd^{(n)}(\cdot,\cdot)$, $\mxd^{(n)}(\cdot)$, $\myd^{(n)}(\cdot)$ and $F_Y^{(n)}(\cdot)$ respectively. We will write $(X_1,Y_1),\dots,(X_n,Y_n)$ to mean the random variables $(X_{n,1},Y_{n,1}),\ldots,(X_{n,n},Y_{n,n})\sim\jtd^{(n)}$ from the triangular array. Also we will use $\bY$ and $\bX$ for the set of random variables $(Y_1,\ldots ,Y_n)$ and $(X_1,\ldots ,X_n)$ respectively.
	\begin{align*}
		&\EE T_1\\
		=&~\dfrac{n}{(n^2-1)^2}\sum_{i=1}^n\EE\Var(\min\{R_i,R_{N(i)}|\bX\})	\\
		=&\EE\int \dfrac{n}{(n^2-1)^2}\sum_{i=1}^{n-1}(\min\{R_i,R_{N(i)}\})^2f_{\bY|\bX}(\by)d\by-\dfrac{n}{(n^2-1)^2}\sum_{i=1}^n\EE\left[\int \min\{R_i,R_{N(i)}\}f_{\bY|\bX}(\by)d\by\right]^2
		\\=&\EE\int \left[\dfrac{n}{(n^2-1)^2}\sum_{i=1}^{n-1}(\min\{R_i,R_{N(i)}\})^2\right]\left\{f_{\bY}(\by)+f_{\bY|\bX}(\by)-f_{\bY}(\by)\right\}d\by
		\\& \hspace*{3cm}-\dfrac{n}{(n^2-1)^2}\sum_{i=1}^n\EE\left[\int \min\{R_i,R_{N(i)}\}\left\{f_{\bY}(\by)+f_{\bY|\bX}(\by)-f_{\bY}(\by)\right\}d\by\right]^2
		\\=:&\dfrac{n}{(n^2-1)^2}\sum_{i=1}^{n-1}\EE\left[\int (\min\{R_i,R_{N(i)}\})^2f_{\bY}(\by)d\by-\left(\int \min\{R_i,R_{N(i)}\}f_{\bY}(\by)d\by\right)^2\right]+Q_{11}+Q_{12}
		\\=&\dfrac{n}{(n^2-1)^2}\sum_{i=1}^n\sum_{k\neq i,\,N(i)}\EE\Var(\ind(Y_k\le \min\{Y_i,Y_{N(i)}\}))
		\\+&\dfrac{n}{(n^2-1)^2}\sum_{i=1}^n\underset{
			\underset{k_1,k_2\neq i,\,N(i)}{k_1\neq k_2}}{\sum\sum }\EE\Cov(\ind(Y_{k_1}\le \min\{Y_i,Y_{N(i)}\}),\ind(Y_{k_2}\le \min\{Y_i,Y_{N(i)}\}))+Q_{11}+Q_{12}
		\\=&\PP(Y_1\le \min\{Y_3,Y_4\},Y_2\le \min\{Y_3,Y_4\})-\PP(Y_1\le \min\{Y_3,Y_4\})\PP(Y_2\le \min\{Y_3,Y_4\})
		\\+&O(n^{-1})+Q_{11}+Q_{12}
		\\=&\dfrac{1}{6}-\dfrac{1}{3}\cdot\dfrac{1}{3}+Q_{11}+Q_{12}+O(n^{-1})
		\\=&\dfrac{1}{18}+Q_{11}+Q_{12}+O(n^{-1}).\numberthis\label{eq:T1}
	\end{align*}
	
	\medspace
	In terms of the notation used in the proof of Theorem 2.1 in the main paper, $Q_1=Q_{11}+Q_{12}$.
	\medspace
	
	\noindent \textbf{Bounding $Q_{11}$}. Since $\min\{R_i,R_{N(i)}\}=1+\disp\sum_{k\neq i,N(i)}\ind(Y_k\le Y_i,Y_{N(i)}),$ we have
	\begin{align*}
		Q_{11}:=&\int \left[\dfrac{n}{(n^2-1)^2}\sum_{i=1}^{n-1}(\min\{R_i,R_{N(i)}\})^2\right](f_{\bY|\bX}(\by)-f_{\bY}(\by))d\by
		\\=& \dfrac{n}{(n^2-1)^2}\sum_{i=1}^n\underset{k_1\neq k_2}{\sum\sum}\int\left[\ind(y_{k_1}\le y_i\wedge y_{N(i)})\ind(y_{k_2}\le y_i \wedge y_{N(i)})\right](f_{\bY|\bX}(\by)-f_{\bY}(\by))d\by
		\\&+O(n^{-1})
		\\=&\int \ind(y_{1}\le \min\{y_3,y_{4}\})\ind(y_{2}\le \min\{y_3,y_{4}\})\left(\prod_{i=1}^4f_{Y|X}(y_i|X_i)-\prod_{i=1}^4f_{Y}(y_i)\right)dy_1dy_2dy_3dy_4
		\\&+O(n^{-1}).\numberthis\label{eq:E11}
	\end{align*}
	The difference of product pdfs can be written out as in the expansion of 
	$$
	a^4-b^4=a^3(a-b)+a^2b(a-b)+ab^2(a-b)+b^3(a-b).
	$$ 	
	For example, term 1 (corresponding to $a^3(a-b)$) will be
	\begin{align*}
		&\EE\int \ind(y_{1},y_2\le \min\{y_3,y_{4}\})f_{Y|X}(y_1|X_1)f_{Y|X}(y_2|X_2)f_{Y|X}(y_3|X_3)\left(f_{Y|X}(y_4|X_{N(3)})-f_Y(y_4)\right)d\by	
		\\=&Q_{11}'+\EE\int\ind(y_{1},y_2\le \min\{y_3,y_{4}\})f_{Y|X}(y_1|X_1)f_{Y|X}(y_2|X_2)f_{Y|X}(y_3|X_3)
		\left(f_{Y|X}(y_4|X_{3})-f_Y(y_4)\right)d\by
		\\=&~Q_{11}'+
		\EE\int \ind(y_{1},y_2\le \min\{y_3,y_{4}\})f_{Y|X}(y_1|X_1)f_{Y|X}(y_2|X_2)
		\\&\hspace{50pt}(f_{Y|X}(y_3|X_3)-f_Y(y_3))
		\left(f_{Y|X}(y_4|X_{3})-f_Y(y_4)\right)d\by
		\\+&~\EE\int \ind(y_{1},y_2\le \min\{y_3,y_{4}\})f_{Y|X}(y_1|X_1)f_{Y|X}(y_2|X_2)f_Y(y_3)
		\left(f_{Y|X}(y_4|X_{3})-f_Y(y_4)\right)d\by
		\\=&~Q_{11}'+\EE\int \ind(y_{1},y_2\le \min\{y_3,y_{4}\})f_{Y|X}(y_1|X_1)f_{Y|X}(y_2|X_2)
		\\&\hspace{60pt}(f_{Y|X}(y_3|X_3)-f_Y(y_3))
		\left(f_{Y|X}(y_4|X_{3})-f_Y(y_4)\right)d\by
		+0
		\\\le &~Q_{11}'+\int \left(\int\left[\int_{y_1\vee y_2}^\infty (f_{Y|X}(y_3|X_3)-f(y_3))dy_3\right]^2f(x_3)dx_3\right)f_{Y}(y_1)f_{Y}(y_2)dy_1dy_2
		\\\le & \,\,Q_{11}'+6\xj. \numberthis\label{eq:E11_term1}
	\end{align*}
	The last expectation term on the third line above is zero by independence of $X_i$ since $\int f_{Y|X} (y|X_3)f_X(x_3)dx_3=f_Y(y).$ We also use Cauchy-Schwarz inequality in the second last line, and the definition of $\xi$ in the last line. The remainder $Q_{11}'$ can be bounded as
	\begin{align*}
		Q_{11}'
		=:	&~\EE\int \ind(y_{1},y_2\le y_3)f_{Y|X} (y_1|X_1)f_{Y|X}(y_2|X_2)f_{Y|X}(y_3|X_3)\\ &~\qquad\qquad\qquad\int_{\max(y_1,y_2)}^\infty(f_{Y|X_{3}}(y_4)-f_{Y|X_{N(3)}}(y_4))d\by \\
		=&~\EE\int \ind(y_{1},y_2\le y_3)f_{Y|X}(y_1|X_1)f_{Y|X}(y_2|X_2)f_{Y|X}(y_3|X_3)\\
		&~\hspace{50pt}(F_{X,Y}(y_1\vee y_2|X_{N(3)})-F_{X,Y}(y_1\vee y_2|X_3))d\by	\\
		\le &~\EE \int f_{Y|X}(y_1|X_1)f_{Y|X}(y_2|X_2)(F_{X,Y}(y_1\vee y_2|X_{N(3)})-F_{X,Y}(y_1\vee y_2|X_3)) dy_1dy_2 \\
		\leq &~ \E\left[\min\left\{(1+L_1^{(n)}(X_3,\max\{Y_1,Y_2\})+L_2^{(n)}(X_{N(3)},\max\{Y_1,Y_2\}))|X_3-X_{N(3)}|^{\eta},1\right\}\right]\\ 
		\lesssim &~\P(|X_3-X_{N(3)}|\geq1)
		+\left\{\E\left(1+(L_1^{(n)}(X_3,\max\{Y_1,Y_2\}))^{\theta}+(L_2^{(n)}(X_3,\max\{Y_1,Y_2\}))^{\theta}\right)\right\}^{\frac{1}{\theta}} \\ 
		&~\times \left\{\E|X_3-X_{N(3)}|^{\frac{\theta\eta}{\theta-1}}\ind(|X_3-X_{N(3)}|\leq 1)\right\}^{\frac{\theta-1}{\theta}},\numberthis\label{eq:usiden3}
	\end{align*}
	where the last two lines follow by Assumptions (A1) and (A2) from the main paper and H\"{o}lder's inequality.
	Plugging in the conclusions from~\eqref{eq:usiden4},~\eqref{eq:usiden5},~and~\eqref{eq:usiden6} into~\eqref{eq:usiden3}, we get that 
	\begin{equation}\label{eq:usiden16}
		Q_{11}' \leq n^{-\frac{\gamma}{\gamma+1}}(\log{n})^2+\left(\frac{(\log{n})^2}{n}\right)^{\left(\frac{\gamma(\theta-1)}{\theta(\gamma+1)}\wedge \frac{\eta \gamma}{\gamma+1}\right)}:=b_n,
	\end{equation}
	By using a similar argument as above, we get the same bound for the other terms in $Q_{11}$, which implies
	$$
	Q_{11}\le \xj+b_n.
	$$
	Also from~\eqref{eq:E11}, provided $X_1$ and $Y_1$ are independent, i.e., $\xj=0$, we also have $Q_{11}=O(n^{-1})$. Therefore,
	\begin{align}\label{eq:E11bd}
		\EE Q_{11}\lesssim b_n\ind(\xj>0)+\frac{1}{n}.
	\end{align}
	
	\medskip
	
	\noindent\textbf{Bounding $Q_{12}$.} Similarly, since $\min\{R_i,R_{N(i)}\}\le n$, we have
	\begin{align*}
		Q_{12}
		:=&\dfrac{n}{(n^2-1)^2}\sum_{i=1}^n\EE\bigg[\int \min\{R_i,R_{N(i)}\}(f_{\bY}(\by)-f_{\bY|\bX}(\by))d\by\times\\
		&\hspace{60pt}\times \int \min\{R_i,R_{N(i)}\}(f_{\bY}(\by)+f_{\bY|\bX}(\by))d\by\bigg]
		\\\le & \, 2n^{-1}\EE \abs*{\int \min\{R_1,R_{N(1)}\}(f_{\bY|\bX}(\by)-f_{\bY}(\by))d\by}
		\\\le &\, 2\abs*{\EE \int \ind(y_3\le (y_1, y_2))\left(f_{Y|X}(y_3|X_3)f_{Y|X}(y_1|X_1)f_{Y|X}(y_2|X_{N(1)})-\prod_{i=1}^3f_Y(y_i)\right)dy_1dy_2dy_3}
		\\\le &2\abs*{\EE \int \ind(y_3\le (y_1, y_2))\left(f_{Y|X}(y_1|X_1)f_{Y|X}(y_2|X_{N(1)})-f_Y(y_1)f_Y(y_2)\right)dy_1dy_2f_Y(y_3)dy_3}
		\\\le &2\abs*{\EE \int \ind(y_3\le (y_1, y_2))\left(f_{Y|X}(y_1|X_1)f_{Y|X}(y_2|X_{N(1)})-f_Y(y_1)f_Y(y_2)\right)dy_1dy_2f_Y(y_3)dy_3}
		\\&+2\abs*{\EE\int \ind(y_3\le y_1)f_{Y|X}(y_1|X_1)dy_1\left(F_{X,Y}(y_3|X_1)-F_{X,Y}(y_3|X_1)\right)f_Y(y_3)dy_3}
		\\\le &2\abs*{\EE\int \ind(y_3\le (y_1, y_2))\left(f_{Y|X}(y_1|X_1)-f_Y(y_1)\right)\left(f_{Y|X}(y_2|X_1)-f_Y(y_2)\right)dy_1dy_2f_Y(y_3)dy_3}
		\\&+2\EE\left[ \min\left\{(1+L_1^{(n)}(X_1,Y_2)+L_2^{(n)}(X_{N(1)},Y_2))|X_1-X_{N(1)}|^{\eta},1\right\}\right]
		\\\lesssim & \int\EE \left(\int 
		\ind(y_3\le y_1)
		\left(f_{Y|X}(y_1|X_2)-f_Y(y_1)\right)dy_1\right)^2f_Y(y_3)dy_3
		\\&\hspace{100pt}+n^{-\frac{\gamma}{\gamma+1}}(\log{n})^2+\left(\frac{(\log{n})^2}{n}\right)^{\left(\frac{\gamma(\theta-1)}{\theta(\gamma+1)}\wedge \frac{\eta \gamma}{\gamma+1}\right)}
		\\\lesssim &~ \xj+n^{-\frac{\gamma}{\gamma+1}}(\log{n})^2+\left(\frac{(\log{n})^2}{n}\right)^{\left(\frac{\gamma(\theta-1)}{\theta(\gamma+1)}\wedge \frac{\eta \gamma}{\gamma+1}\right)}.
	\end{align*}
	In the above, we have used the fact that $\int f_{Y|X}(y|x)f_X(x)dx=f_Y(y)$ in the third and fifth inequalities. The last line uses Cauchy-Schwarz inequality for the first term and Assumption (A1) in the main paper, coupled with equations \eqref{eq:usiden4}-\eqref{eq:usiden6} for the second term. 
	
	On the other hand, notice that when $\xj=0$, $X_1$ and $Y_1$ are independent and $Q_{12}=0$ from definition. Thus
	\begin{equation}\label{eq:E12bd}
		Q_{12}\lesssim \xj +b_n\ind(\xj>0)
	\end{equation}
	Plugging in \eqref{eq:E11bd} and \eqref{eq:E12bd} into \eqref{eq:T1} yields
	$$
	\abs*{\EE T_1-\dfrac{1}{18}}\lesssim n^{-1}+\xj+b_n\ind(\xj>0).
	$$
	A similar set of computations can be used in $\EE T_i$, $i=2,3,4$. We skip the relevant algebraic details for	brevity. We present the corresponding conclusions below:
	$$
	\max\{\EE\abs*{T_1-\tfrac{1}{18}},\EE\abs*{T_2-\tfrac{1}{45}},\EE\abs*{T_3-\tfrac{4}{45}},\EE\abs*{T_4+\tfrac{2}{45}}\} \lesssim n^{-1}+\xj+b_n\ind(\xj>0).
	$$
	By the definitions of $Q_1,\dots,Q_5$ this finishes the proof of \cref{lem:evxierr}.
	\qed
	
	\subsection{Proof of Lemma \ref{lem:varex}}\label{sec:pfvarex} As before, we hide the dependence on $n$ and write $\jtd$, $f_X$, $f_Y,$ $F_Y$ to mean $\jtd^{(n)}$, $f_X^{(n)}$, $f^{(n)}_Y$, $\F$ respectively. We will write $(X_1,Y_1),\dots,(X_n,Y_n)$ to mean the random variables $(X_{n,1},Y_{n,1}),\ldots,(X_{n,n},Y_{n,n})$ for some fixed $n$ in the triangular array.
	
	\medspace
	
	\textbf{Bounding $\mathrm{Var}(\E[\sqrt{n}\xi^{\stp}_n|\bX^{(n)}])$.}
	
	We will start this proof with some definitions:
	$$H_1:=\frac{6n}{n^2-1}\sum_{i=1}^n \int \min\{F_Y(y_1),F_Y(y_2)\}f_{Y|X}(y_1|X_i)f_{Y|X}(y_2|X_{N(i)})\,dy_1\,dy_2,$$
	$$H'_1:=\frac{6n}{n^2-1}\sum_{i=1}^n \int \min\{F_Y(y_1),F_Y(y_2)\}f_{Y|X}(y_1|X_i)f_{Y|X}(y_2|X_{i})\,dy_1\,dy_2,$$
	$$H_2:=\frac{6}{n^2-1}\underset{i\neq j}{\sum\sum} \int \min\{F_Y(y_1),F_Y(y_2)\}f_{Y|X}(y_1|X_i)f_{Y|X}(y_2|X_{j})\,dy_1\,dy_2.$$
	Based on the above notation, note that:
	\begin{equation}\label{eq:vofe1}
		\mathrm{Var}(\E[\sqrt{n}\xi^{\stp}_n|\bX^{(n)}])=n\E(H_1-H_2-\E H_1+\E H_2)^2.
	\end{equation}
	Using~\eqref{eq:vofe1}, it is easy to see that:
	\begin{equation}\label{eq:vofe2}
		\mathrm{Var}(\E[\sqrt{n}\xi^{\stp}_n|\bX^{(n)}])\lesssim n\E(H_1-H'_1-\E (H'_1- H_2))^2+n\mathrm{Var}(H'_1-H_2).
	\end{equation}
	We will now bound the two terms on the right hand side of~\eqref{eq:vofe2} separately.
	
	We start with the first term. Towards this direction, define $Z_n:=\sqrt{n}(H_1-H'_1)$. The first term on the right hand side of~\eqref{eq:vofe2} then equals $\mathrm{Var}(Z_n)$. Next observe that:
	\begin{align}\label{eq:vofe3}
		\E |Z_n| &\leq \frac{1}{\sqrt{n}}\sum_{i=1}^n \E\int \P(Y\geq y|X_i)|\P(Y\geq y|X_i)-\P(Y\geq y|X_{N(i)})|\myd(y)\,dy\nonumber \\ &\leq \frac{1}{\sqrt{n}}\sum_{i=1}^n \E\min\left\{\int (1+L_1^{(n)}(X_i,y)+L_2^{(n)}(X_{N(i)},y))|X_i-X_{N(i)}|^{\eta}\myd^{(n)}(y)\,dy,1\right\} \nonumber \\ &\leq \sqrt{n}b_n,
	\end{align}
	where $b_n$ is defined as in~\eqref{eq:xi_varcovs} and the last line follows using similar computations as in~\eqref{eq:usiden6}~and~\eqref{eq:usiden3}. Note that, if $X_1$ and $Y_1$ are independent, then $Z_n=0$, i.e.,
	\begin{equation}\label{eq:vofe7}
		\xj=0 \quad \implies \quad Z_n=0.
	\end{equation}
	Combining~\eqref{eq:vofe7} with~\eqref{eq:vofe3}, we get:
	\begin{align}\label{eq:vofe4}
		\E |Z_n|\lesssim \sqrt{n}b_n\ind(\xj>0).
	\end{align}
	The above gives us a bound on the first moment of $|Z_n|$. However, we are interested in the second moment. To make this transition, we will use Lemma~\ref{lem:expocon}. 
	
	By~\cref{lem:expocon}, there exists $K>0$, a positive constant, such that 
	\begin{equation}\label{eq:vofe5}
		\P\left(|Z_n-\E Z_n|\geq K\sqrt{\log{n}}\right)\lesssim \frac{1}{n^6}.
	\end{equation}
	
	Choose such a $K>0$ satisfying~\eqref{eq:vofe5} and consider the following decomposition:
	\begin{align}\label{eq:vofe6}
		&\;\;\;\E(Z_n-\E Z_n)^2\nonumber \\&=\E\left[(Z_n-\E Z_n)^2\ind(|Z_n-\E Z_n|\geq K\sqrt{\log{n}})\right]+\E\left[(Z_n-\E Z_n)^2\ind(|Z_n-\E Z_n|\leq K\sqrt{\log{n}})\right].
	\end{align}
	
	For the first term, we begin by observing that $\E(Z_n-\E Z_n)^4\lesssim 1$ by~\cref{lem:expocon}. Consequently by the Cauchy-Schwarz inequality,
	\begin{align}\label{eq:vofe8}
		&\;\;\;\;\E\left[(Z_n-\E Z_n)^2\ind(|Z_n-\E Z_n|\geq K\sqrt{\log{n}})\right]\nonumber \\&\leq \sqrt{\E(Z_n-\E Z_n)^4}\sqrt{\P(|Z_n-\E Z_n|\geq K\sqrt{\log{n}})}\lesssim n^{-3}
	\end{align}
	where the last line follows from~\eqref{eq:vofe5}. 
	
	For the second term, note that
	\begin{equation}\label{eq:vofe9}
		\E\left[(Z_n-\E Z_n)^2\ind(|Z_n-\E Z_n|\leq K\sqrt{\log{n}})\right]\leq K\sqrt{\log{n}}\E|Z_n|\lesssim b_n\sqrt{n\log{n}},
	\end{equation}
	where the last line follows from~\eqref{eq:vofe3}. Combining~\eqref{eq:vofe9},~\eqref{eq:vofe8}~and~\eqref{eq:vofe7}, we then get:
	\begin{equation}\label{eq:vofe10}
		n\E(H_1-H'_1-E (H'_1- H_2))^2\lesssim \left(b_n\sqrt{n\log{n}}+n^{-3}\right)\ind(\xj>0),
	\end{equation}
	which provides a bound for the first term in~\eqref{eq:vofe2}.
	
	For the second term in~\eqref{eq:vofe2}, we can use the same technique that we used to bound the $S_i$'s in the proof of~\cref{lem:evxsterr}, to get:
	\begin{equation}\label{eq:vofe11}
		n\mathrm{Var}(H'_1-H_2)\lesssim \xj+n^{-1}.
	\end{equation}
	Finally, combining~\eqref{eq:vofe10}~and~\eqref{eq:vofe11} completes the proof of Part 1.
	
	\medspace
	
	\textbf{Bounding $\mathrm{Var}(\E[\sqrt{n}\xi^{\prime}_n|\bX^{(n)}])$.} Recall that 
	\begin{equation}\label{eq:base1}
		\min{(R_i,R_{N(i)})}=1+\sum_{j\notin\{i,N(i)\}} \ind\left(Y_j\leq \min{(Y_i,Y_{N(i)})}\right).
	\end{equation}
	We will decompose $f_{Y|X}(y|x)=f_Y(y)+f_{Y|X}(y|x)-f_Y(y).$ Then using~\eqref{eq:base1} in the expression of $\xi'_n$, we get:
	\begin{align}\label{eq:base2}
		&\EE[\sqrt{n}\xi'_n|\bX^{(n)}]/6 \nonumber 
		\\=&
		\frac{\sqrt{n}}{n^2-1}\sum_{\substack{i=1\\ j\neq i,N(i)}}^n \int \ind (y_3\le\min\{y_1,y_2\})\jtdc(y_1|X_i)\jtdc(y_2|X_{N(i)})\jtdc(y_3|X_j)dy_1dy_2dy_3+O(\tfrac{1}{n})\nonumber 
		\\=&\frac{\sqrt{n}}{n^2-1}\sum_{i=1}^n \sum_{j\neq \{i,N(i)\}}\Bigg[ \int \ind(y_3\leq \min(y_1,y_2))f_Y(y_1)f_Y(y_2)f_Y(y_3)\,dy_1\,dy_2\,dy_3\nonumber 
		\\ &+\int \ind(y_3\leq \min(y_1,y_2))(f_{Y|X}(y_1|X_i)-f_Y(y_1))f_{Y|X}(y_2|X_{N(i)})f_{Y|X}(y_3|X_j)\,dy_1\,dy_2\,dy_3\nonumber  
		\\ &+\int \ind(y_3\leq \min(y_1,y_2))f_Y(y_1)(f_{Y|X}(y_2|X_{N(i)})-f_Y(y_2))f_{Y|X}(y_3|X_j)\,dy_1\,dy_2\,dy_3\nonumber 
		\\ &+\int \ind(y_3\leq \min(y_1,y_2))f_Y(y_1)f_Y(y_2)(f_{Y|X}(y_3|X_j)-f_Y(y_3))\,dy_1\,dy_2\,dy_3
		\Bigg]+O(n^{-1}).
	\end{align}
	Note that the first term within the sum is just $\PP(Y_3\le\min\{Y_1,Y_2\})+O(n^{-1})=\frac{1}{3}+O(n^{-1}).$ 
	
	Among the rest, let us define
	\begin{align*}
		&T_1\\
		:=&\frac{\sqrt{n}}{n^2-1}\sum_{i=1}^n\sum_{j\neq i,N(i)}\int \ind(y_3\leq \min(y_1,y_2))(f_{Y|X}(y_1|X_i)-f_Y(y_1))\\
		&\hspace{100pt}f_{Y|X}(y_2|X_{N(i)})f_{Y|X}(y_3|X_j)\,dy_1\,dy_2\,dy_3
		\\=&\frac{\sqrt{n}}{n^2-1}\sum_{i=1}^n\sum_{j\neq i,N(i)}\int (F_Y(y)-F_{Y|X}(y|X_i))(1-F_{Y|X}(y|X_{N(i)}))f_{Y|X}(y|X_j)dy
		\\=:&\frac{\sqrt{n}}{n^2-1}\sum_{i=1}^n\sum_{j\neq i,N(i)}\int (F_Y(y)-F_{Y|X}(y|X_i))(1-F_{Y|X}(y|X_{i}))f_{Y|X}(y|X_j)dy+T_1'(\bX^{(n)})
		\\=:&\frac{\sqrt{n}}{(n^2-1)}\sum_{i=1}^n\sum_{j\neq i, N(i)}\mathfrak{F}(X_i,X_j)+T_1'
		\\=:&T_1^*+T_1'.\numberthis\label{eq:var_T1}
	\end{align*}
	%where $A_1'(\bX)=\frac{1}{n^2-1}\sum_{i=1}^n\sum_{j\neq i,N(i)}\int (F(y_3)-F(y_3|X_i))(F(y_3|X_{i})-F(y_3|X_{N(i)}))f(y_3|X_j)dy_3.$ 
	Note that $\mathfrak{F}(X_i,X_j)\le 1.$ Then
	\begin{align*}
		&\Var(T_1^*) \\
		=&\Var\left(\frac{\sqrt{n}}{n^2-1}\sum_{i=1}^n\sum_{j\neq i,N(i)}\int (F_Y(y)-F_{Y|X}(y|X_i))(1-F_{Y|X}(y|X_{i}))f_{Y|X}(y|X_j)dy\right)
		\\=&\frac{n}{(n^2-1)^2}\sum_{i_1=1}^n\sum_{i_2=1}^n\sum_{j_1\neq i_1,N(i_1) }\sum_{j_2\neq i_2,N(i_2)}\Cov\left(\mathfrak{F}(X_{i_1},X_{j_1}),\,\mathfrak{F}(X_{i_2},X_{j_2})\right)
		\\=&\frac{n}{(n^2-1)^2}\underset{i_1\neq i_2}{\sum\sum}\sum_{j\neq i_1,N(i_1),i_2,N(i_2)}\Cov\left(\mathfrak{F}(X_{i_1},X_{j}),\,\mathfrak{F}(X_{i_2},X_{j})\right)
		\\&+\frac{n}{(n^2-1)^2}\sum_{i=1}^n\underset{j_1\neq j_2\neq i,N(i)}{\sum\sum}\Cov\left(\mathfrak{F}(X_{i},X_{j_1}),\,\mathfrak{F}(X_{i},X_{j_2})\right)+O(n^{-2})\numberthis\label{eq:varT1*1}
		%\\\lesssim & \,\,\EE\left(\mathfrak{F}(X_1,X_3)\mathfrak{F}(X_2,X_3)\right)+\EE\left(\mathfrak{F}(X_1,X_2)\mathfrak{F}(X_1,X_3)\right)
	\end{align*}
	where the last step follows by the independence of $X_1,\dots,X_n$. For any $1\le i_1,i_2\le n$ and $j\neq i_1,N(i_1),i_2,N(i_2)$
	\begin{align*}
		&\EE\left(\mathfrak{F}(X_{i_1},X_{j})\mathfrak{F}(X_{i_2},X_{j})\right)
		\\=&\int \int \left[\int (F_Y(y_1)-F_{Y|X}(y_1|X_{i_1}))(1-F_{Y|X}(y_1|X_{i_1}))f_X(x_{i_1})dx_{i_1}\right]f_{Y|X}(y_1|X_j)dy_1\times \\
		&\times \int \left[\int (F_Y(y_2)-F_{Y|X}(y_2|X_{i_2}))(1-F_{Y|X}(y_2|X_{i_2}))f_X(x_{i_2})dx_{i_2}\right]f_{Y|X}(y_2|X_j)dy_1f_X(x_j) dx_j
		%\\\le & \left[\int \left\{\int \left[\int (F(y_1)-F(y_1|X_{i_1}))(1-F(y_1|X_{i_1}))d\mu(x_{i_1})\right]f(y_1|X_j)dy_1\right\}^2d\mu(x_j)\int\left\{\int \dots \right\}^2d\mu(x_j)\right]^{1/2}
		\\\le &\int \int \left[\int (F_Y(y)-F_{Y|X}(y|X))(1-F_{Y|X}(y|X))f_X(x)dx\right]^2f_{Y|X}(y|X_j)dyf_X(x_j)dx_j
		\\\le &\int \left[\int (F_Y(y)-F_{Y|X}(y|X))^2(1-F_{Y|X}(y|X))^2f_X(x)dx\right]f_Y(y)dy\\\le &\,\xj.\numberthis\label{eq:varT1*2}
	\end{align*}
	Here we have used Cauchy-Schwarz inequality in the second and third inequalities. The exact same calculation also implies that 
	$$
	\EE\left(\mathfrak{F}(X_{i},X_{j_1})\mathfrak{F}(X_{i},X_{j_2})\right)\le \xj.
	$$
	By \eqref{eq:varT1*1} and \eqref{eq:varT1*2} we have
	\begin{equation}\label{eq:varT1*}
		\Var(T_1^*)\lesssim \xj+n^{-2}.
	\end{equation}
	On the other hand, 
	\begin{align*}
		&\EE \abs*{T_1'}\\
		=&\dfrac{\sqrt{n}}{n^2-1}\EE\sum_{i=1}^n	\sum_{j\neq i,N(i)}\int \abs*{(F_Y(y)-F_{Y|X}(y|X_i))(F_{Y|X}(y|X_{N(i)})-F_{Y|X}(y|X_{i}))}f_{Y|X}(y|X_j)dy\\
		\lesssim &\dfrac{\sqrt{n}}{n+1}\sum_{i=1}^n\EE\int   \abs*{\PP(Y\ge y|X_i)-\PP(Y\ge y|X_{N(i)})}f_Y(y)dy
		\lesssim \sqrt{n}b_n \numberthis\label{eq:vofe3'}
	\end{align*}
	following \eqref{eq:vofe3}. Note that if $X_1$ and $Y_1$ are independent then $T_1'(\bX^{(n)})=0$, i.e.,
	\begin{equation}\label{eq:vofe7'}
		\xj=0\implies T_1'=0
	\end{equation}
	which together with \eqref{eq:vofe3'} implies 
	\begin{equation}\label{eq:vofe4'}
		\EE \abs*{T_1'}\lesssim \sqrt{n}b_n\ind(\xj>0). 
	\end{equation}
	As in part 1 of the lemma, we make a transition to the second moment via McDiarmid's inequality. Notice that just as in \cref{lem:expocon}, we have a constant $C>0$ such that for any $n\ge 1$ and $t\ge 0$
	$$
	\PP(\abs*{T_1'(\bX^{(n)})-\EE T_1'(\bX^{(n)})}\ge t)\le 2\exp(-Cnt^2).
	$$
	In particular, there exists a constant $K>0$ such that
	\begin{equation}\label{eq:vofe5'}
		\PP(\abs*{T_1'-\EE T_1'}\ge K\sqrt{\log n})\lesssim \frac{1}{n^6}.
	\end{equation}
	One can then follow \eqref{eq:vofe6}-\eqref{eq:vofe9} to get
	\begin{align*}
		&\EE(T_1'-\EE T_1')^2\\
		=&\EE\left[(T_1'-\EE T_1')^2\ind(\abs*{T_1'-\EE T_1'}\ge K\sqrt{\log n})\right]+\EE\left[(T_1'-\EE T_1')^2\ind(\abs*{T_1'-\EE T_1'}\le K\sqrt{\log n})\right]\\
		\le & \sqrt{\EE(T_1'-\EE T_1')^4}\sqrt{\PP(\abs*{T_1'-\EE T_1'}\ge K\sqrt{\log n})}+K\sqrt{\log n}\EE\abs*{Z_n}\\
		\lesssim & \,\,n^{-3}+b_n\sqrt{n\log n}. \numberthis \label{eq:vofe9'}
	\end{align*}
	which, together with \eqref{eq:vofe7'} implies that
	\begin{equation}\label{eq:varT1'}
		\EE(T_1'-\EE T_1')^2\lesssim \left(n^{-3}+b_n\sqrt{n\log n}\right)\ind(\xj>0).
	\end{equation}
	Combining \eqref{eq:varT1*}, \eqref{eq:varT1'} and \eqref{eq:var_T1}, we have
	\begin{equation}
		\Var(T_1)\le \xj+\left(n^{-3}+b_n\sqrt{n\log n}\right)\ind(\xj>0)+n^{-2}.
	\end{equation}
	Using a similar calculation, we have the same bound for the variance of the third and fourth terms on the right-hand side of \eqref{eq:base2}. This finishes the proof of the second bound, and hence \cref{lem:varex} follows.
	\qed
	
	\subsection{Proof of Lemma \ref{lem:interaction_rule}}\label{sec:pfinteraction_rule} 
	Let us consider $W_\ell$ defined in \eqref{eq:def_w_l} and observe that $W_\ell(\bm x)$  depends on $(x_\ell,y_\ell)$ and $(x_{N(\ell)},y_{N(\ell)})$. Consider $\bm x, \bm x^\prime \in (\mathbb{R}^2)^n$ and $\{i,j\} \in [n]\times[n]$, such that the edge $\{i,j\}$ does not exist in $\mathcal{G}(\bm x), \mathcal{G}(\bm x^{i}), \mathcal{G}(\bm x^{j})$ and $\mathcal{G}(\bm x^{ij})$. We shall show that for all $\ell \in [n]$,
	\begin{equation}
		\label{eq:balance_cond}
		W_\ell(\bm x)-W_\ell(\bm x^{i})-W_\ell(\bm x^{j})+W_\ell(\bm x^{ij})=0.
	\end{equation}
	By \eqref{eq:def_d}, we have
	\begin{equation}
		\label{eq:eq_8}
		\Big|D_{\bm x}(i,j)-D_{\bm x^\prime}(i,j)\Big| \le \#\Big\{t: x_t \neq x^\prime _t\Big\}.
	\end{equation}
	Let us fix an $\ell \in [n]$ such that $D_{\bm x}(\ell,j) \le 1$. That means there exists at most one $k$, such that, $x_{\ell} < x_k < x_j$. We shall show
	\[
	W_\ell(\bm x)=W_\ell(\bm x^i) \quad \mbox{and} \quad W_\ell(\bm x^j)=W_\ell(\bm x^{ij}).
	\]
	As the edge $\{i,j\}$ is absent in $\mathcal{G}(\bm x)$, we have
	\begin{equation}
		\label{eq:eq_12}
		D_{\bm x}(\ell, i) > 2.
	\end{equation}
	In particular, $i$ is different from $\ell$ and $j$. Again, as the edge $\{i,j\}$ is absent in $\mathcal{G}(\bm x^i)$, we further have,
	\begin{equation}
		\label{eq:eq_13}
		D_{\bm x^{i}}(\ell, i) > 2.
	\end{equation}
	This implies $N(\ell)$ does not change in $\bm x$ and $\bm x^{i}$. Hence,
	\[
	W_\ell(\bm x)=W_\ell(\bm x^i).
	\]
	Next we show that if $D_{\bm x}(\ell,j) \le 1$ then we have $D_{\bm x^j}(\ell,i) \ge 2$ and $D_{\bm x^{ij}}(\ell,i) \ge 2$. Suppose not, let $D_{\bm x^j}(\ell,i) \le 1$. If $j = \ell$, then clearly this is false as the edge $\{i,j\}$ is absent in $\mathcal{G}(\bm x^j)$. If $j \neq \ell$, then by \eqref{eq:eq_8} and \eqref{eq:eq_12}, we get,
	\[
	D_{\bm x^j}(\ell,i) \ge 2.
	\]
	This is a contradiction. Further, if $D_{\bm x^{ij}}(\ell,i) \le 1$ and  $j = \ell$, then this is similarly false as the edge $\{i,j\}$ edge is absent in $\mathcal{G}(\bm x^{ij})$. If $j \neq \ell$, then by \eqref{eq:eq_8} and \eqref{eq:eq_13}, we get,
	\[
	D_{\bm x^{ij}}(\ell,i) \ge 2.
	\]
	Again we get a contradiction. As $D_{\bm x^j}(\ell,i) \ge 2$ and $D_{\bm x^{ij}}(\ell,i) \ge 2$, $N(\ell)$ is same in $\bm x^{j}$ and $\bm x^{ij}$. Hence,
	\[
	W_\ell(\bm x^j)=W_\ell(\bm x^{ij}).
	\]
	This implies if $D_{\bm x}(\ell,j) \le 1$, then \eqref{eq:balance_cond} holds. Similarly if $D_{\bm x^{i}}(\ell,j) \le 1$ or $D_{\bm x^j}(\ell,j) \le 1$ or $D_{\bm x^{ij}}(\ell,j) \le 1$, then \eqref{eq:balance_cond} holds. Now if $D_{\bm x}(\ell,j), D_{\bm x^{i}}(\ell,j), D_{\bm x^{j}}(\ell,j), D_{\bm x^{ij}}(\ell,j) >1$, then $N(\ell)$ does not change in $\bm x$ and $\bm x^{j}$ implying 
	\[
	W_\ell(\bm x)=W_\ell(\bm x^j).
	\]
	Also $N(\ell)$ does not change in $\bm x^i$ and $\bm x^{ij}$ implying 
	\[
	W_\ell(\bm x^i)=W_\ell(\bm x^{ij}).
	\]
	This implies the lemma.	
	\qed
	
	\subsection{Proof of Lemma A.1}\label{sec:pfbias_rems}

	To reduce notation we write $\jtd(\cdot,\cdot)$, $\mxd(\cdot)$, $\myd(\cdot)$ and $F_Y(\cdot)$ to mean $\jtd^{(n)}(\cdot)$, $\mxd^{(n)}(\cdot)$, $\myd^{(n)}(\cdot)$ and $F_Y^{(n)}(\cdot)$ respectively. We will write $(X_1,Y_1),\dots,(X_n,Y_n)$ to mean the random variables $(X_{n,1},Y_{n,1}),\ldots,(X_{n,n},Y_{n,n})\sim \jtd^{(n)}$ from the triangular array. We will also write $\bX$ to mean $(X_1,\dots,X_n)$. It is useful to note some properties of $g(y|x)$ which we will use throughout the proof.
	\begin{itemize}
		\item[a)] \begin{equation}\label{eq:g_cond_den}
			\int g(y|x) dy=1\quad \text{and}\quad \int g(y|x)f_X(x)dx=f_Y(y).
		\end{equation}
		In this sense, $g(y|x)$ is like a conditional density, although it can take negative values.
		\item[b)] Recalling the definition of $g$, we write 
		\begin{equation}\label{eq:g_cond_CDF2}
			G(y|x)=\int _{-\infty}^yg(t|x)dt=\tfrac{1}{p}\jtd(y|x)+(1-\tfrac{1}{p})F_Y(y);\quad\quad \int G(y|x)f_X(x)dx=F_Y(y).
		\end{equation}
		\item[c)] For any $y,\,x_1,\,x_2$, Assumption (A1) in the main paper implies there exist $\eta\in(0,1],$ $\theta>1$ and $C>0$ such that
		\begin{equation}\label{eq:alt_lipschitz}
			\begin{split}
				\abs*{G(y|x_1)-G(y|x_2)}&\le (1+L_1^{(n)}(x_1,y)+L_2^{(n)}(x_2,y))\abs*{x_1-x_2}^{\eta}
				\\
				\underset{n\to \infty}{\lim\sup}\int & (L_1^{(n)}(x,y))^{\theta}f_X(x)f_Y(y)dy\le C.
			\end{split}
		\end{equation}
	\end{itemize}
	We will prove the three parts of Lemma A.1 as \cref{lem:T1,lem:T2,lem:T3} presented below. 
	
	\begin{lemma}
		\label{lem:T1}
		$\abs*{T_1-1}\lesssim \left(n^{-\frac{\gamma}{\gamma+1}}(\log{n})^2+\left(\dfrac{(\log n)^2}{n}\right)^{\left(\tfrac{\gamma(\theta-1)}{\theta(\gamma+1)}\wedge \tfrac{\eta\gamma}{\gamma+1}\right)}\right)\ind(\xj>0)$.
	\end{lemma}
	
	\subsubsection{Proof of Lemma \ref{lem:T1}}\label{sec:pfT1}

	$T_1$ involves two contributions from $f_Y(y)$ and one from $g(y|x)$. Thus
	\begin{equation}
		\begin{split}
			T_1&=\int \sum_{j\neq 1}\ind(N(1)=j) \ind(y_k<\min\{y_1,y_j\})f_Y(y_1)f_Y(y_k)g(y_j|x_j)dy_1dy_kdy_jf_{\bX}(\bx)\,d\bx
			\\&+\int \sum_{j\neq 1}\ind(N(1)=j) \ind(y_k<\min\{y_1,y_j\})f_Y(y_1)g(y_k|x_k)f_Y(y_j)dy_1dy_kdy_j f_{\bX}(\bx)\,d\bx
			\\&+\int \sum_{j\neq 1}\ind(N(1)=j) \ind(y_k<\min\{y_1,y_j\})g(y_1|x_1)f_Y(y_k)f_Y(y_j)dy_1dy_kdy_j f_{\bX}(\bx)\,d\bx.
		\end{split}
	\end{equation}
	Now for the first term of $T_1$, we use Fubini's theorem to write
	\begin{align*}&\int \sum_{j\neq 1}\ind(N(1)=j) \ind(y_k<\min\{y_1,y_j\})f_Y(y_1)f_Y(y_k)g(y_j|x_j)dy_1dy_kdy_j f_{\bX}(\bx)\,d\bx
		\\=&\int\sum_{j\neq 1}\ind(N(1)=j) \ind(y_k<\min\{y_1,y_j\})(g(y_j|x_1)+g(y_j|x_j)-g(y_j|x_1))f_{\bX}(\bx)\,d\bx
		\\& \hspace{30pt}f_Y(y_k)f_Y(y_1)dy_1dy_kdy_j
		\\=:&\int\ind(y_k<\min\{y_1,y\})\int g(y|x_1)f_X(x_1)dx_1f_Y(y_1)f_Y(y_k)dydy_1dy_k+E_n
		\\=&\int \ind(y_k<\min\{y_1,y\})f_Y(y)f_Y(y_1)f_Y(y_k)dydy_1dy_k+E_n=\dfrac{1}{3}+E_n,
	\end{align*}
	where we use the fact that $\int g(y|x)f_X(x)dx=f_Y(y)$ from \eqref{eq:g_cond_den}. The remainder term is
	\begin{align*}\label{eq:rem_bound1}
		E_n=&\int \sum_{j\neq 1}\ind(N(1)=j) \int (g(y_j|x_j)-g(y_j|x_1)) f_Y(y_1)f_Y(y_k)dy_1dy_kdy_j f_{\bX}(\bx)\,d\bx
		\\= &\int \int\limits_{y_k}^\infty \int \sum_{j\neq 1}\ind(N(1)=j) \int\limits_{y_k}^\infty (g(y_j|x_j)-g(y_j|x_1)) f_Y(y_k)f_Y(y_1)dy_1dy_kdy_j f_{\bX}(\bx)\,d\bx
		\\=&\int  \int\limits_{y_k}^\infty \int \sum_{j\neq 1}\ind(N(1)=j)[G_{Y|X_1=x_1}(y_k)-G_{Y|X_{j}=x_j}(y_k)] 
		f_Y(y_k)f_Y(y_1)dy_1dy_k f_{\bX}(\bx)\,d\bx
		\\\le &\int  \int\limits_{y_k}^\infty \int \abs*{G_{Y|X_1}(y_k)-G_{Y|X_{N(1)}}(y_k)}f_{\bX}(\bx)\,d\bx
		f_Y(y_k)f_Y(y_1)dy_1dy_k
		\\\le &\int \int\{1\wedge (1+L_1^{(n)}(x_1,y_k)+L_2^{(n)}(x_{N(1)},y))|x_1-x_{N(1)}|^{\eta},1\} f_{\bX}(\bx)\,d\bx(1-F_Y(y_k))f_Y(y_k)dy_k
		\\\lesssim & \,\P(|X_1-X_{N(1)}|\geq 1)\\
		&~+\left\{\E\left(1+(L_1^{(n)}(X_1,Y_k))^{\theta}+(L_2^{(n)}(X_1,Y_k))^{\theta}\right)\right\}^{\frac{1}{\theta}}
		\left\{\E|X_1-X_{N(1)}|^{\frac{\theta\eta}{\theta-1}}\ind(|X_1-X_{N(1)}|\leq 1)\right\}^{\frac{\theta-1}{\theta}}
		\\\lesssim & \,\,n^{-\frac{\gamma}{\gamma+1}}(\log{n})^2+\left(\frac{(\log{n})^2}{n}\right)^{\left(\frac{\gamma(\theta-1)}{\theta(\gamma+1)}\wedge \frac{\eta \gamma}{\gamma+1}\right)},\numberthis
	\end{align*}
	where we use \eqref{eq:alt_lipschitz} in the second last line, and  \cref{lem:nndist} in the last line.	By analogous calculation, the second and third terms in $T_1$ can be written as
	$$\begin{aligned}
		&\int  \ind(y_k<\min\{y_1,y\})f_Y(y)f_Y(y_1)f_Y(y_k)dydy_1dy_k+O\left(\left(\frac{(\log n)^2}{n}\right)^{\left(\tfrac{\gamma(\theta-1)}{\theta(\gamma+1)}\wedge\tfrac{\eta\gamma}{\gamma+1}\right)}\right) \\
		=&\dfrac{1}{3}+O\left(n^{-\frac{\gamma}{\gamma+1}}(\log{n})^2+\left(\frac{(\log{n})^2}{n}\right)^{\left(\frac{\gamma(\theta-1)}{\theta(\gamma+1)}\wedge \frac{\eta \gamma}{\gamma+1}\right)}\right).
	\end{aligned}
	$$
	Finally, when $\xj=0$ $X_i$ and $Y_i$ are independent, and hence $g(y|x)=f_Y(y)$ for all $x$. It can be checked that in this case $E_n=0$ and similarly the other error terms are all zero. This finishes the proof of \cref{lem:T1}, i.e., part i) of Lemma A.1.
	\qed
	\begin{lemma}\label{lem:T2}
		$$\abs*{T_2-\left(\dfrac{2}{3}+H(f,g)\right)}\lesssim \left(n^{-\frac{\gamma}{\gamma+1}}(\log{n})^2+\left(\dfrac{(\log{n})^2}{n}\right)^{\left(\frac{\gamma(\theta-1)}{\theta(\gamma+1)}\wedge \frac{\eta \gamma}{\gamma+1}\right)}\right)\ind(\xj>0).$$
	\end{lemma}
	\subsubsection{Proof of Lemma \ref{lem:T2}}\label{sec:pfT2}
	The second term deals with the case where one of the three
	is from $f_Y(y)$ and the other two are from $g(y|x)$:
	\begin{equation}
		\begin{split}
			T_2&=\int \sum_{j\neq 1}\ind(N(1)=j) \ind(y_k<\min\{y_1,y_j\})g(y_1|x_1)g(y_k|x_k)f_Y(y_j)dy_1dy_kdy_j f_{\bX}(\bx)\,d\bx
			\\&+\int  \sum_{j\neq 1}\ind(N(1)=j) \ind(y_k<\min\{y_1,y_j\})f_Y(y_1)g(y_k|x_k)g(y_j|x_j)dy_1dy_kdy_j f_{\bX}(\bx)\,d\bx
			\\&+\int \sum_{j\neq 1}\ind(N(1)=j) \ind(y_k<\min\{y_1,y_j\})g(y_1|x_1)f_Y(y_k)g(y_j|x_j)dy_1dy_kdy_j f_{\bX}(\bx)\,d\bx.
		\end{split}
	\end{equation}
	For the first term of $T_2$, by interchanging the integrals,
	\begin{align*}\label{eq:term_T2_1}
		&\int \sum_{j\neq 1}\ind(N(1)=j) \ind(y_k<\min\{y_1,y_j\})g(y_1|x_1)g(y_k|x_k)f_Y(y_j)dy_1dy_kdy_j f_{\bX}(\bx)\,d\bx
		\\=&\int  \sum_{j\neq 1}\ind(N(1)=j) \ind(y_k<\min\{y_1,y_j\})g(y_1|x_1)g(y_k|x_k)f_{\bX}(\bx)\,d\bx f_Y(y_j)dy_1dy_kdy_j
		\\=&\int \ind(y_k<\min\{y_1,y\})\left[\int g(y_1|x_1)f_X(x_1) dx_1 \int g(y_k|x_k)f_X(x_k) dx_k\right]f_Y(y_j)dy_1dy_kdy_j
		\\=&\int \ind(y_k<\min\{y_1,y\})f_Y(y_1) f_Y(y_k)f_Y(y_j)dy_1dy_kdy_j=\dfrac{1}{3}\numberthis
	\end{align*}
	where the last line follows from \eqref{eq:g_cond_den}.	The second term in $T_2$ is
	\begin{align*}\label{eq:term_T2_2}
		&\int \sum_{j\neq 1}\ind(N(1)=j) \ind(y_k<\min\{y_1,y_j\})f_Y(y_1)g(y_k|x_k)g(y_j|x_j)dy_1dy_kdy_j f_{\bX}(\bx)\,d\bx
		\\=&\int\sum_{j\neq 1}\ind(N(1)=j) \ind(y_k<\min\{y_1,y_j\})f_Y(y_1)g(y_k|x_k)\\
		&\hspace{100pt}(g(y_j|x_1)+g(y_j|x_j)-g(y_j|x_1)) f_{\bX}(\bx)\,d\bx dy_1dy_kdy_j
		\\=:&\int \ind(y_k<\min\{y_1,y\})\left[\int g(y_j|x_1)f_X(x_1) dx_1 \int g(y_k|x_k) f_X(x_k)dx_k\right]f_Y(y_1)dy_1dy_kdy_j+E'_n
		\\=&\int \ind(y_k<\min\{y_1,y\})f_Y(y_1) f_Y(y_k)f_Y(y_j)dy_1dy_kdy_j+E_n'=\dfrac{1}{3}+E_n'\numberthis
	\end{align*}
	where the last line follows from \eqref{eq:g_cond_den}.	Just as in equation \eqref{eq:rem_bound1}, the remainder term
	\begin{align*}\label{eq:rem_bound2}
		&E_n'
		\\=&\int\sum_{j\neq 1,k}\ind(N(1)=j) \ind(y_k<\min\{y_1,y_j\})(g(y_j|x_j)-g(y_j|x_1))g(y_k|x_k)d\mu(\bx) f_Y(y_1)dy_1dy_kdy_j
		\\=&\int \left[\int  \sum_{j\neq 1,k}\ind(N(1)=j) \ind(y_k<\min\{y_1,y_j\})(g(y_j|x_j)-g(y_j|x_1))f_{\bX}(\bx)\,d\bx\right]
		\\&\hspace{50pt}\times \left[\int g(y_k|x_k)f_X(x_k)dx_k\right] f_Y(y_1)dy_1dy_kdy_j
		\\=&E_n
		\\\lesssim &\,\,n^{-\frac{\gamma}{\gamma+1}}(\log{n})^2+\left(\frac{(\log{n})^2}{n}\right)^{\left(\frac{\gamma(\theta-1)}{\theta(\gamma+1)}\wedge \frac{\eta \gamma}{\gamma+1}\right)}.\numberthis
	\end{align*}
	We use the fact that $j\neq k$ in the second equality and~\cref{lem:nndist} in the last step. With a remainder term $E_n''$ defined similarly as $E_n$ and $E_n',$ we can write the third term as
	\begin{align*}
		&\int \sum_{j\neq 1}\ind(N(1)=j) \ind(y_k<\min\{y_1,y_j\})g(y_1|x_1)f_Y(y_k)g(y_j|x_j)dy_1dy_kdy_j f_{\bX}(\bx) \,d\bx
		\\=:&\int \ind(y_k<\min\{y_1,y\})\left[\int g(y_1|x_1) g(y|x_1)f_X(x_1) dx_1\right]f_Y(y_k)dy_1dy_kdy+E''_n
		\\=&\int\EE\left(\int_t^\infty g(y|X)dy\right)^2f_Y(t)dt+E_n''
		\\=&~H(f,g)+E_n''.	\numberthis\label{eq:term_T3_3}
	\end{align*}
	Indeed
	\begin{align*}\label{eq:rem_bound3}
		&E_n''
		\\=& \int \sum_{j\neq 1,k}\ind(N(1)=j) \ind(y_k<
		y_1\wedge y_j)(g(y_j|x_j)-g(y_j|x_1))g(y_1|x_1)f_{\bX}(\bx)d\bx f_Y(y_k)dy_1dy_kdy_j
		\\=&\int\sum_{j\neq 1,k}\ind(N(1)=j) \int\limits_{y_k}^{\infty}(g(y_j|x_j)-g(y_j|x_1))dy_j\int\limits_{y_k}^\infty g(y_1|x_1)dy_1 f_{\bX}(\bx)\,d\bx f_Y(y_k)dy_k
		\\=&\int\left[ \int \sum_{j\neq 1,k}\ind(N(1)=j) [G_{Y|X_1=x_1}(y_k)-G_{Y|X_j=x_j}(y_k)](1-G_{Y|X_1=x_1}(y_k)) f_{\bX}(\bx)\,d\bx\right] f_Y(y_k)dy_k
		\\=&\int \left[
		\int  [G_{Y|X_{1}}(y_k)-G_{Y|X_{N(1)}}(y_k)](1-G_{Y|X_1}(y_k)) f_{\bX}(\bx)\,d\bx\right] f_Y(y_k)dy_k
		\\\lesssim &~\,\P(|X_1-X_{N(1)}|\geq 1)\\
		&~+\left\{\E\left(1+(L_1^{(n)}(X_1,Y_k))^{\theta}+(L_2^{(n)}(X_1,Y_k))^{\theta}\right)\right\}^{\frac{1}{\theta}}
		\left\{\E|X_1-X_{N(1)}|^{\frac{\theta\eta}{\theta-1}}\ind(|X_1-X_{N(1)}|\leq 1)\right\}^{\frac{\theta-1}{\theta}}
		\\\lesssim & \,\,n^{-\frac{\gamma}{\gamma+1}}(\log{n})^2+\left(\frac{(\log{n})^2}{n}\right)^{\left(\frac{\gamma(\theta-1)}{\theta(\gamma+1)}\wedge \frac{\eta \gamma}{\gamma+1}\right)},\numberthis
	\end{align*}
	where the last step follows by mimicking the last line of \eqref{eq:rem_bound1}. Adding \eqref{eq:term_T2_1}-\eqref{eq:rem_bound3} proves \cref{lem:T2} when $\xj>0$.
	
	Once again $\xj=0$ means $g(y|x)=f_Y(y)$ for all $x$, which implies that $H(f,g)$, $E_n'$ and $E_n''$ are equal to zero.
	\qed
	
	\begin{lemma}\label{lem:T3}
		$$\abs*{T_3-H(f,g)}\lesssim  \,\,\left(n^{-\frac{\gamma}{\gamma+1}}(\log{n})^2+\left(\dfrac{(\log{n})^2}{n}\right)^{\left(\frac{\gamma(\theta-1)}{\theta(\gamma+1)}\wedge \frac{\eta \gamma}{\gamma+1}\right)}\right)\ind(\xj>0).$$
	\end{lemma}	
	
	\subsubsection{Proof of Lemma \ref{lem:T3}}\label{sec:pfT3}
	Finally in the third term, the contributions for all of $Y_k$; $Y_1$; $Y_{N(1)}$ are from $g(y|x)$:
	\begin{align*}
		T_3&=\int \sum_{j\neq 1,k}\ind(N(1)=j) \ind(y_k<\min\{y_1,y_j\})g(y_1|x_1)g(y_k|x_k)g(y_j|x_j)dy_1dy_kdy_j f_{\bX}(\bx)\,d\bx
		\\&=\int \sum_{j\neq 1}\ind(N(1)=j) \ind(y_k<\min\{y_1,y_j\})g(y_1|x_1)f_Y(y_k)g(y_j|x_j)dy_1dy_kdy_j f_{\bX}(\bx)\,d\bx
		\\&=H(f,g)+E_n''
		\\&=H(f,g)+O\left(n^{-\frac{\gamma}{\gamma+1}}(\log{n})^2+\left(\dfrac{(\log{n})^2}{n}\right)^{\left(\frac{\gamma(\theta-1)}{\theta(\gamma+1)}\wedge \frac{\eta \gamma}{\gamma+1}\right)}\right).\nonumber
	\end{align*}
	The first line is by definition of $T_3.$ The second line follows by integrating over $x_k$ (notice that $j\neq k$). The third and fourth lines follow by comparing to \eqref{eq:term_T3_3} and \eqref{eq:rem_bound3} respectively. Once again $\xj=0$ means $g(y|x)=f_Y(y)$ for all $x$, which implies that $H(f,g)$, $E_n''$ are equal to zero.
	\qed
	
	\cref{lem:T1,lem:T2,lem:T3} together prove Lemma A.1.
	\qed
	
	\subsection{Auxiliary lemmas}
	\begin{lemma}\label{lem:nndist}
		Under Assumption (A2) from the main paper, given any $\epsilon>0$, $$\P\left(|X_1-X_{N(1)}|\geq\epsilon\right)\lesssim n^{-\frac{\gamma}{\gamma+1}}+\frac{(\log{n})^2 n^{-\frac{\gamma}{\gamma+1}}}{\epsilon}.$$
		Consequently, for any $p\leq 1$,
		uxil	$$\E|X_1-X_{N(1)}|^p\ind(|X_1-X_{N(1)}|\leq 1)\lesssim \left(\frac{(\log{n})^2}{n}\right)^{p}.$$
	\end{lemma}
	
	\begin{proof}[Proof of \cref{lem:nndist}] 
		The proof follows by retracing the steps of \cite[Lemma 14.1]{azadkia2019simple} and a straightforward application of Lyapunov's inequality.
	\end{proof}
	
	\medskip
	
	\begin{lemma}\label{lem:expocon}
		There exists a fixed positive constant $C>0$ such that the following holds for any $n\geq 1$ and $t\geq 0$:
		$$\P\left(|Z_n-\E Z_n|\geq t\right)\leq 2\exp(-Cnt^2).$$
	\end{lemma}
	
	\begin{proof}[Proof of Lemma~\ref{lem:expocon}]
		We omit the proof of this lemma since it follows simply from McDiarmid's inequality (see~\cite{Mcdiarmid1989}) as used in~\cite[Lemma 9.11]{sc_corr}.    
	\end{proof}
	
\bibliographystyle{plain} 
\bibliography{SigNoise}

\begin{thebibliography}{10}

\bibitem{ansari2022simple}
Jonathan Ansari and Sebastian Fuchs.
\newblock A simple extension of azadkia $\&$ chatterjee's rank correlation to a
  vector of endogenous variables.
\newblock {\em arXiv preprint arXiv:2212.01621}, 2022.

\bibitem{arias2020detection}
Ery Arias-Castro, Rong Huang, and Nicolas Verzelen.
\newblock {Detection of sparse positive dependence}.
\newblock {\em Electron. J. Stat.}, 14(1):702 -- 730, 2020.

\bibitem{ADN21}
Arnab Auddy, Nabarun Deb, and Sagnik Nandy.
\newblock Exact detection thresholds and minimax optimality of {C}hatterjee's
  correlation.
\newblock 2023.

\bibitem{azadkia2019simple}
Mona Azadkia and Sourav Chatterjee.
\newblock A simple measure of conditional dependence.
\newblock {\em arXiv preprint arXiv:1910.12327}, 2019.

\bibitem{azadkia2021fast}
Mona Azadkia, Armeen Taeb, and Peter B{\"u}hlmann.
\newblock A fast non-parametric approach for causal structure learning in
  polytrees.
\newblock {\em arXiv preprint arXiv:2111.14969}, 2021.

\bibitem{bergsma2014consistent}
Wicher Bergsma and Angelos Dassios.
\newblock A consistent test of independence based on a sign covariance related
  to kendall’s tau.
\newblock {\em Bernoulli}, 20(2):1006--1028, 2014.

\bibitem{berrett2021optimal}
Thomas~B Berrett, Ioannis Kontoyiannis, and Richard~J Samworth.
\newblock Optimal rates for independence testing via u-statistic permutation
  tests.
\newblock {\em Ann. Stat.}, 49(5):2457--2490, 2021.

\bibitem{Biau2015}
G\'{e}rard Biau and Luc Devroye.
\newblock {\em Lectures on the nearest neighbor method}.
\newblock Springer Series in the Data Sciences. Springer, Cham, 2015.

\bibitem{bickel2022measures}
Peter~J Bickel.
\newblock Measures of independence and functional dependence.
\newblock {\em arXiv preprint arXiv:2206.13663}, 2022.

\bibitem{Blomqvist1950}
Nils Blomqvist.
\newblock On a measure of dependence between two random variables.
\newblock {\em Ann. Math. Stat.}, 21:593--600, 1950.

\bibitem{Blum1961}
J.~R. Blum, J.~Kiefer, and M.~Rosenblatt.
\newblock Distribution free tests of independence based on the sample
  distribution function.
\newblock {\em Ann. Math. Stat.}, 32:485--498, 1961.

\bibitem{cao2020correlations}
Sky Cao and Peter~J Bickel.
\newblock Correlations with tailored extremal properties.
\newblock {\em arXiv preprint arXiv:2008.10177}, 2020.

\bibitem{chatterjee2020insights}
Sompriya Chatterjee, Abbas Salimi, and Jin~Yong Lee.
\newblock Insights into amyotrophic lateral sclerosis linked pro525arg mutation
  in the fused in sarcoma protein through in silico analysis and molecular
  dynamics simulation.
\newblock {\em J. Biomol. Struct. Dyn.}, pages 1--14, 2020.

\bibitem{chatt_nn}
Sourav Chatterjee.
\newblock A new method of normal approximation.
\newblock {\em Ann. Probab.}, 36(4):1584--1610, 07 2008.

\bibitem{sc_corr}
Sourav Chatterjee.
\newblock A new coefficient of correlation.
\newblock {\em J. Am. Stat. Assoc.}, pages 1--21, 2020.

\bibitem{chatterjee2022survey}
Sourav Chatterjee.
\newblock A survey of some recent developments in measures of association.
\newblock {\em arXiv preprint arXiv:2211.04702}, 2022.

\bibitem{chatterjee2022estimating}
Sourav Chatterjee and Mathukumalli Vidyasagar.
\newblock Estimating large causal polytree skeletons from small samples.
\newblock {\em arXiv preprint arXiv:2209.07028}, 2022.

\bibitem{csorgHo1985testing}
S{\'a}ndor Cs{\"o}rg{\H{o}}.
\newblock Testing for independence by the empirical characteristic function.
\newblock {\em J. Multivar. Anal.}, 16(3):290--299, 1985.

\bibitem{deb2020kernel}
Nabarun Deb, Promit Ghosal, and Bodhisattva Sen.
\newblock Measuring association on topological spaces using kernels and
  geometric graphs.
\newblock {\em arXiv preprint arXiv:2010.01768}, 2020.

\bibitem{Deb2023}
Nabarun Deb and Bodhisattva Sen.
\newblock Multivariate rank-based distribution-free nonparametric testing using
  measure transportation.
\newblock {\em J. Amer. Statist. Assoc.}, 118(541):192--207, 2023.

\bibitem{dette2013copula}
Holger Dette, Karl~F Siburg, and Pavel~A Stoimenov.
\newblock A copula-based non-parametric measure of regression dependence.
\newblock {\em Scand. Stat. Theory Appl.}, 40(1):21--41, 2013.

\bibitem{Dhar2016}
Subhra~Sankar Dhar, Angelos Dassios, and Wicher Bergsma.
\newblock A study of the power and robustness of a new test for independence
  against contiguous alternatives.
\newblock {\em Electron. J. Stat.}, 10(1):330--351, 2016.

\bibitem{Drton2020}
Mathias Drton, Fang Han, and Hongjian Shi.
\newblock High-dimensional consistent independence testing with maxima of rank
  correlations.
\newblock {\em Ann. Stat.}, 48(6):3206--3227, 2020.

\bibitem{even2020independence}
Chaim Even-Zohar.
\newblock independence: Fast rank tests.
\newblock {\em arXiv preprint arXiv:2010.09712}, 2020.

\bibitem{even2021counting}
Chaim Even-Zohar and Calvin Leng.
\newblock Counting small permutation patterns.
\newblock In {\em Proc. ACM-SIAM SODA}, pages 2288--2302. SIAM, 2021.

\bibitem{Farlie1960}
D.~J.~G. Farlie.
\newblock The performance of some correlation coefficients for a general
  bivariate distribution.
\newblock {\em Biometrika}, 47:307--323, 1960.

\bibitem{Farlie1961}
D.~J.~G. Farlie.
\newblock The asymptotic efficiency of {D}aniels's generalized correlation
  coefficients.
\newblock {\em J. R. Stat. Soc. Ser. B Methodol.}, 23:128--142, 1961.

\bibitem{fruciano2020does}
Carmelo Fruciano, Paolo Colangelo, Riccardo Castiglia, and Paolo Franchini.
\newblock Does divergence from normal patterns of integration increase as
  chromosomal fusions increase in number? a test on a house mouse hybrid zone.
\newblock {\em Curr. Zool.}, 66(5):527--538, 2020.

\bibitem{fuchs2021quantifying}
Sebastian Fuchs.
\newblock Quantifying directed dependence via dimension reduction.
\newblock {\em arXiv preprint arXiv:2112.10147}, 2021.

\bibitem{gamboa2018sensitivity}
Fabrice Gamboa, Thierry Klein, and Agn\'{e}s Lagnoux.
\newblock Sensitivity analysis based on cram\'{e}r--von mises distance.
\newblock {\em SIAM-ASA J. Uncertain. Quantif.}, 6(2):522--548, 2018.

\bibitem{gieser1993}
Peter~William Gieser.
\newblock {\em A new nonparametric test for independence between two sets of
  variates}.
\newblock PhD thesis, University of Florida, 1993.

\bibitem{gini1914ammontare}
Corrado Gini.
\newblock {\em L'ammontare e la composizione della ricchezza delle nazioni},
  volume~62.
\newblock Fratelli Bocca, 1914.

\bibitem{Griessenberger2022}
Florian Griessenberger, Robert~R. Junker, and Wolfgang Trutschnig.
\newblock On a multivariate copula-based dependence measure and its estimation.
\newblock {\em Electron. J. Stat.}, 16(1):2206--2251, 2022.

\bibitem{Hajek1999}
Jaroslav H\'{a}jek, Zbyn\v{e}k \v{S}id\'{a}k, and Pranab~K. Sen.
\newblock {\em Theory of rank tests}.
\newblock Probability and Mathematical Statistics. Academic Press, Inc., San
  Diego, CA, second edition, 1999.

\bibitem{Han2017}
Fang Han, Shizhe Chen, and Han Liu.
\newblock Distribution-free tests of independence in high dimensions.
\newblock {\em Biometrika}, 104(4):813--828, 2017.

\bibitem{han2022azadkia}
Fang Han and Zhihan Huang.
\newblock Azadkia-chatterjee's correlation coefficient adapts to manifold data.
\newblock {\em arXiv preprint arXiv:2209.11156}, 2022.

\bibitem{heller2016computing}
Yair Heller and Ruth Heller.
\newblock Computing the bergsma dassios sign-covariance.
\newblock {\em arXiv preprint arXiv:1605.08732}, 2016.

\bibitem{Hoeffding1948}
Wassily Hoeffding.
\newblock A non-parametric test of independence.
\newblock {\em Ann. Math. Stat.}, 19:546--557, 1948.

\bibitem{holma2022correlation}
Agnes Holma.
\newblock Correlation coefficient based feature screening: With applications to
  microarray data, 2022.

\bibitem{huang2020kernel}
Zhen Huang, Nabarun Deb, and Bodhisattva Sen.
\newblock Kernel partial correlation coefficient--a measure of conditional
  dependence.
\newblock {\em arXiv preprint arXiv:2012.14804}, 2020.

\bibitem{ingster1987minimax}
Yu~I Ingster.
\newblock Minimax testing of nonparametric hypotheses on a distribution density
  in the l\_p metrics.
\newblock {\em Theory Probab. Its Appl.}, 31(2):333--337, 1987.

\bibitem{ingster1993asymptotically}
Yuri~I Ingster.
\newblock Asymptotically minimax hypothesis testing for nonparametric
  alternatives. i, ii, iii.
\newblock {\em Math. Methods Statist.}, 2(2):85--114, 1993.

\bibitem{kendall1938new}
Maurice~G Kendall.
\newblock A new measure of rank correlation.
\newblock {\em Biometrika}, 30(1/2):81--93, 1938.

\bibitem{kim2022minimax}
Ilmun Kim, Sivaraman Balakrishnan, and Larry Wasserman.
\newblock Minimax optimality of permutation tests.
\newblock {\em Ann. Stat.}, 50(1):225--251, 2022.

\bibitem{Konijn1956}
H.~S. Konijn.
\newblock On the power of certain tests for independence in bivariate
  populations.
\newblock {\em Ann. Math. Stat.}, 27:300--323, 1956.

\bibitem{Kossler2007}
Wolfgang K\"{o}ssler and Egmar R\"{o}del.
\newblock The asymptotic efficacies and relative efficiencies of various linear
  rank tests for independence.
\newblock {\em Metrika}, 65(1):3--28, 2007.

\bibitem{Ledwina1986}
Teresa Ledwina.
\newblock On the limiting {P}itman efficiency of some rank tests of
  independence.
\newblock {\em J. Multivar. Anal.}, 20(2):265--271, 1986.

\bibitem{lin2021boosting}
Zhexiao Lin and Fang Han.
\newblock On boosting the power of chatterjee's rank correlation.
\newblock {\em arXiv preprint arXiv:2108.06828}, 2021.

\bibitem{lin2022limit}
Zhexiao Lin and Fang Han.
\newblock Limit theorems of chatterjee's rank correlation.
\newblock {\em arXiv preprint arXiv:2204.08031}, 2022.

\bibitem{lin2023failure}
Zhexiao Lin and Fang Han.
\newblock On the failure of the bootstrap for chatterjee's rank correlation.
\newblock {\em arXiv preprint arXiv:2303.14088}, 2023.

\bibitem{Mcdiarmid1989}
Colin McDiarmid.
\newblock On the method of bounded differences.
\newblock In {\em Surveys in combinatorics, 1989 ({N}orwich, 1989)}, volume 141
  of {\em London Math. Soc. Lecture Note Ser.}, pages 148--188. Cambridge Univ.
  Press, Cambridge, 1989.

\bibitem{Morgenstern1956}
Dietrich Morgenstern.
\newblock Einfache {B}eispiele zweidimensionaler {V}erteilungen.
\newblock {\em Mitteilungsbl. Math. Statist.}, 8:234--235, 1956.

\bibitem{nandy2016large}
Preetam Nandy, Luca Weihs, and Mathias Drton.
\newblock Large-sample theory for the bergsma-dassios sign covariance.
\newblock {\em Electron. J. Stat.}, 10(2):2287--2311, 2016.

\bibitem{nikitin2003pitman}
Ya~Yu Nikitin and NA~Stepanova.
\newblock Pitman efficiency of independence tests based on weighted rank
  statistics.
\newblock {\em J. Math. Sci.}, 118(6):5596--5606, 2003.

\bibitem{Nikitin1995}
Yakov Nikitin.
\newblock {\em Asymptotic efficiency of nonparametric tests}.
\newblock Cambridge University Press, Cambridge, 1995.

\bibitem{pearson1895notes}
K~Pearson.
\newblock Notes on regression and inheritance in the case of two parents.
\newblock {\em Proc. R. Soc. Lond.}, 58:240--242, 1895.

\bibitem{Rosenblatt1975}
M.~Rosenblatt.
\newblock A quadratic measure of deviation of two-dimensional density estimates
  and a test of independence.
\newblock {\em Ann. Stat.}, 3:1--14, 1975.

\bibitem{shi2020power}
Hongjian Shi, Mathias Drton, and Fang Han.
\newblock On the power of chatterjee rank correlation.
\newblock {\em arXiv preprint arXiv:2008.11619}, 2020.

\bibitem{shi2021azadkia}
Hongjian Shi, Mathias Drton, and Fang Han.
\newblock On azadkia-chatterjee's conditional dependence coefficient.
\newblock {\em arXiv preprint arXiv:2108.06827}, 2021.

\bibitem{spearman1906footrule}
Charles Spearman.
\newblock Footrule for measuring correlation.
\newblock {\em Br. J. Psychol.}, 2(1):89, 1906.

\bibitem{spearman1961proof}
Charles Spearman.
\newblock The proof and measurement of association between two things.
\newblock 1961.

\bibitem{Stepanova2008}
Natalia Stepanova and Shu Wang.
\newblock Asymptotic efficiency of the {B}lest-type tests for independence.
\newblock {\em Aust. N. Z. J. Stat.}, 50(3):217--233, 2008.

\bibitem{strothmann2022rearranged}
Christopher Strothmann, Holger Dette, and Karl~Friedrich Siburg.
\newblock Rearranged dependence measures.
\newblock {\em arXiv preprint arXiv:2201.03329}, 2022.

\bibitem{Tsybakov2009}
Alexandre~B Tsybakov.
\newblock Introduction to nonparametric estimation, 2009.
\newblock {\em URL https://doi. org/10.1007/b13794. Revised and extended from
  the}, 9(10), 2004.

\bibitem{vaart}
A.~W. van~der Vaart.
\newblock {\em Asymptotic Statistics}.
\newblock Cambridge Series in Statistical and Probabilistic Mathematics.
  Cambridge University Press, 1998.

\bibitem{Wang2017}
X.~Wang, B.~Jiang, and J.~S. Liu.
\newblock Generalized {R}-squared for detecting dependence.
\newblock {\em Biometrika}, 104(1):129--139, 2017.

\bibitem{Weihs2018}
L.~Weihs, M.~Drton, and N.~Meinshausen.
\newblock Symmetric rank covariances: a generalized framework for nonparametric
  measures of dependence.
\newblock {\em Biometrika}, 105(3):547--562, 2018.

\bibitem{weihs2016efficient}
Luca Weihs, Mathias Drton, and Dennis Leung.
\newblock Efficient computation of the bergsma--dassios sign covariance.
\newblock {\em Comput. Stat.}, 31(1):315--328, 2016.

\bibitem{yanagimoto1970measures}
Takemi Yanagimoto.
\newblock On measures of association and a related problem.
\newblock {\em Ann. Inst. Stat. Math.}, 22(1):57--63, 1970.

\bibitem{Zhang2023}
Qingyang Zhang.
\newblock On the asymptotic null distribution of the symmetrized {C}hatterjee's
  correlation coefficient.
\newblock {\em Statist. Probab. Lett.}, 194:Paper No. 109759, 7, 2023.

\end{thebibliography}

\appendix	
	
\end{document}